\documentclass[12pt,a4paper]{amsart}

\usepackage{amssymb}
\usepackage{mathabx}
\usepackage{bbm}
\usepackage{amsthm}
\usepackage{dsfont}
\usepackage{amsmath}
\usepackage{color,graphics,srcltx}
\usepackage{bm}
\usepackage{easybmat}
\usepackage{multirow,bigdelim}
\usepackage{enumerate}
\usepackage{graphicx}
\usepackage{ulem}

\newtheorem{proposition}{Proposition}

\newtheorem{lemma}[proposition]{Lemma}
\newtheorem{corollary}[proposition]{Corollary}
\newtheorem{theorem}[proposition]{Theorem}

\newcommand{\GG}{{\mathcal G}}

\newcommand{\RR}{\mathbb{R}}

\def\mC{\mathcal{C}}
\def\KK{\mathcal{K}}
\def\cQ{\mathcal{Q}}

\definecolor{damjana}{rgb}{.8,.2,.2}

\def\II{\mathds I}

\def\CC{\mathcal C}

\def\FF{\mathcal F}

\begin{document}

\title[Spearman's footrule and Gini's gamma]{Spearman's footrule and Gini's gamma: Local bounds for bivariate copulas and the exact region with respect to Blomqvist's beta}

\author[D. Kokol Bukov\v sek]{Damjana Kokol Bukov\v{s}ek}
\address{School of Economics and Business, University of Ljubljana, and Institute of Mathematics, Physics and Mechanics, Ljubljana, Slovenia}
\email{damjana.kokol.bukovsek@ef.uni-lj.si}

\author[T. Ko\v sir]{Toma\v{z} Ko\v{s}ir}
\address{Faculty of Mathematics and Physics, University of Ljubljana, and Institute of Mathematics, Physics and Mechanics, Ljubljana, Slovenia}
\email{tomaz.kosir@fmf.uni-lj.si}

\author[B. Moj\v skerc]{Bla\v{z} Moj\v{s}kerc}
\address{School of Economics and Business, University of Ljubljana, and Institute of Mathematics, Physics and Mechanics, Ljubljana, Slovenia}
\email{blaz.mojskerc@ef.uni-lj.si}

\author[M. Omladi\v c]{Matja\v{z} Omladi\v{c}}
\address{Institute of Mathematics, Physics and Mechanics, Ljubljana, Slovenia}
\email{matjaz@omladic.net}

\begin{abstract}
Copulas are becoming an essential tool in analyzing data thus encouraging interest in related questions. In the early stage of exploratory data analysis, say, it is helpful to know local copula bounds with a fixed value of a given measure of association. These bounds have been computed for Spearman's rho, Kendall's tau, and Blomqvist's beta. The importance of another two measures of association, Spearman's footrule and Gini's gamma, has been reconfirmed recently. It is the main purpose of this paper to fill in the gap and present the mentioned local bounds for these two measures as well. It turns out that this is a quite non-trivial endeavor as the bounds are quasi-copulas that are not copulas for certain values of the two measures. We also give relations between these two measures of association and Blomqvist's beta.
\end{abstract}

\thanks{All four authors acknowledge financial support from the Slovenian Research Agency (research core funding No. P1-0222).\\ \date}
\keywords{Copula; dependence concepts; supremum and infimum of a set of copulas; measures of concordance; quasi-copula; local bounds}
\subjclass[2010]{Primary:     60E05; Secondary: 60E15, 62H20}

\maketitle

\section{ Introduction }

In the study of complex systems it is of interest to synthesize the information coming from different sources into a single output, which is either numerical or represented by a suitable function, graph, etc. For instance, the copula representation has proved to be a suitable tool to describe uncertain inputs in a probabilistic framework (see, e.g., \cite{DuSe,Nels}) as well as in an imprecise setting (see, e.g., \cite{DuSp,OmSt2}). When seeking a copula that fits given data best, a practitioner would perform what is called an \textit{Exploratory Data Analysis}. A possible way to do that would be to go through the following steps: (1) starting with a rank plot, (2) measuring association, (3) testing exchangeability, (4) testing for independence; and of course, performing any additional test necessary depending on the situation. The motivation is given in the overview paper by Genest and Favre \cite{GeFa}, cf. also\ \cite{GeNeRe,GeRe}.

Step (2) may be seen as one of the reasons why it has become so popular recently to study local Fr\'echet-Hoeffding bounds of families of copulas with mutually related measures of concordance. Some history of explorations connected to the kind of bounds can be found in \cite[Sections 3\&4]{KoBuKoMoOm2}. In \cite{OmSt3} theoretical aspects of local bounds are given, called there \textit{constrained bounds}, and studied in more details. In particular, the local bounds of sets of copulas having a fixed value of Spearman's rho, Kendall's tau, respectively Blomqvist's beta, have been worked out \cite{NQRU,NeUbFl}. On the other hand, the analogous question is still open for the Spearman's footrule and Gini's gamma, two measures of association whose importance has recently been brought up in \cite{GeNeGh}.

We denote by $\CC$ the set of all bivariate copulas and by $\II$ the interval $[0,1]\subseteq\RR$. Some transformations are naturally defined on $\CC$: 
The transpose of copula $C$ will be denoted by $C^t$, i.e., $C^t(u,v)= C(v,u)$. We denote by $C^{\sigma_1}$ and $C^{\sigma_2}$ the two reflections of a copula $C$ defined by $C^{\sigma_1}(u,v) =v-C(1-u,v)$ and $C^{\sigma_2}(u,v)=u-C(u,1-v)$ (see \cite[{\S}1.7.3]{DuSe}), and by $\widehat{C}=\left(C^{\sigma_1}\right)^{\sigma_2}$ the survival copula of $C$.

Several orders can be introduced on  $\CC$. Copula $C$ is preceding $D$ in the \textit{concordance order} if $C(u,v)\leqslant D(u,v)$ and $\widehat{C}(u,v)\leqslant \widehat{D}(u,v)$ for all $(u,v)\in\II^2$ {\cite[Definition 2.4]{Joe}. Copula $C$ is preceding $D$ in the \textit{pointwise order} if only $C(u,v)\leqslant D(u,v)$ for all $(u,v)\in\II^2$. (See \cite[Definition 2.8.1]{Nels} and \cite[\S{2.11}]{Joe} for further details.) The concordance order and the pointwise order coincide on the set of two-dimensional copulas. (See \cite[\S{2.2.1}]{Joe97} for a proof of this statement.) Hence, we will simply refer to them as the order, and write $C\leqslant D$ for $C,D\in\CC$ if $C(u,v)\leqslant D(u,v)$ for all $(u,v)\in\II^2$.
It is well known that $\CC$ is 
{a partially ordered set} with respect to the order{,  but not a lattice \cite[Theorem 2.1]{NeUbF2},} and that $W(u,v)=\max\{0,u+v-1\}$ and $M(u,v)=\min\{u,v\}$ are the lower and the upper bound of all copulas, respectively. Copulas $W$ and $M$ are called the \textit{Fr\'echet-Hoeffding lower and upper bound}, respectively. It was proved in \cite{NeUbF2} that the set of all bivariate quasi-copulas is a complete lattice that is order isomorphic to the Dedekind--MacNeille completion of the set of all bivariate copulas.

The paper is organized as follows. Sections 2 and 3 present preliminaries on measures of concordance and on local bounds. Local bounds of the set of all copulas corresponding to a fixed value of the Spearman's footrule $\phi\in[-\frac12,1]$ are given in Section 4. So, Theorem 5 is one of the two main results of the paper. The second main result is Theorem 7 presented in Section 5; the local bounds of the set of copulas that have the value of Gini's gamma $\gamma\in[-1, 1]$ fixed are determined there. Section 6 is devoted to comparison of these local bounds and Section 7 to the study of relations between the two measures of concordance and a third one, namely Blomqvist's beta.

\section{Preliminaries on measures of concordance}\label{sec:prelim}

A mapping $\kappa:\CC\to [-1,1]$ is called a \textit{measure of concordance} if it satisfies the following properties (see \cite[Definition 2.4.7]{DuSe}):
\begin{enumerate}[(C1)]
\item $\kappa(C)=\kappa(C^t)$ for every $C\in\CC$.
\item $\kappa(C)\leqslant\kappa(D)$ when $C\leqslant D$.  \label{monotone}
\item $\kappa(M)=1$.
\item $\kappa(C^{\sigma_1})=\kappa(C^{\sigma_2})=-\kappa(C)$.
\item If a sequence of copulas $C_n$, $n\in\mathbb{N}$, converges uniformly to $C\in\CC$, then $\lim_{n\to\infty}\kappa(C_n)=\kappa(C)$.
\end{enumerate}

We will refer to property (C\ref{monotone}) above simply by saying that a measure of concordance under consideration is \textit{monotone}.

{Certain properties that are sometimes stated in definitions of a measure of concordance follow from the properties listed above. Namely, a measure of concordance also satisfies the following (see \cite[\S{3}]{KoBuKoMoOm2} for further details):
\begin{enumerate}
\item[(C{6})] $\kappa(\Pi)=0$, where $\Pi$ is the independence copula $\Pi(u,v)=uv$. \label{kappa_Pi}
\item[(C{7})] $\kappa(W)=-1$.\label{kappaW}
\item[(C{8})] $\kappa(C)=\kappa(\widehat{C})$ for every $C\in\CC$.
\end{enumerate}

{Because of their significance in statistical analysis, measures of concordance and their relatives measures of association and measures of dependence are a classical topic of research. It was Scarsini \cite{Scar} who introduced formal axioms of a measure of concordance. Some of more recent references on bivariate measures of concordance include \cite{EdTa,FrNe,FuSch,Lieb,NQRU,NQRU2}. Their multivariate generalization was studied e.g. in \cite{BeDoUF2,DuFu,Tayl,UbFl}. For bivariate copulas the main measures of concordance are naturally studied through symmetries that are a consequence of properties of the \textit{concordance function} $\cQ$ (see for instance \cite{BeDoUF,EdMiTa,EdMiTa2}). } The concordance function is defined for a pair of random vectors $(X_1,Y_1)$ and $(X_2,Y_2)$. If the corresponding copulas are $C_1$ and $C_2$ and if the distribution functions of $X_1$, $X_2$, $Y_1$ and $Y_2$ are continuous, then we have
\begin{equation}\label{concordance}
  \cQ=\cQ(C_1,C_2)= 4 \int_{\II^2} C_2(u,v)dC_1(u,v) -1.
\end{equation}
(See \cite[Theorem 5.1.1]{Nels}.)
The concordance function was introduced by Kruskal \cite{Krus} and it has a number of useful properties \cite[Corollary 5.1.2]{Nels} and \cite[\S{3}]{KoBuKoMoOm2}.}} In the sequel, we use the following two:
\begin{enumerate}[(Q1)]
\item It remains unchanged when both copulas are replaced by their survival copulas: \\ {$\cQ\left(C_1,C_2\right)=\cQ\left(\widehat{C}_1,\widehat{C}_2\right)$.}
\item When both copulas are replaced by their reflected copulas the sign changes: \\ $\cQ\left(C_1^{\sigma_j},C_2^{\sigma_j}\right)=-\cQ\left(C_1,C_2\right)$ for $j=1,2$.
\end{enumerate}

The {four} most commonly used measures of concordance of a copula $C$ are Kendall's tau, Spearman's rho,
Gini's gamma and Blomqvist's beta. {We refer to \cite{Lieb} for an extended definition of a measure of concordance. If we replace Property (C4) by Property (C6) in the definition of a measure of concordance, we get what Liebscher in \cite{Lieb} calls a \textit{weak measure of concordance}. Spearman's footrule is an example of such a weak measure of concordance.} The range of a measure of concordance is the interval $[-1,1]$, while the range of Spearman's footrule is equal to $\left[-\frac12,1\right]$ (see \cite[\S4]{UbFl}). The sets of all copulas where the bounds $-\frac12$ and $1$ for Spearman's footrule are attained are given in \cite{FuMcC}, where also the generalization to multidimensional setting $d\geqslant 3$ is given.

To simplify the discussion from now on, we include Spearman's footrule when we talk about measures of concordance in general thus omitting the word 'weak'. In formal statements of our results we include the word 'weak' for precision.

Statistical significance of all five measures of concordance is already well established. See e.g. \cite{DiaGra,CoNi,GeNeGh,Nels1998,UbFl,SSQ} for Gini's gamma, Blomqvist's beta and Spearman's footrule. Nelsen \cite{Nels1998} discusses the $l_1$ nature of Gini's gamma and Spearman's footrule as compared to $l_2$ nature of Spearman's rho. Spearman's footrule depends only on $l_1$ distance of copula $C$ to the upper bound $M$, while Gini's gamma depends on $l_1$ distances to both bounds $W$ and $M$ \cite{Nels1998}.
All three measures of concordance that we study are of degree one \cite{EdTa}.
\\

In this paper, we focus on Spearman's footrule, Gini's gamma, and, in the last section, on their relation with Blomqvist's beta. The first two of them may be defined in terms of the concordance function $\cQ$. The \textit{Spearman's footrule} is defined by
\begin{equation}\label{phi}
\phi(C)=
\frac12\left(3\cQ(C,M)-1\right)=6\int_0^1 C(t,t) dt - 2
\end{equation}
and \textit{Gini's gamma} by
\begin{equation}\label{gamma}
\gamma(C)=\cQ(C,M)+\cQ(C,W) = 4\int_0^1 \left(C(t,t)+C(t,1-t)\right) dt - 2.
\end{equation}
On the other hand, \textit{Blomqvist's beta} is defined by
\begin{equation}\label{beta}
\beta(C)=
4\,C\left(\frac12,\frac12\right)-1.
\end{equation}
(See \cite[{\S}2.4]{DuSe} and \cite[Ch. 5]{Nels}.) Note that Gini's gamma and Blomqvist's beta are measures of concordance, so (C1)-(C8) hold for them. On the other hand, only Properties (C4) and (C7) do not hold for Spearman's footrule. Property (C8) holds for it since
\begin{equation}\label{C8 for phi}
\phi(C)=\frac12\left(3\cQ(C,M)-1\right)=\frac12\left(3\cQ(\widehat{C},\widehat{M})-1\right)=
\frac12\left(3\cQ(\widehat{C},{M})-1\right)=\phi(\widehat{C}).
\end{equation}
Here we used (Q1) and the fact that $\widehat{M}=M$.

Gini's gamma and Spearman's footrule are related. In fact, Gini's gamma is a symmetrized version of Spearman's footrule \cite{Nels2004}, since
\begin{equation}\label{symm_gama}
\gamma(C)=\frac23\left(\phi(C)-\phi(C^{\sigma_i})\right)
\end{equation}
for either $i=1$ or $i=2$. To show \eqref{symm_gama} we use (Q2) and the fact that $W^{\sigma_i}=M$:
$$\gamma(C)=\cQ(C,M)+\cQ(C,W)=\cQ(C,M)-\cQ(C^{\sigma_i},M)=\frac23\left(\phi(C)-\phi(C^{\sigma_i})\right).$$

\section{Preliminaries on local bounds}\label{sec:local}

Besides the Fr\'{e}chet-Hoeffding upper and lower bound, which are global bounds for the ordered set of copulas, one often studies local bounds of certain subsets. Perhaps among the first known examples of the kind is given in Theorem 3.2.3 of Nelsen's book \cite{Nels} (cf.\ also \cite[Theorem 1]{NQRU}), where the bounds of the set of copulas $C\in\CC$ with $C(a,b)=d$ for fixed $a,b\in\II$ and $d\in[W(a,b),M(a,b)]$ are given. In general, if $\mathcal{C}_0$ is a set of copulas, we let
\begin{equation}\label{eq:inf:sup}
  \underline{C}=\inf\mathcal{C}_0\quad\text{and}\quad\overline{C} =\sup\mathcal{C}_0,
\end{equation}
where the infimum and the supremum are taken point-wise. In \cite{NQRU} the authors study the bounds for the set of copulas whose Kendall's tau equals a given number $t\in [-1,1]$ and for the set of copulas whose Spearman's rho equals a given number $t\in [-1,1]$. In both cases the bounds are copulas that do not belong to the set. Similar bounds for the set of copulas having a fixed value of Blomqvist's beta were found in \cite{NeUbFl}. In \cite{BBMNU14} the authors present the local bounds for the set of copulas having a fixed value of the degree of non-exchangeability. In all these cases the bounds are again copulas. We know that this is not true in general since the bound of a set of copulas may be a proper quasi-copula. This is true in several cases considered in our paper.\\

Suppose that $\kappa:\mathcal{C}\to[-1,1]$ is a given measure of concordance and that $k\in [-1,1]$ is a fixed value in the range of $\kappa$. Then we write
\begin{equation}\label{eq:kappa}
  \KK_{k}:=\{C\in\mC\,|\,\kappa(C)=k\}.
\end{equation}
We denote by $\underline{K}_{k}=\inf\mathcal{K}_{k}$ {and} $\overline{K}_{k} =\sup\mathcal{K}_{k}$ the lower and the upper bound of \eqref{eq:kappa}, respectively. The symmetries that the concordance function and measures of concordance possess imply symmetries on the bounds \cite{BeDoUF,EdMiTa,EdMiTa2}.

\begin{lemma}\label{lem:symm}
\textit{(a)} Suppose that $\kappa$ is at least a weak measure of concordance. Then the lower and the upper bounds  $\underline{K}_{k}$ {and} $\overline{K}_{k}$ are symmetric and radially symmetric: $$\underline{K}_{k}(a,b)=\underline{K}_{k}(b,a)\ \text{and}\ \underline{K}_{k}(a,b)=\widehat{\underline{K}}_{k}(a,b)$$
and
$$\overline{K}_{k}(a,b)=\overline{K}_{k}(b,a)\ \text{and}\ \overline{K}_{k}(a,b)=\widehat{\overline{K}}_{k}(a,b).$$

\textit{(b)} If $\kappa$ is a (proper) measure of concordance then also
$$\underline{K}_{k}^{\sigma_i}(a,b)=\overline{K}_{-k}(a,b)\ \text{and}\ \overline{K}_{k}^{\sigma_i}(a,b)={\underline{K}}_{\, -k}(a,b)$$
for $i=1,2.$
\end{lemma}

\begin{proof}
Suppose that $\kappa$ satisfies Properties (C1) and (C8). Then for a copula $C\in\mC$ we have $C\in\KK_{k}$ if and only if $C^t\in\KK_k$, and $C\in\KK_{k}$ if and only if  $\widehat{C}\in\KK_{k}$. So we conclude that
\begin{equation*}
\underline{K}_{k}(a,b)=
\inf_{C\in \KK_{k}}C(b,a)\\
=\underline{K}_{k}(b,a)
\end{equation*}
and
\begin{equation*}
\underline{K}_{k}(a,b)=\inf_{C\in \KK_{k}}C(a,b)
=a+b-1+\inf_{C\in \KK_{k}}C(1-a,1-b)
=\widehat{\underline{K}}_{\, k}(a,b).
\end{equation*}
The equalities for the upper bound are proved analogously.

Suppose now that $\kappa$ satisfies property (C4). Then for each $i$ we have $C\in \underline{K}_{k}$ if and only if $C^{\sigma_i}\in \underline{K}_{\, -k}$. This implies that
\begin{equation*}
\underline{K}_{k}(a,b)=\inf_{C\in \KK_{k}}C(a,b)
=\inf_{C\in \KK_{-k}}C^{\sigma_1}(a,b)
=b-\sup_{C\in \KK_{-k}}C(1-a,b)
=\overline{K}_{-k}^{\sigma_1}(a,b).
\end{equation*}
The other equalities follow analogously.
\end{proof}

The notion \textit{maximal asymmetry function} was introduced in \cite[\S 2]{KoBuKoMoOm1} following the ideas of \cite{KlMe}, its value at a fixed point $(u, v)\in\II^2$ was computed as
\[
    d^*_\FF (u,v) = \sup_{C\in\FF} \{|C(u,v)-C(v,u)|\},
\]
where $\FF\subseteq\CC$ is an arbitrary family of copulas. If $\FF=\CC$, this supremum is attained since $\CC$ is a compact set by \cite[Theorem 1.7.7]{DuSe}. Klement and Mesiar \cite{KlMe} and Nelsen \cite{N} showed that
\begin{equation}\label{eq:kle mes}
  d_\CC^*(u,v)=\min\{u,v,1-u,1-v,|v-u|\}.
\end{equation}

In \cite{KoBuKoMoOm2}, extremal copulas where the asymmetry bounds are attained were introduced. Choose $(a,b)\in\II^2$ and a $c\in\II$ such that $0\leqslant c\leqslant d_\CC^*(a,b)$. Define $\CC_0$ to be the set of all $C$ such that
\begin{equation}\label{eq:asym_point}
  C(a,b)-C(b,a) = c.
\end{equation}
Note that this set is nonempty since the set $\CC$ is convex by
\cite[Theorem 1.4.5]{DuSe}. The local bounds $\underline{C}$ and $\overline{C}$ of this set were computed in \cite[Theorem 1]{KoBuKoMoOm2}

\begin{equation*}
    \underline{C}^{(a,b)}_{c}(u,v)= \max\{W(u,v),\min\{d_1,u-a+d_1,v-b+d_1,u+v-a-b+d_1\}\},
\end{equation*}
and
\begin{equation*}
    \overline{C}^{(a,b)}_{c}(u,v)=\min\{M(u,v),\max\{d_2,u-b+d_2,v-a+d_2,u+v-a-b+d_2\}\},
\end{equation*}
where
\begin{equation*}
d_1=W(a,b)+c,
{\text{ and }
d_2=M(a,b)-c}
\end{equation*}
for $0 \leqslant c \leqslant d^*_C(a,b)$. Observe that $c$ is small enough so that everywhere close to the boundary of the square $\II^2$ copula $W$ prevails in the definition of $\underline{C}^{(a,b)}_{c}$, and that copula $M$ prevails close to the boundary of the square $\II^2$ in the definition of $\overline{C}^{(a,b)}_{c}$.

Copulas $\underline{C}^{(a,b)}_{c}$ and $\overline{C}^{(a,b)}_{c}$ can be considered for any $c$ such that $0\leqslant c\leqslant\min\{a,b,1-a,1-b\}$, not necessarily $c \leqslant |b-a|$. It turns out that they are exactly the minimal and  the maximal copulas with the property $C(a,b) = d_1$ and $C(a, b) = d_2$, respectively \cite[Theorem 3.2.3]{Nels}.

Note that $\underline{C}^{(a,b)}_{c}$ and $\overline{C}^{(a,b)}_{c}$ are shuffles of $M$, compare \cite[{\S}3.2.3]{Nels} and  \cite[\S3.6]{DuSe} (cf.\ also \cite{N}), so they are automatically copulas. More precisely, as shuffles of $M$ they are rewritten as
\begin{equation}\label{eq:C shuffle}
  \begin{split}
     \underline{C}^{(a,b)}_{c} & =M(4,\{[0,a-d_1],[a-d_1,a],[a, 1-b+d_1], [1-b+d_1,1]\},(4,2,3,1),-1)\\
     \overline{C}^{(a,b)}_{c}  & =M(4,\{[0,d_2],[d_2,b],[b,a+b-d_2], [a+b-d_2,1]\},(1,3,2,4),1),
  \end{split}
\end{equation}
for $0 \leqslant c \leqslant min\{a, b, 1-a, 1-b\}$, where the last parameter on the righthand-side in the above expressions is a function $f:\{1,2,\ldots,n\}\to\{-1,1\}$ which is in the first line of Equation \eqref{eq:C shuffle} identically equal to $-1$ and in the second one identically equal to 1.


To compute the values of various measures of concordance of these copulas we need the values of the concordance function $\cQ$ introduced in Section \ref{sec:prelim} for various copulas such as $W$, $M$, and $\underline{C}^{(a,b)}_{c}$, respectively $\overline{C}^{(a,b)}_{c}$.

The following proposition is proved in \cite[Propositions 4\&5]{KoBuKoMoOm2}. It was also pointed out there that these results are symmetric with respect to the main diagonal and to the counter-diagonal \cite[Proposition 6]{KoBuKoMoOm2}.

\begin{proposition} \label{prop1}
Let $(a,b)\in\II^2$ and $0\leqslant c\leqslant\min\{a,b,1-a,1-b\}$. For copulas $\underline{C}^{(a,b)}_{c}$ and $\overline{C}^{(a,b)}_{c}$ it holds:
\begin{enumerate}[(a)]
\item $\cQ(M, \underline{C}^{(a,b)}_{c}) =$ \\  $=\left\{ \begin{array}{ll}
        0;                        & \text{if } b \geqslant d_1 + \frac12, \vspace{1mm}\\
        (2d_1+1-2b)^2;            & \text{if } \frac12(1+d_1) \leqslant b \leqslant d_1 + \frac12, a \leqslant b-d_1, \vspace{1mm}\\
        (1+d_1-a-b)(1+3d_1+a-3b); & \text{if } \frac12(1+d_1) \leqslant b \leqslant d_1 + \frac12, a \geqslant b-d_1, \vspace{1mm}\\
        d_1(2+3d_1-4b);           & \text{if } b \leqslant \frac12(1+d_1), a \leqslant b-d_1, \vspace{1mm}\\
        2d_1(1+d_1-a-b)-(a-b)^2;  & \text{if } d_1 \geqslant 2a-1, d_1 \geqslant 2b-1, d_1 \geqslant a-b,d_1 \geqslant b-a, \vspace{1mm}\\
        d_1(2+3d_1-4a);           & \text{if } a \leqslant \frac12(1+d_1), b \leqslant a-d_1, \vspace{1mm}\\
        (1+d_1-a-b)(1+3d_1-3a+b); & \text{if } \frac12(1+d_1) \leqslant a \leqslant d_1 + \frac12, b \geqslant a-d_1, \vspace{1mm}\\
        (2d_1+1-2a)^2;            & \text{if } \frac12(1+d_1) \leqslant a \leqslant d_1 + \frac12, b \leqslant a-d_1, \vspace{1mm}\\
                0;                        & \text{if } a \geqslant d_1 + \frac12,
                \end{array} \right.$
\item $\cQ(M, \overline{C}^{(a,b)}_{c}) = 1 - 4(a-d_2)(b-d_2),$
\item $\cQ(W, \underline{C}^{(a,b)}_{c}) = 4d_1(1-a-b+d_1)-1,$
\end{enumerate}
\end{proposition}

\section{Local bounds for Spearman's footrule}\label{sec:footrule}

In this section we compute local bounds of the set of all copulas corresponding to a fixed value of Spearman's footrule $\phi\in[-\frac12,1]$~:
\begin{equation}\label{eq:phi}
  \FF_{\phi}:=\{C\in\CC\,|\,\phi(C)=\phi\}.
\end{equation}
We choose a point $(a,b)$ in the interior of the unit square $\II^2$. We need to find the minimal and maximal value of $C(a,b)$ for all copulas in $\FF_{\phi}$.

Suppose now that $C\in \FF_{\phi}$ and $C(a,b)=d$. By \cite[Theorem 3.2.3]{Nels} it follows that
$$\underline{C}^{(a,b)}_{c_1} \leqslant C \leqslant \overline{C}^{(b,a)}_{c_2},$$
where $d=\underline{C}^{(a,b)}_{c_1}(a,b)=W(a,b)+c_1$ and $d=\overline{C}^{(b,a)}_{c_2}(a,b)=M(a,b)-c_2$, so that
\begin{equation}\label{c_i}
c_1=d-W(a,b)\ \text{and}\ c_2=M(a,b)-d.
\end{equation}
Here we prefer to view $c=c(d)$ as a function of $d=C(a,b)$. Note that $c$ takes values on $[c_1,c_2]$.
Since concordance functions are monotone it follows that
$$\underline{f}_{a,b}(d)\leqslant \phi(C)\leqslant \overline{f}_{a,b}(d),$$
where we write
\begin{equation}
\label{f-lower}\underline{f}_{a,b}(d)=\phi\left(\underline{C}^{(a,b)}_{c_1}\right)=\phi\left(\underline{C}^{(a,b)}_{d-W(a,b)}\right)
\end{equation}
and
\begin{equation}\label{f-upper}
\overline{f}_{a,b}(d)=\phi\left(\overline{C}^{(b,a)}_{c_2}\right)=\phi\left(\overline{C}^{(b,a)}_{M(a,b)-d}\right).
\end{equation}


\begin{theorem}\label{thm_phi}\label{thm_phi_low}
The pointwise infimum $\underline{F}_{\phi}$ 
of $\FF_{\phi}$ for $\phi\in[-\frac12,1]$ is given by
\begin{equation}\label{F_lower}
\underline{F}_{\phi}(a,b)=\left\{ \begin{array}{ll}
        \frac12\left(a+b-\sqrt{\frac{2}{3}(1-\phi)+(b-a)^2}\right);      & \text{if } b\notin\{0,1\}, \text{ and }\frac{1-\phi}{6b} \leqslant a \leqslant 1- \frac{1-\phi}{6(1-b)}, \vspace{1mm}\\
        W(a,b);            & \text{otherwise, }
        \end{array} \right.
\end{equation}
for any $(a,b)\in\II^2$.
\end{theorem}

\begin{proof}
Proposition \ref{prop1}\textit{(b)} implies that function $\overline{f}_{a,b}$ of \eqref{f-upper} is given by
\begin{equation}\label{f_ab_dep_on_d_2}
\overline{f}_{a,b}(d)=\frac32\cQ\left(M,\overline{C}^{(b,a)}_{c_2}\right)-\frac12=1-6(a-d)(b-d).
\end{equation}
Since
\begin{equation}\label{d_two}
d=\overline{C}^{(b,a)}_{c_2}(a,b)
\end{equation}
it follows that $W(a,b)\leqslant d\leqslant M(a,b)$. For such values of $d$, the expression on the right-hand side of \eqref{f_ab_dep_on_d_2} is increasing in $d$ since its maximum is achieved at $\frac12(a+b)$ that is greater or equal to $M(a,b)$. Thus, the  minimal possible value of $\overline{f}_{a,b}(d)$ is achieved when $d=W(a,b)$. Then, we have
$$\overline{f}_{a,b}\left(W(a,b)\right)=\left\{ \begin{array}{ll}
        1-6ab;           & \text{if } a+b \leqslant 1, \vspace{1mm}\\
        1-6(1-a)(1-b);   & \text{if } a+b \geqslant  1.
        \end{array} \right.$$
For function $\overline{f}_{a,b}:[W(a,b), M(a,b)]\to [\overline{f}_{a,b}\left(W(a,b)\right),1]$ we need to find its inverse.
If for a given value $\phi \in \left[ -\frac12, 1 \right]$ it holds that $\phi \leqslant \overline{f}_{a,b}\left(W(a,b)\right)$, then we take $d=W(a,b)$.
Otherwise, we take the inverse of the expression \eqref{f_ab_dep_on_d_2}, i.e.
\begin{equation}\label{bound for d_2}
d = \frac12\left(a+b-\sqrt{\frac{2}{3}(1-\phi)+(b-a)^2}\right).
\end{equation}
The inequality $\phi \geqslant \overline{f}_{a,b}\left(W(a,b)\right)$ gives us the condition
$$b\notin\{0,1\}, \text{ and }\frac{1-\phi}{6b} \leqslant a \leqslant 1- \frac{1-\phi}{6(1-b)}.$$
We conclude that the required lower bound is given by \eqref{F_lower}.
\end{proof}

\begin{corollary}\label{cor_phi}
Suppose that $\underline{F}_{\phi}$ is the infimum given in Theorem \ref{thm_phi_low}. Then:
\begin{enumerate}[(i)]
 \item $\underline{F}_{\phi}$ is  a copula for every $\phi \in \left[ -\frac12, 1\right]$.
 \item We have $\underline{F}\,_{-\frac12}=W$ and $\underline{F}_{1}=M$, while for every $\phi \in \left( -\frac12, 1 \right)$ the copula $\underline{F}_{\phi}$ is different from Fr\'echet-Hoeffding lower and upper bounds $W$ and $M$. It has a singular component distributed on graphs of hyperbolas
     $$ab=\frac{1-\phi}{6}\ \text{ and }\ (1-a)(1-b)=\frac{1-\phi}{6}$$
     for $a\in [\frac12-\ell(\phi),\frac12+\ell(\phi)]$ and on the anti-diagonal $a+b=1$ for other $a\in\II$. Here, we write $\ell(\phi)=\frac16{\sqrt{3(1+2\phi)}}$.  The absolutely continuous part of  $\underline{F}_{\phi}$ is distributed inside the region enclosed by both hyperbolas. (See Figure \ref{fig Fspodaj}.)
 \item $\underline{F}_{\phi}$ is increasing in $\phi$ (in the concordance order).
 \item $\underline{F}_{\phi}$ is symmetric and radially symmetric: $\underline{F}_{\phi}(a,b)=\underline{F}_{\phi}(b,a)$ and $\underline{F}_{\phi}(a,b)=\widehat{\underline{F}}_{\phi}(a,b)$.
 \item For $\phi\in\left( -\frac12, 1 \right)$, copula $\underline{F}_{\phi}$ is not a member of $\FF_{\phi}$, but it holds that $\phi\left(\underline{F}_{\phi}\right)<\phi$. (See Figure \ref{phi(phi)}.)
 \end{enumerate}
 \end{corollary}


\begin{figure}[h]
            \includegraphics[width=5cm]{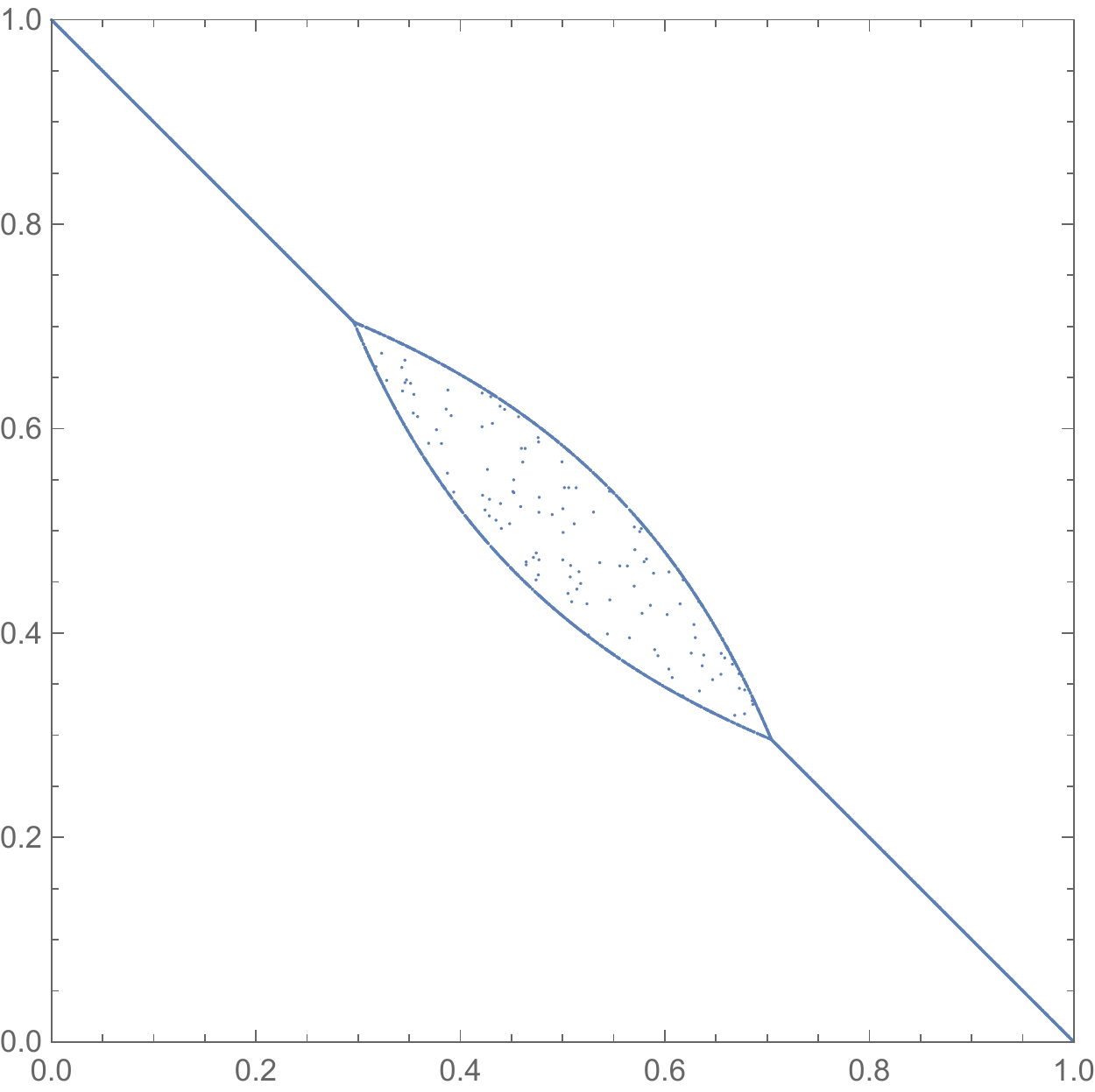} \hfil \includegraphics[width=5cm]{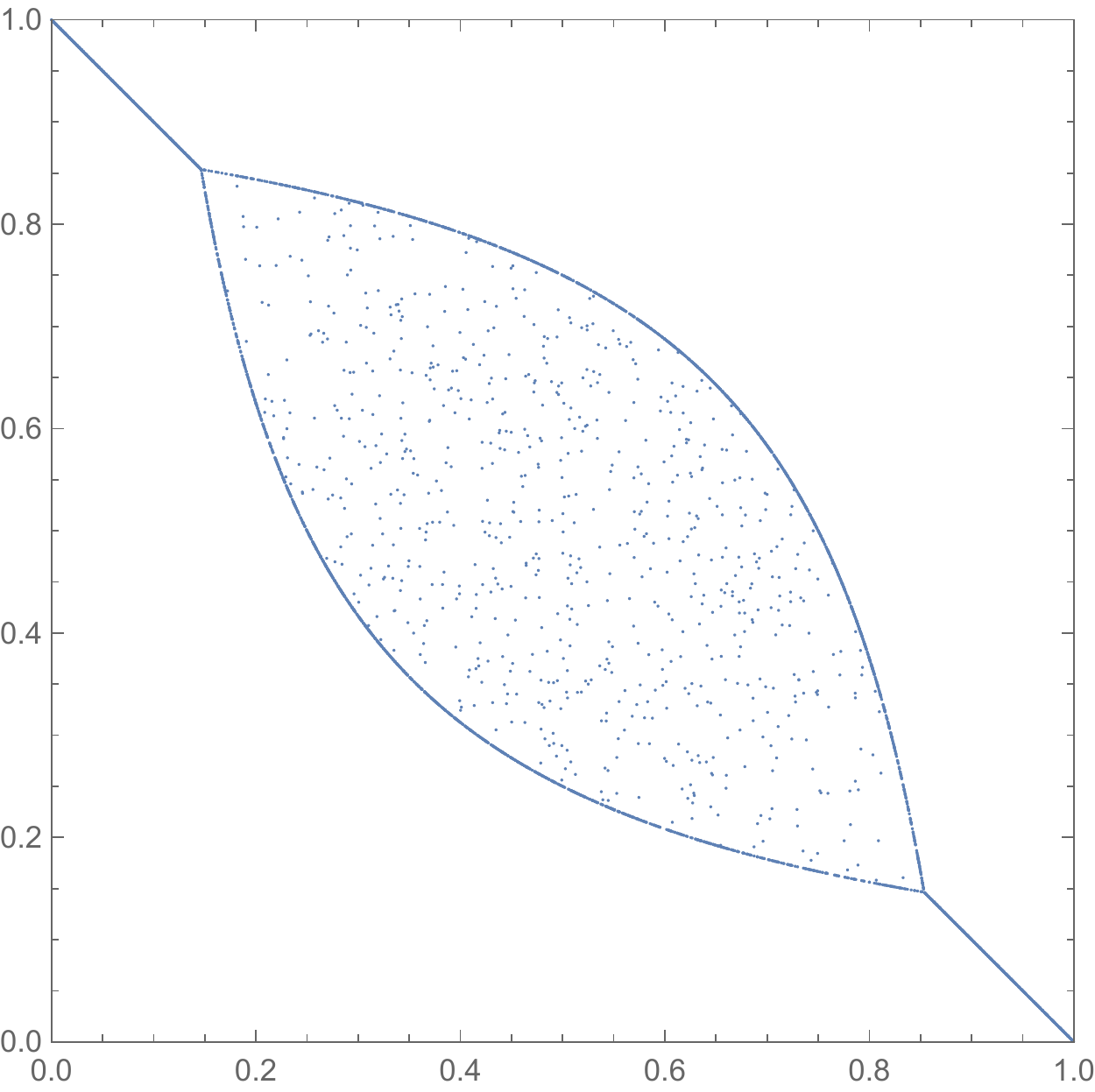} \\
						\includegraphics[width=5cm]{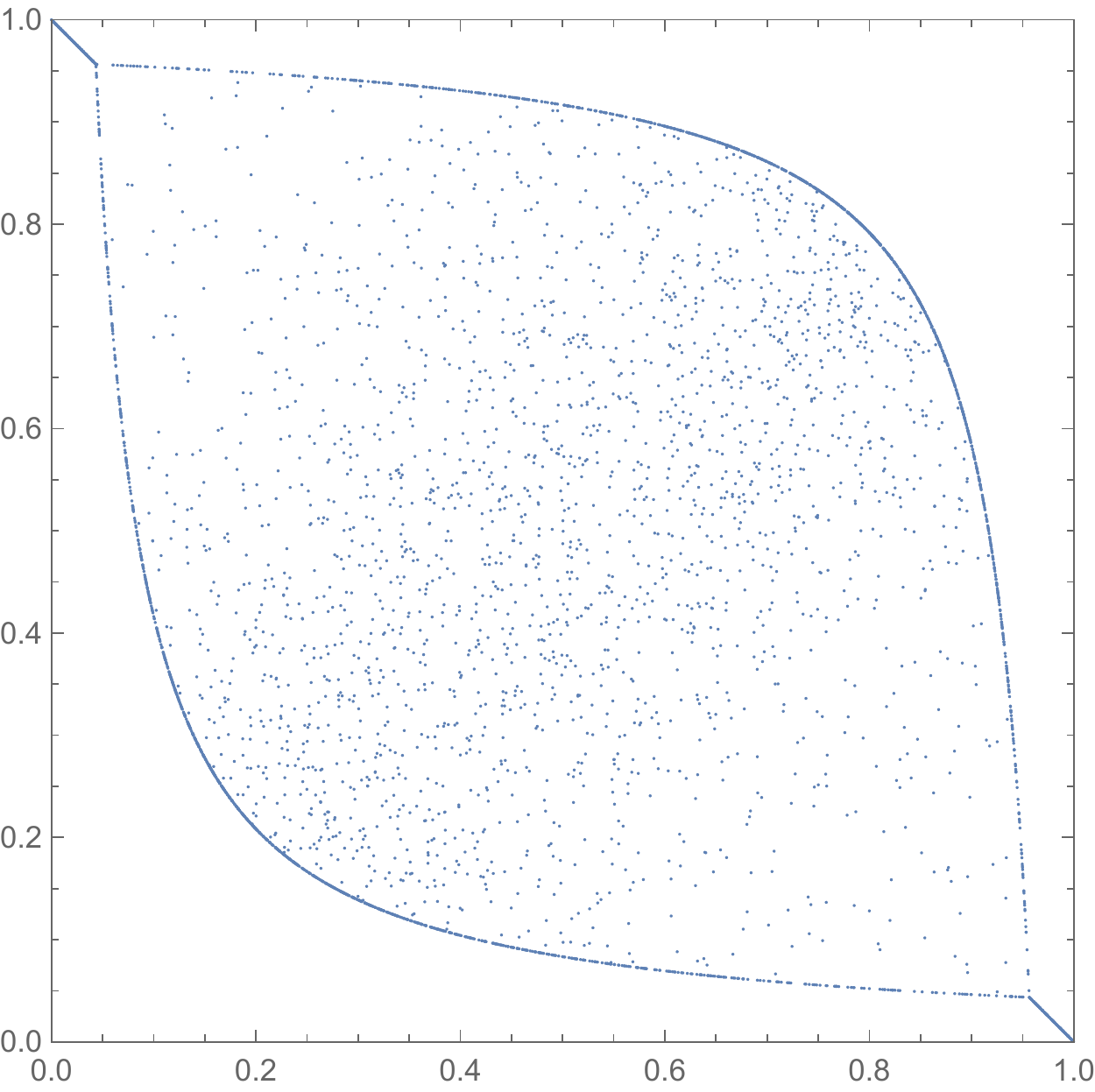} \hfil \includegraphics[width=5cm]{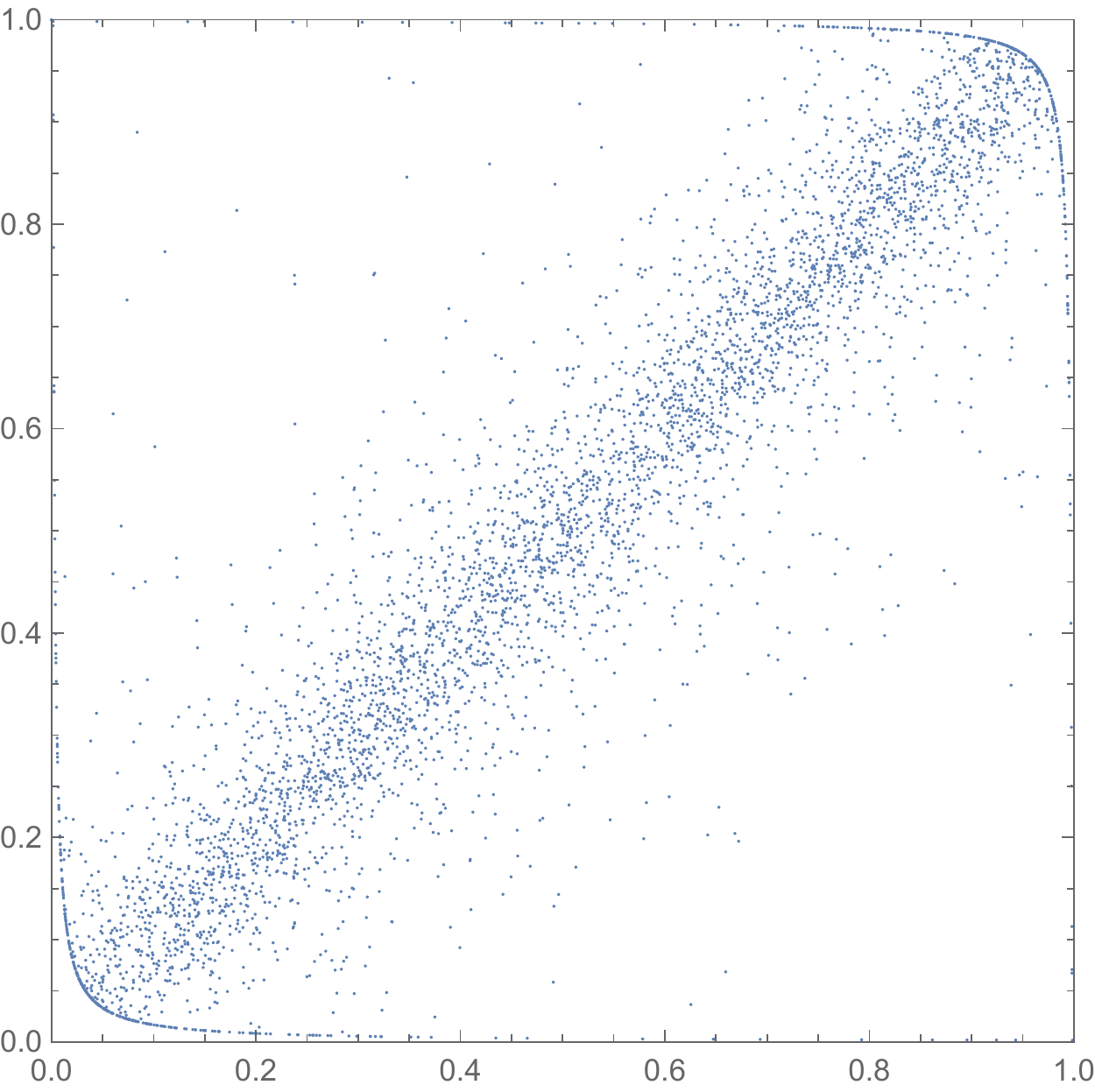}
            \caption{ Scatterplots of $\underline{F}_{\phi}$ for  $\phi = -\frac14, \frac14$, (first row), and $\phi= \frac34, 0.99$ (second row).} \label{fig Fspodaj}
\end{figure}

\begin{proof}
From \eqref{F_lower} we see that $\underline{F}\,_{-\frac12}=W$ and $\underline{F}_{1}=M$. Assume now that $\phi \in \left( -\frac12, 1 \right)$. To prove that $\underline{F}_{\phi}$ is a copula we use \cite[Theorem 2.1]{DuJa}. (See \cite{DuJa} for the definitions of Dini's derivatives as well.) For fixed $b$ the righthand side upper Dini derivative of $\underline{F}_{\phi}(a,b)$ is
\begin{equation}\label{F_lower-a}
D^+\underline{F}_{\phi}(a,b)=\left\{ \begin{array}{ll}
        0; & \text{if } b=0 \text{ or }b>0  \text{ and }a < \min\left\{1-b, \frac{1-\phi}{6b}\right\}, \vspace{1mm}\\
        \frac12\left(1+\frac{b-a}{\sqrt{\frac{2}{3}(1-\phi)+(b-a)^2}}\right);      & \text{if } b\notin\{0,1\}, \text{ and }\frac{1-\phi}{6b} \leqslant a < 1- \frac{1-\phi}{6(1-b)}, \vspace{1mm}\\
        1;            & \text{otherwise. }
        \end{array} \right.
\end{equation}
For $(a,b)\in\II^2$ such that $\frac{1-\phi}{6b} < a < 1- \frac{1-\phi}{6(1-b)}$, we have
\begin{equation}\label{F_lower-ab}
\frac{\partial^2}{\partial a\partial b}\underline{F}_{\phi}(a,b)=
        \frac{\frac{2}{3}(1-\phi)}{\left(\frac{2}{3}(1-\phi)+(b-a)^2\right)^{\frac32}}.
\end{equation}
Since the second derivative in \eqref{F_lower-ab} is positive and the Dini derivative in \eqref{F_lower-a} has a positive jump at points on the graphs of hyperbolas
     $$ab=\frac{1-\phi}{6}\text{ and }(1-a)(1-b)=\frac{1-\phi}{6}$$
     for $a\in [\frac12-\ell(\phi),\frac12+\ell(\phi)]$ or on the anti-diagonal $a+b=1$ for other $a\in\II$, it follows that statements in \textit{(i)} and \textit{(ii)} hold.

The derivative with respect to $\phi$ of the expression on the righthand side of \eqref{bound for d_2} is positive for $\phi \in \left( -\frac12, 1 \right)$. Thus \textit{(iii)} follows.

Statement \textit{(iv)} is a special case of Lemma \ref{lem:symm}.

Finally, we have $\phi(\underline{F}_{\phi})=6\int_0^1 \underline{F}_{\phi}(t,t)\, dt-2=2-\phi-\sqrt{6(1-\phi)}$, which is less then $\phi$ for $\phi\in\left(-\frac12,1\right)$. So, \textit{(v)} holds as well.
\end{proof}

The computation of the upper bound $\overline{F}_{\phi}$ is done using the opposite bounds to the ones that were used in the above proof. However, the computation is much more involved and it requires a careful analysis of several cases depending on $\phi$. We divide the unit square in several areas. With increasing value of $\phi$ their shapes evolve and they disappear one after another. (See Figure \ref{fig obmocja phi}.)

We define the areas in the unit square as
\begin{align*}
\Delta_{\phi}^1 = \bigg\{ (a,b)\in \II^2; \, & a \leqslant \frac12 \left( 1- \frac{\sqrt{3}}{3} \sqrt{1+2\phi} \right), \, b \geqslant  \frac12 \left( 1+ \frac{\sqrt{3}}{3} \sqrt{1+2\phi} \right),  \\ &   b \leqslant  a+\frac12 \left( 1- \frac{\sqrt{3}}{3} \sqrt{1+2\phi} \right) \bigg\} \\
\Delta_{\phi}^2 = \bigg\{ (a,b)\in \II^2; \, & b \leqslant  \frac12 \left( 1+ \frac{\sqrt{3}}{3} \sqrt{1+2\phi} \right), \\ & \frac13 \left( 2b-1 +\sqrt{(2b-1)^2+1+2\phi} \right) \leqslant a \leqslant \frac13 \left( b+1 -\sqrt{(2b-1)^2+1+2\phi} \right) \bigg\}
\\
\Delta_{\phi}^3 = \bigg\{ (a,b)\in \II^2; \, & a \geqslant \frac12 \left( 1- \frac{\sqrt{3}}{3} \sqrt{1+2\phi} \right), \\ & \frac13 \left( a+1 +\sqrt{(2a-1)^2+1+2\phi} \right) \leqslant b \leqslant \frac13 \left(2a+2 -\sqrt{(2a-1)^2+1+2\phi} \right) \bigg\} \\
\Delta_{\phi}^4 = \bigg\{ (a,b)\in \II^2; \, &  \frac13 \left(b+1 -\sqrt{(2b-1)^2+1+2\phi} \right) \leqslant a \leqslant \frac13 \left( b+1 +\sqrt{(2b-1)^2+1+2\phi} \right),\\
& \frac13 \left( a+1 -\sqrt{(2a-1)^2+1+2\phi} \right) \leqslant b \leqslant \frac13 \left(a+1 +\sqrt{(2a-1)^2+1+2\phi} \right), \\
& a \leqslant \sqrt{\frac23 (1-\phi)-(b-1)^2}, \, b \leqslant \sqrt{\frac23 (1-\phi)-(a-1)^2}  \bigg\} \\
\Delta_{\phi}^5 = \bigg\{ (a,b)\in \II^2; \, & b \geqslant \frac12 \left( 1- \frac{\sqrt{3}}{3} \sqrt{1+2\phi} \right), \\ & \frac13 \left( b+1 +\sqrt{(2b-1)^2+1+2\phi} \right) \leqslant a \leqslant \frac13 \left(2b+2 -\sqrt{(2b-1)^2+1+2\phi} \right) \bigg\}
\end{align*}
\begin{align*}
\Delta_{\phi}^6 = \bigg\{ (a,b)\in \II^2; \, & a \leqslant  \frac12 \left( 1+ \frac{\sqrt{3}}{3} \sqrt{1+2\phi} \right), \\ & \frac13 \left( 2a-1 +\sqrt{(2a-1)^2+1+2\phi} \right) \leqslant b \leqslant \frac13 \left( a+1 -\sqrt{(2a-1)^2+1+2\phi} \right) \bigg\}
\\
\Delta_{\phi}^7 = \bigg\{ (a,b)\in \II^2; \, & b \leqslant \frac12 \left( 1- \frac{\sqrt{3}}{3} \sqrt{1+2\phi} \right), \, a \geqslant  \frac12 \left( 1+ \frac{\sqrt{3}}{3} \sqrt{1+2\phi} \right),  \\ &   b \geqslant  a-\frac12 \left( 1- \frac{\sqrt{3}}{3} \sqrt{1+2\phi} \right) \bigg\}
\end{align*}

Notice that for $\phi$ close to $-\frac12$ all these areas are nonempty. When $\phi$ increases some of the areas vanish. More precisely, all the areas are nonempty for $\phi \in [-\frac12, \, -\frac13]$. For $\phi=-\frac12$ area $\Delta_\phi^4$ is reduced to the main diagonal. For $\phi \in \left( -\frac13,\, -\frac15 \right]$ only areas $\Delta_\phi^1$ and $\Delta_\phi^7$ are empty. For $\phi \in \left( -\frac15, \frac14 \right]$ only area $\Delta_\phi^4$ is nonempty. For $\phi \in \left( \frac14, 1\right]$ all the areas are empty.
In Figure \ref{fig obmocja phi} this dynamics is illustrated by the regionplots of the areas  for $\phi = -\frac12, -\frac25$ (first row), and $\phi= -\frac{32}{100}, 0$ (second row).

\begin{figure}[h]
            \includegraphics[width=5cm]{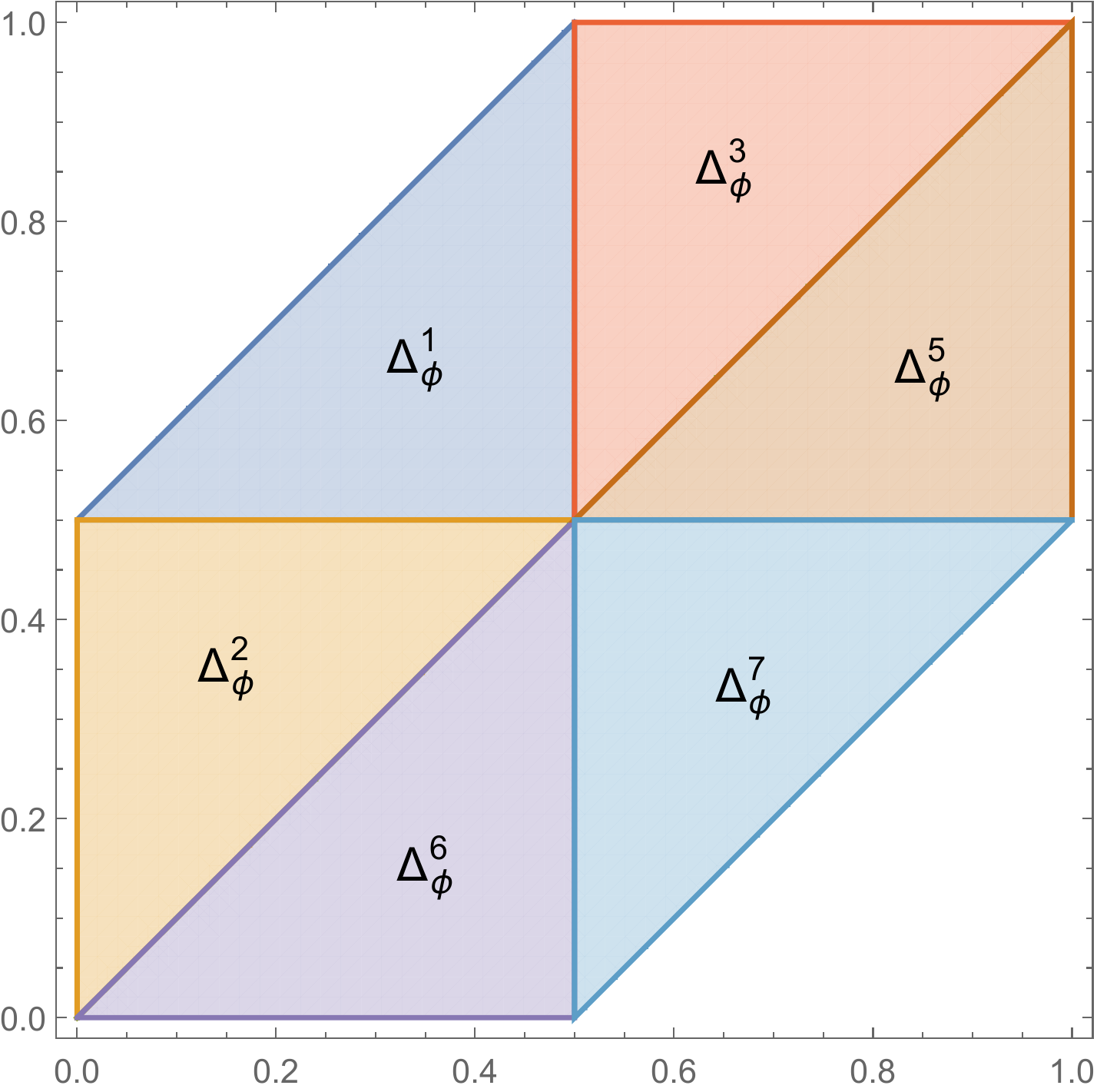} \hfil \includegraphics[width=5cm]{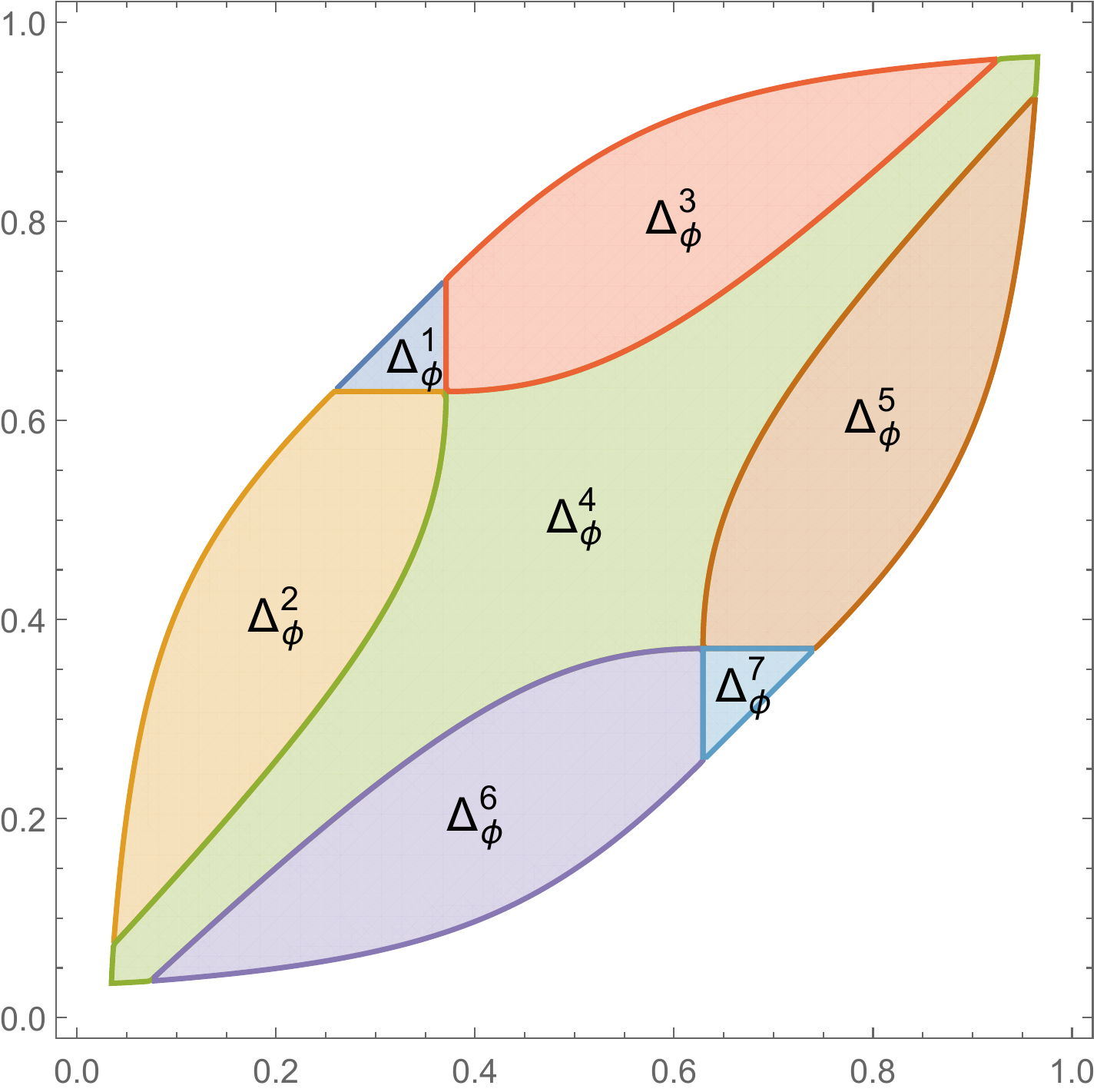} \\
						\includegraphics[width=5cm]{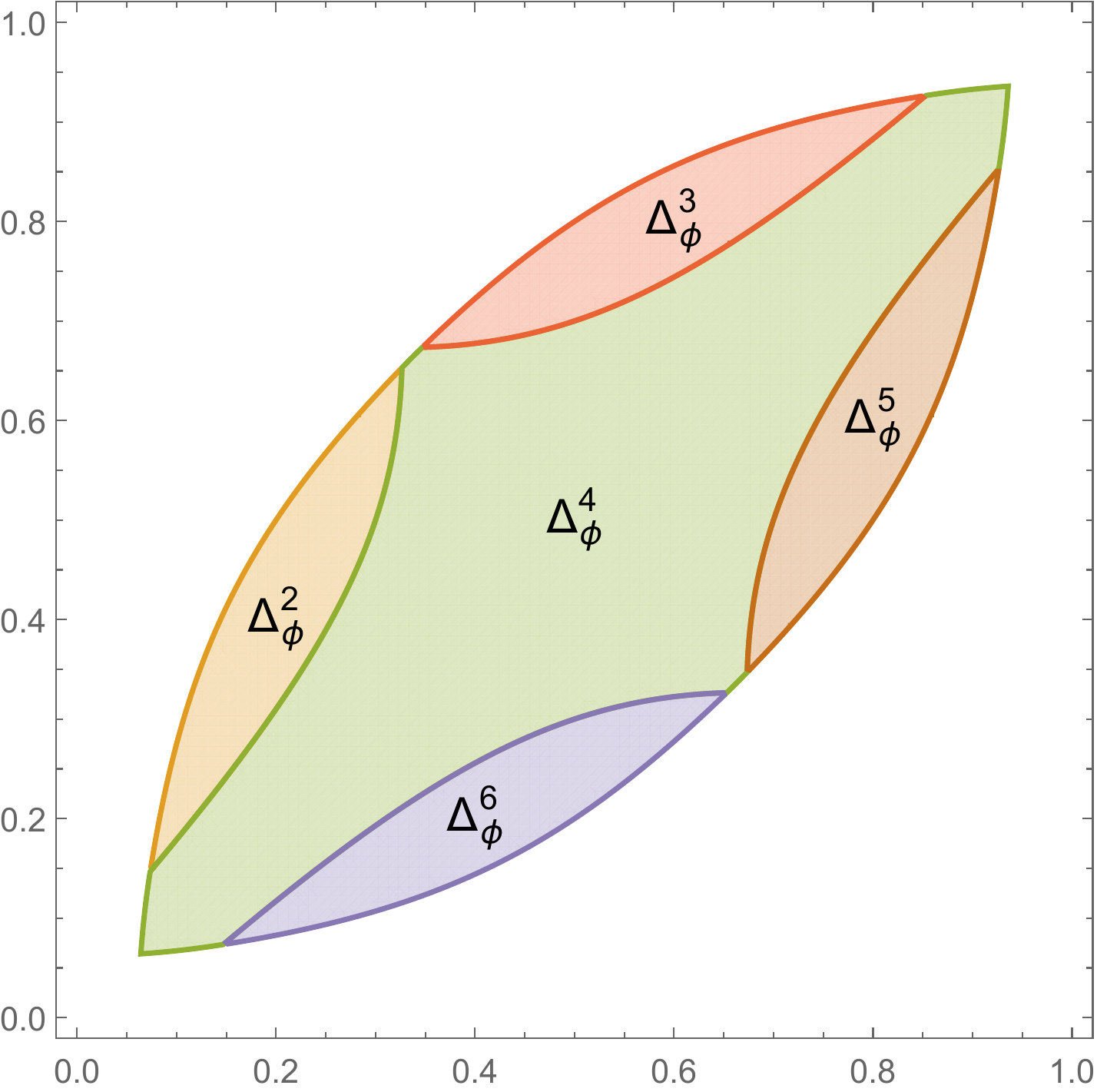} \hfil \includegraphics[width=5cm]{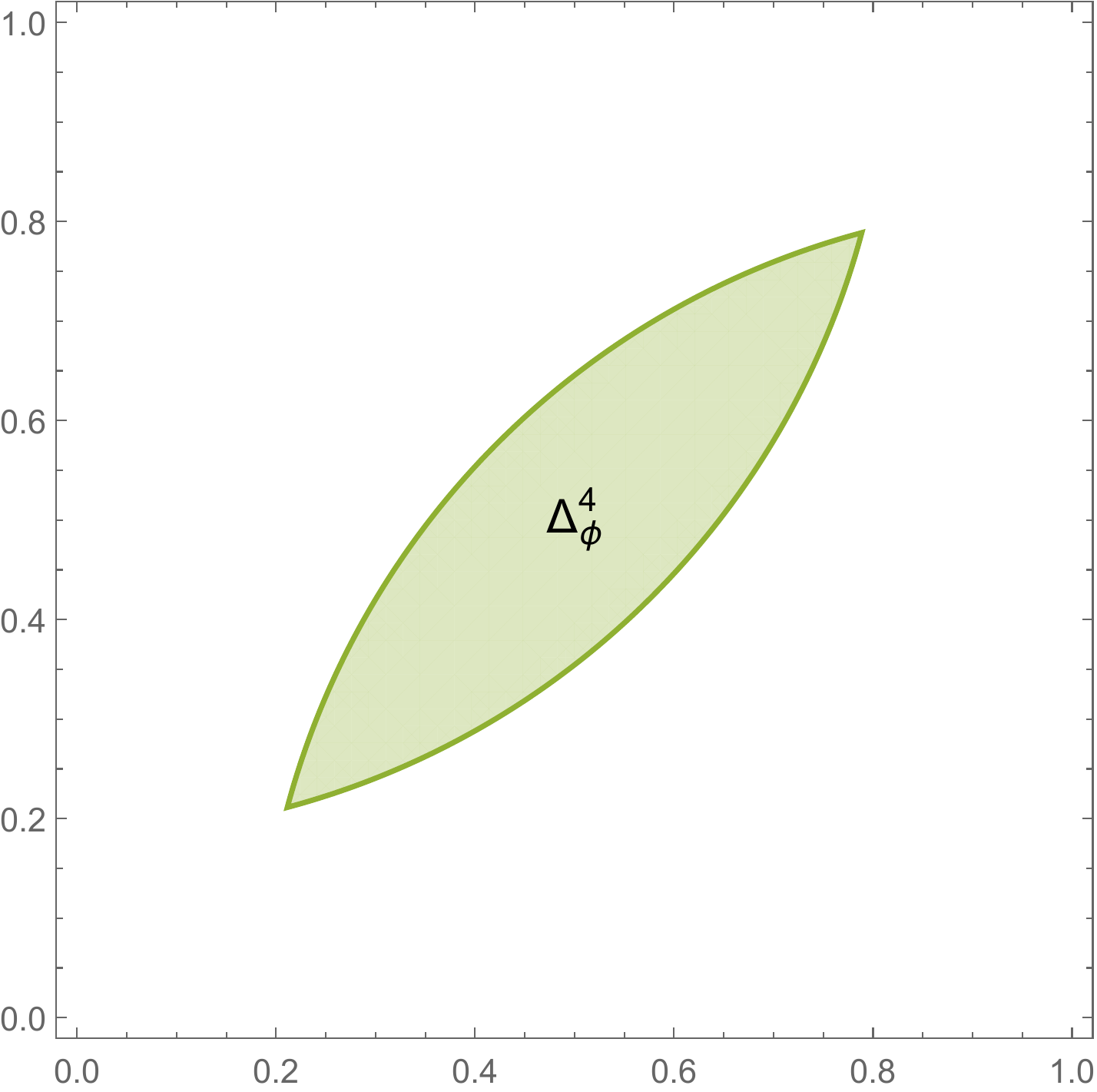}
            \caption{ Regionplots of the areas $\Delta_\phi^1, \ldots, \Delta_\phi^7$ for  $\phi = -\frac12, -\frac25, -\frac{32}{100}, 0$.} \label{fig obmocja phi}
\end{figure}

Next, we define functions of $\phi$ depending on $(a,b)$ on these areas. Since they are, for a fixed value of $(a,b)$, inverses of $\underline{f}_{a,b}(d)$ we chose to denote them by $\delta^i_{a,b}(\phi)$. They are
\begin{align*}\label{delta_ab^i}
\delta_{a,b}^1(\phi)&=\frac12 \left(2b-1+\frac{\sqrt{3}}{3}\sqrt{1+2\phi}\right),\\
\delta_{a,b}^2(\phi)&=\frac13 \left(2b-1+\sqrt{(1-2b)^2+1+2\phi}\right),
\\
\delta_{a,b}^3(\phi)&=\frac13 \left(a+3b-2+\sqrt{(1-2a)^2+1+2\phi}\right),\\
\delta_{a,b}^4(\phi)&=\frac12 \left(a+b-1+\sqrt{3(b-a)^2+(1-2a)(1-2b)+\frac23\left(1+2\phi\right)}\right),\\
\delta_{a,b}^5(\phi)&=\frac13 \left(3a+b-2+\sqrt{(1-2b)^2+1+2\phi}\right),\\
\delta_{a,b}^6(\phi)&=\frac13 \left(2a-1+\sqrt{(1-2a)^2+1+2\phi}\right),\\
\delta_{a,b}^7(\phi)&=\frac12 \left(2a-1+\frac{\sqrt{3}}{3}\sqrt{1+2\phi}\right).
\end{align*}

We are now ready to state one of our main results.

\begin{theorem}\label{thm_phi_upp1}
The pointwise supremum $\overline{F}_{\phi}$ 
of $\FF_{\phi}$ for any $\phi\in[-\frac12, 1]$ and for any $(a,b)\in\II^2$ 
is given by
\begin{equation}\label{F_upper1}
\overline{F}_{\phi}(a,b)=\left\{ \begin{array}{ll}
        \delta_{a,b}^1(\phi);      & \text{if } (a,b) \in \Delta_{\phi}^1 , \vspace{1mm}\\
        \delta_{a,b}^2(\phi);      & \text{if } (a,b) \in \Delta_{\phi}^2 , \vspace{1mm}\\
        \delta_{a,b}^3(\phi);      & \text{if } (a,b) \in \Delta_{\phi}^3 , \vspace{1mm}\\
        \delta_{a,b}^4(\phi);      & \text{if } (a,b) \in \Delta_{\phi}^4 , \vspace{1mm}\\
        \delta_{a,b}^5(\phi);      & \text{if } (a,b) \in \Delta_{\phi}^5 , \vspace{1mm}\\
        \delta_{a,b}^6(\phi);      & \text{if } (a,b) \in \Delta_{\phi}^6 , \vspace{1mm}\\
        \delta_{a,b}^7(\phi);      & \text{if } (a,b) \in \Delta_{\phi}^7 , \vspace{1mm}\\
        M(a,b);       & \text{otherwise. }
       \end{array} \right.
\end{equation}
\end{theorem}

Before we give the proof we gather some observations in the following corollary:

\begin{corollary}\label{cor phi2}
Suppose that $\overline{F}_{\phi}$ is the supremum given in Theorem \ref{thm_phi_upp1}. Then:
\begin{enumerate}[(i)]
 \item We have $\overline{F}_{-\frac12}$ is a shuffle of $M$, i.e., $\overline{F}_{-\frac12}  = M(2,\{[0,\frac12],[\frac12,1]\},(2,1),1)$ and $\overline{F}_{\phi}=M$ for $\phi \in \left[ \frac14, 1 \right]$.
 \item For $\phi \in \left( -\frac12, \frac14\right)$ the bound $\overline{F}_{\phi}$ is not a copula, but a proper quasi-copula.
 \item $\overline{F}_{\phi}$ is increasing in $\phi$ (in the concordance order on quasicopulas).
 \item $\overline{F}_{\phi}$ is symmetric and radially symmetric: $\overline{F}_{\phi}(a,b)=\overline{F}_{\phi}(b,a)$ and $\overline{F}_{\phi}(a,b)=\widehat{\overline{F}}_{\phi}(a,b)$.
 \item If we extend the weak measure of concordance $\phi$ to any quasicopula $Q$ by defining $$\phi(Q)=6\int_0^1 Q(t,t) dt - 2$$
     then 
     we have $\phi\left(\overline{F}_{\phi}\right)<\phi$ for all $\phi\in(0,1)$. (See Figure \ref{phi(phi)}.)
 \end{enumerate}
\end{corollary}

\begin{figure}[h]
            \includegraphics[width=6cm]{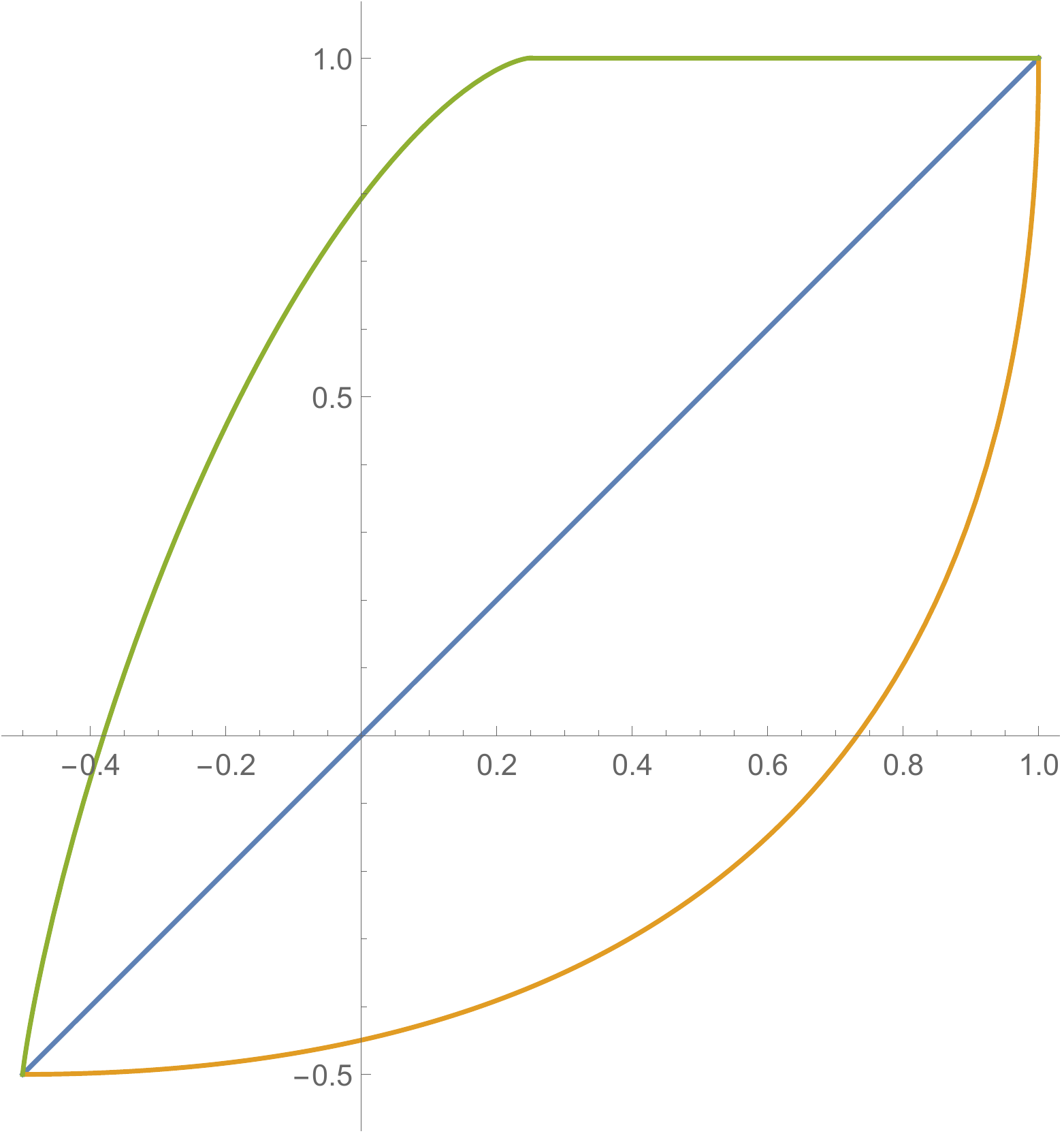}
            \caption{Graphs of values of $\phi(\underline{F}_{\phi})$ (orange) and $\phi(\overline{F}_{\phi})$ (green). } \label{phi(phi)}
\end{figure}

\begin{proof} First notice that $\overline{F}_{\phi}$ equals the Fr\'echet-Hoeffding upper bound $M$ for any $\phi \in [\frac14, 1]$, since then all the regions $\Delta_{\phi}^i$ disappear. If $\phi = -\frac12$, then $\overline{F}_{\phi}$ is a shuffle of $M$, namely
$\overline{F}_{-\frac12}  = M(2,\{[0,\frac12],[\frac12,1]\},(2,1),1)$. For any $\phi \in (-\frac12, \frac14)$,  the point $(\frac12,\frac12)$ lies inside the area $\Delta_\phi^4$ and the second derivative $\frac{\partial^2 \delta_{a,b}^4(\phi)}{\partial a \partial b}$ at the point $(\frac12,\frac12)$ equals $-\frac{\sqrt{3}}{2 \sqrt{4 \phi +2}}<0$. Therefore, \cite[Theorem 2.1]{DuSe} implies that $\overline{F}_{\phi}$ is {not} a copula. So, \textit{(i)} and \textit{(ii)} hold.

A careful analysis that we omit shows that $\overline{F}_{\phi}$ is an increasing function of $\phi$, as \textit{(iii)} asserts. Statement \textit{(iv)} is a consequence of Lemma \ref{lem:symm}.

A rather technical calculation shows that
$$
\phi(\overline{F}_{\phi})=\left\{ \begin{array}{ll}
        \frac12\left(2-\sqrt{3-12\phi}+(1+2\phi)\log\left(3+\sqrt{3-12\phi}\right)\right); & \text{if } \ -\frac12\leqslant\phi\leqslant\frac14, \vspace{1mm}\\
        1;            & \text{otherwise. }
        \end{array} \right.
$$
So \textit{(v)} holds (see Figure \ref{phi(phi)}).
\end{proof}

In the Figure \ref{f zgoraj} we give the 3D plots of the quasicopulas $\overline{F}_{\phi}$ for $\phi = -\frac25$, $-\frac{32}{100}$ and 0.\\

We remark that the set $\FF_{1}$ consists of $M$ only, while $\FF_{-\frac12}$ consist of all copulas $C$ such that $C\leqslant\overline{F}_{-\frac12}$ since $\phi$ is monotone. This is a special case of a result of Fuchs and McCord \cite{FuMcC}.

\begin{figure}[h]
            \includegraphics[width=6cm]{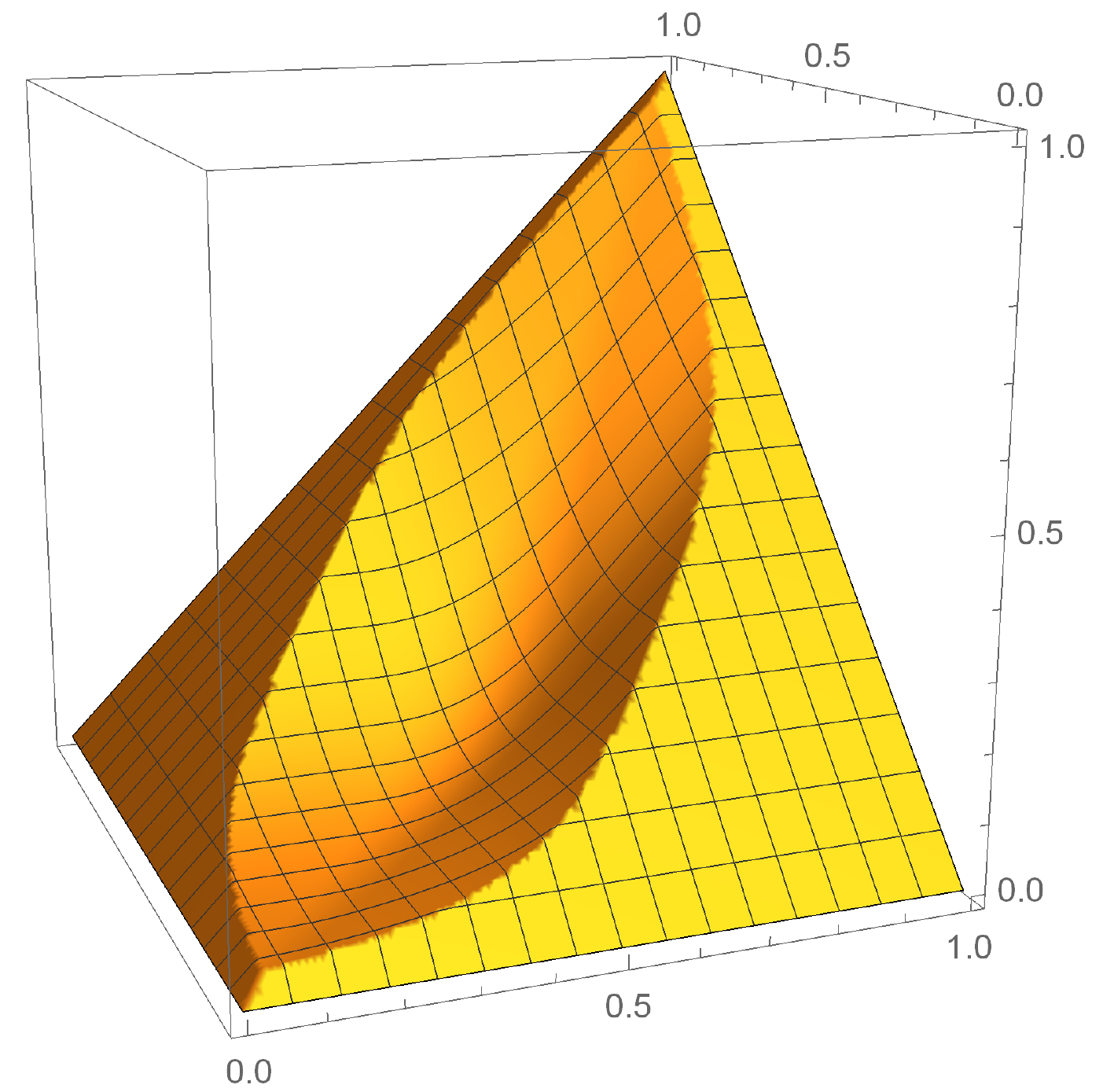} \hfil \includegraphics[width=6cm]{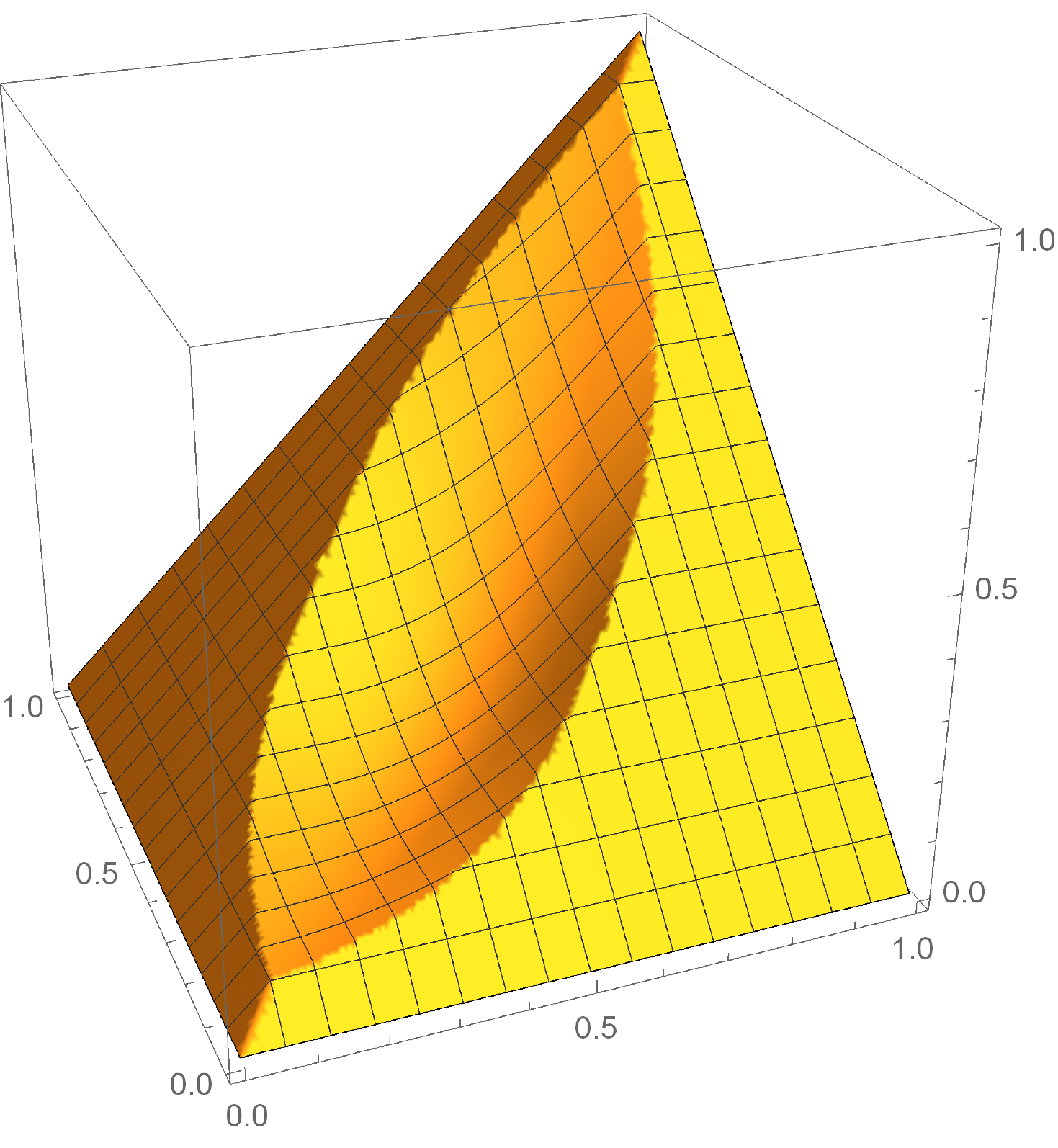}\\
            \includegraphics[width=6cm]{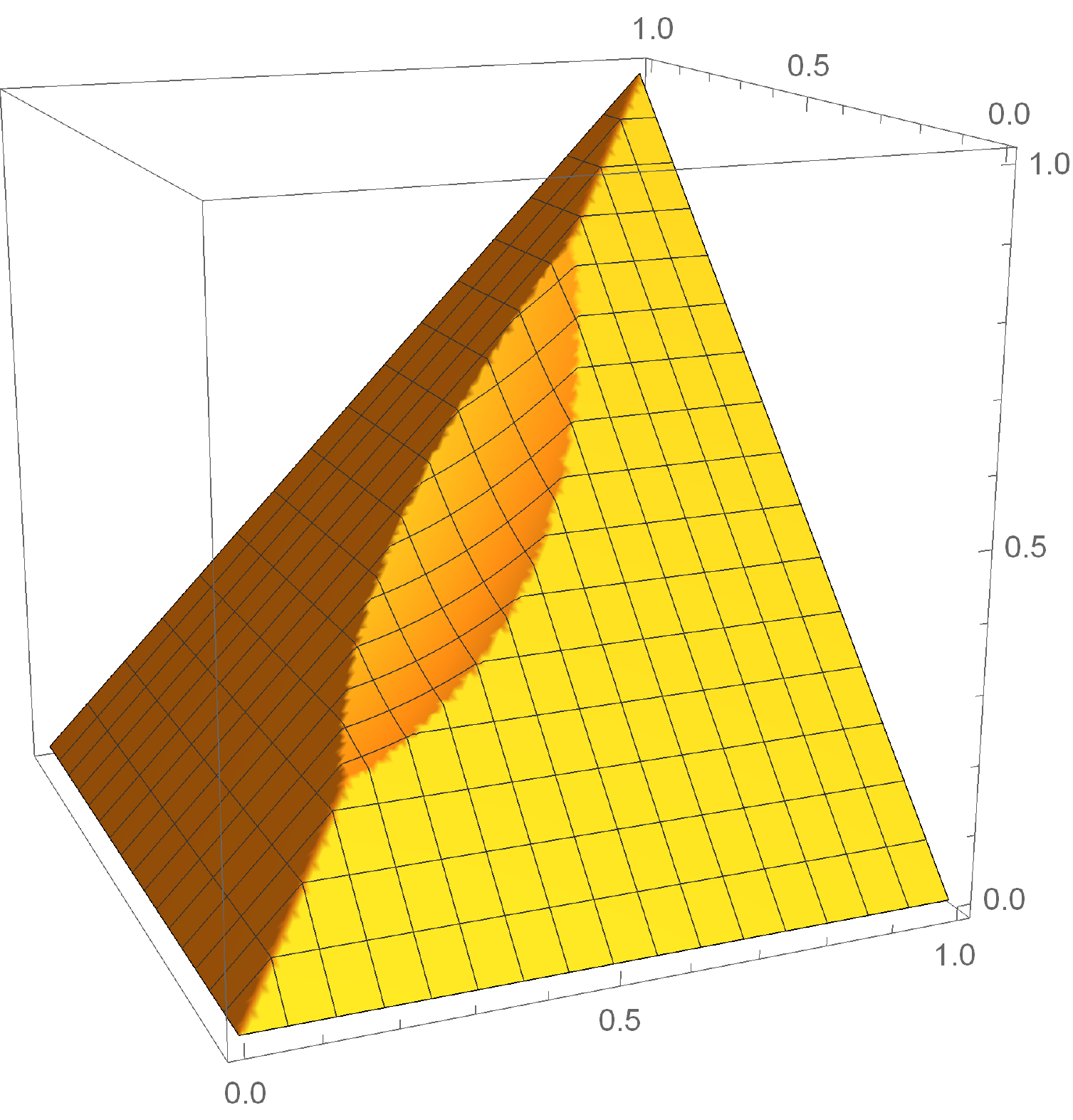}
            \caption{ Graphs of quasicopulas $\overline{F}_{\phi}$ for $\phi = -\frac25$, $-\frac{32}{100}$ and 0.} \label{f zgoraj}
\end{figure}

\begin{proof}[Proof of Theorem \ref{thm_phi_upp1}]
The pointwise supremum $\overline{F}_{\phi}$ is symmetric and radially symmetric by Lemma \ref{lem:symm}.
Thus we may assume that the point $(a,b)$ lies in the triangle $\Delta= \{ (a,b) \in \II^2; \, a \leqslant b, \, a+b \leqslant 1 \}$.

Now, we use Proposition \ref{prop1} to show that for $(a,b)\in\Delta$ we have
\begin{equation}\label{f underline phi}
  \begin{split}
       \underline{f}_{a,b}(d)&=\phi\left(\underline{C}^{(a,b)}_{d-W(a,b)}\right)=\frac32\cQ\left(M,\underline{C}^{(a,b)}_{c_1}\right)-\frac12= \\
       &=\left\{ \begin{array}{ll}
        -\frac12;               & \text{if } b \geqslant d + \frac12, \vspace{1mm}\\
        f_{a,b}^1(d);           & \text{if } \frac12(1+d) \leqslant b \leqslant d + \frac12,  \vspace{1mm}\\
        f_{a,b}^2(d);           & \text{if } a+d \leqslant b \leqslant \frac12(1+d), \vspace{1mm}\\
        f_{a,b}^4(d);           & \text{if } b \leqslant a+d.
                \end{array} \right. ,
	\end{split}
\end{equation}
where
\begin{align*}
f_{a,b}^1(d)&=\frac32\left(2d+1-2b\right)^2-\frac12,\\
f_{a,b}^2(d)&=\frac32 d\left(2+3d-4b\right)-\frac12,\\
f_{a,b}^4(d)&=3d\left(1+d-a-b\right)-\frac32 \left(b-a\right)^2-\frac12.
\end{align*}

Since $d=\underline{C}^{(a,b)}_{c_1}(a,b)$
it follows that $W(a,b)=0\leqslant d\leqslant M(a,b)=a$. For such values of $d$ the expression on the right-hand side of \eqref{f underline phi} is increasing in $d$ and thus, the  maximal possible value of $\underline{f}_{a,b}(d)$ is achieved when $d=a$. Then, we have
$$\underline{f}_{a,b}(a)=\left\{ \begin{array}{ll}
        -\frac12;                              & \text{if } b \geqslant a + \frac12, \vspace{1mm}\\
        6(b-a)^2-6(b-a)+1;                     & \text{if } \frac12(1+a) \leqslant b \leqslant a + \frac12, \vspace{1mm}\\
        \frac32 a\left(3a-4b+2\right)-\frac12; & \text{if } 2a \leqslant b \leqslant \frac12(1+a), \vspace{1mm}\\
        1-\frac32\left((a-1)^2+b^2\right);     & \text{if } b \leqslant 2a.
                \end{array} \right. $$
For the function $\underline{f}_{a,b}:[0, a]\to [-\frac12, 
\underline{f}_{a,b}(a)]$ we need to find its inverse.
If for a given value $\phi \in \left[ -\frac12, 1 \right]$ it holds that $\phi \geqslant \underline{f}_{a,b}(a)$, we take $d=a=M(a,b)$.
Otherwise, we take the inverses of the expressions for $f_{a,b}^i$ which are $\delta_{a,b}^i(\phi)$ for $i=1, 2, 4$.
The inequality $\phi \geqslant \underline{f}_{a,b}(a)$ gives us the area
\begin{align*}
\bigg\{ (a,b)\in \Delta; \,
& \left( b \geqslant  a+\frac12 \left( 1- \frac{\sqrt{3}}{3} \sqrt{1+2\phi} \right) \textrm{ and } b \geqslant \frac12(1+a) \right) \textrm{ or } \\
& \left( a \leqslant \frac13 \left( 2b-1 +\sqrt{(2b-1)^2+1+2\phi} \right)  \textrm{ and } 2a \leqslant b \leqslant \frac12(1+a) \right) \textrm{ or } \\
& \left( b \geqslant \sqrt{\frac23 (1-\phi)-(a-1)^2} \textrm{ and } b\leqslant 2a \right)
\bigg\}
\end{align*}
which is equal to the area $\Delta \setminus (\Delta^1_\phi \cup \Delta^2_\phi\cup \Delta^4_\phi)$. {We continue by considering the inequalities $\phi \geqslant \underline{f}^i_{a,b}(a)$ for $i=1,2,4$. Their  }
careful consideration yields areas where each of the expressions $\delta_{a,b}^i(\phi)$ is valid {and these are exactly} the areas $\Delta^i_\phi \cap \Delta$ for $i=1,2,4$. Now, we reflect the expressions $\delta_{a,b}^1(\phi),  \delta_{a,b}^2(\phi),  \delta_{a,b}^4(\phi)$ over the main diagonal and over the counter-diagonal to obtain the expressions $ \delta_{a,b}^1(\phi), \ldots,  \delta_{a,b}^7(\phi)$. The areas where they are valid are the reflections of the areas $\Delta^1_\phi\cap \Delta, \Delta^2_\phi\cap \Delta, \Delta^4_\phi\cap \Delta$ over the main diagonal and over the counter-diagonal, i.e., the areas $\Delta^1_\phi, \ldots, \Delta^7_\phi$.

We conclude that the required upper bound is given by \eqref{F_upper1}. The detailed calculations of the functions and their domains were done with a help of Wolfram Mathematica software \cite{Mathematica}.
\end{proof}

\section{Local bounds for Gini's gamma}\label{sec:gamma}

In this section we compute the local bounds for Gini's gamma. For each value $\gamma\in[-1,1]$ we write
\begin{equation}\label{eq:gamma}
  \GG_{\gamma}:=\{C\in\CC\,|\,\gamma(C)=\gamma\}.
\end{equation}
Our aim is to find the upper and lower bound of $\GG_{\gamma}$. We denote by $\underline{G}_{\gamma}(a,b)$ the pointwise infimum of $\GG_{\gamma}$ and by $\overline{G}_{\gamma}(a,b)$ the pointwise supremum of $\GG_{\gamma}$. The computation of the upper bound of $\GG_{\gamma}$ is done in a way that is similar to the one used for the upper bound of $\FF_{\phi}$ in the previous section. The lower bound of $\GG_{\gamma}$ is then obtained using the symmetries that hold for Gini's gamma and are proved in Lemma \ref{lem:symm}. Part (b) of Lemma \ref{lem:symm} does not hold for Spearman's footrule, so the argument there had to be different. The reason is the fact that Spearman's footrule is only a weak measure of concordance while Gini's gamma is a measure of concordance.

Now, we fix a value $\gamma\in[-1,1]$ and we choose a point $(a,b)$ in the interior of the unit square $\II^2$. We will find the minimal and maximal value of $C(a,b)$ for all copulas in $\GG_{\gamma}$.
Suppose that $C\in \GG_{\gamma}$ and $C(a,b)=d$. By results of \cite{KoBuKoMoOm2} it follows that
$$\underline{C}^{(a,b)}_{c_1} \leqslant C \leqslant \overline{C}^{(b,a)}_{c_2}.$$
Recall that $d=\underline{C}^{(a,b)}_{c_1}(a,b)=W(a,b)+c_1$ and $d=\overline{C}^{(b,a)}_{c_2}(a,b)=M(a,b)-c_2$ and so (\ref{c_i}) holds. Since concordance functions are monotone it follows that
$$\underline{g}_{a,b}(d)\leqslant \gamma(C)\leqslant \overline{g}_{a,b}(d),$$
where we write
\begin{equation}
\label{g-lower}\underline{g}_{a,b}(d)=\gamma\left(\underline{C}^{(a,b)}_{c_1}\right)=\gamma\left(\underline{C}^{(a,b)}_{d-W(a,b)}\right)
\end{equation}
and
\begin{equation}\label{g-upper}
\overline{g}_{a,b}(d)=\gamma\left(\overline{C}^{(b,a)}_{c_2}\right)=\gamma\left(\overline{C}^{(b,a)}_{M(a,b)-d}\right).
\end{equation}

First, we compute the upper bound $\overline{G}_{\gamma}(a,b)$. To simplify the expressions we introduce some new notation. We divide the unit square in several areas depending on the value of $\gamma\in [-1,1]$.
The shapes of these regions evolve and they vanish one after another with increasing value of $\gamma$. The dynamics can be observed in Figure \ref{fig obmocja gamma}. We define the areas in the unit square as
\begin{align*}
\Omega_{\gamma}^1 = \bigg\{ (a,b)\in \II^2; \,  & a\leqslant \frac12, \,\frac12 \left( 1+\frac{1+\gamma}{1-2a} \right)  \leqslant b \leqslant 1 -\frac{1+\gamma}{4a}\bigg\}
\\
\Omega_{\gamma}^2 = \bigg\{ (a,b)\in \II^2; \, & b \leqslant \frac12 \left(1+ \frac{1+\gamma}{1-2a} \right), (1+2a-2b)^2+4a(1-b) \geqslant 1+\gamma, \\
& b \geqslant \frac13 \left(a+1 +\frac12 \sqrt{(2a-1)^2+3(1+\gamma)} \right), b \geqslant \frac14 \left(6a-1+\frac{1+\gamma}{1-2a} \right)
\bigg\}
\end{align*}
\begin{align*}
\Omega_{\gamma}^3 = \bigg\{ (a,b)\in \II^2; \, &
b \leqslant \frac13 \left(a+1 +\frac12 \sqrt{(2a-1)^2+3(1+\gamma)} \right), b \leqslant \frac18 \left(3a+6- \frac{1+\gamma}{a} \right), \\
& a \leqslant \frac{1}{11} \left(3+5b- \sqrt{9(2b-1)^2+11(1+\gamma)} \right)
\bigg\}  \\
\Omega_{\gamma}^4 =  \bigg\{ (a,b)\in \II^2; \, & a \geqslant \frac13 \left(b+1 - \frac12 \sqrt{(2b-1)^2+3(1+\gamma)} \right),
a \geqslant \frac18 \left(3b-1+ \frac{1+\gamma}{1-b} \right), \\
& b \geqslant \frac{1}{11} \left(3+5a + \sqrt{9(2a-1)^2+11(1+\gamma)} \right)
\bigg\}\\
\Omega_{\gamma}^5 = \bigg\{ (a,b)\in \II^2; \, & \frac{1}{11} \left(3+5b- \sqrt{9(2b-1)^2+11(1+\gamma)} \right) \leqslant a \leqslant \frac{1}{11} \left(3+5b + \sqrt{9(2b-1)^2+11(1+\gamma)} \right), \\
& \frac{1}{11} \left(3+5a- \sqrt{9(2a-1)^2+11(1+\gamma)} \right) \leqslant b \leqslant \frac{1}{11} \left(3+5a + \sqrt{9(2a-1)^2+11(1+\gamma)} \right), \\
& b \leqslant -2a+ \sqrt{3a(a+2)-(1+\gamma)}, a \leqslant -2b+ \sqrt{3b(b+2)-(1+\gamma)}
 \bigg\}  \\
\Omega_{\gamma}^6 =\bigg\{ (a,b)\in \II^2; \, & b \geqslant \frac13 \left(a+1 - \frac12 \sqrt{(2a-1)^2+3(1+\gamma)} \right),
b \geqslant \frac18 \left(3a-1+ \frac{1+\gamma}{1-a} \right), \\
& a \geqslant \frac{1}{11} \left(3+5b + \sqrt{9(2b-1)^2+11(1+\gamma)} \right)
\bigg\}
\\
\Omega_{\gamma}^7 = \bigg\{ (a,b)\in \II^2; \, &  a \leqslant \frac13 \left(b+1 +\frac12 \sqrt{(2b-1)^2+3(1+\gamma)} \right), a \leqslant \frac18 \left(3b+6- \frac{1+\gamma}{b} \right), \\
& b \leqslant \frac{1}{11} \left(3+5a- \sqrt{9(2a-1)^2+11(1+\gamma)} \right)
\bigg\} \\
\Omega_{\gamma}^8 =  \bigg\{ (a,b)\in \II^2; \, & a \leqslant \frac12 \left(1+ \frac{1+\gamma}{1-2b} \right), (1-2a+2b)^2+4b(1-a) \geqslant 1+\gamma, \\
& a \geqslant \frac13 \left(b+1 +\frac12 \sqrt{(2b-1)^2+3(1+\gamma)} \right), a \geqslant \frac14 \left(6b-1+\frac{1+\gamma}{1-2b} \right)
\bigg\} \\
\Omega_{\gamma}^9 = \bigg\{ (a,b)\in \II^2; \,  & b\leqslant \frac12, \,\frac12 \left( 1+\frac{1+\gamma}{1-2b} \right)  \leqslant a \leqslant 1 -\frac{1+\gamma}{4b}\bigg\}
\end{align*}

Notice that for $\gamma$ close to $-1$ all these areas are nonempty. When $\gamma$ increases some of the areas vanish. More precisely, all the areas are nonempty for $\gamma \in [-1, \, -\frac34]$. For $\gamma=-1$ area $\Omega_\gamma^5$ is reduced to the main diagonal and the areas $\Omega_\gamma^2$ and $\Omega_\gamma^8$ are reduced to unions of two perpendicular line segments on the lines $a=\frac12$, $b=\frac12$. For $\gamma \in \left( -\frac34,\, -\frac49 \right]$ only areas $\Omega_\gamma^1$ and $\Omega_\gamma^9$ are empty. For $\gamma \in \left( -\frac49,\, -\frac{4}{13} \right]$ the areas $\Omega_\gamma^1$, $\Omega_\gamma^2$, $\Omega_\gamma^8$ and $\Omega_\gamma^9$ are empty. For $\gamma \in \left( -\frac{4}{13}, \frac12 \right]$ only area $\Omega_\gamma^5$ is nonempty. For $\gamma \in \left( \frac12, 1\right]$ all the areas are empty.
In Figure \ref{fig obmocja gamma} we give the regionplots of the areas  for $\gamma = -1, -\frac{24}{25}, -\frac45$, (first row), and $\gamma= -\frac{7}{10}, -\frac{43}{100}, \, 0$ (second row).

\begin{figure}[h]
            \includegraphics[width=5cm]{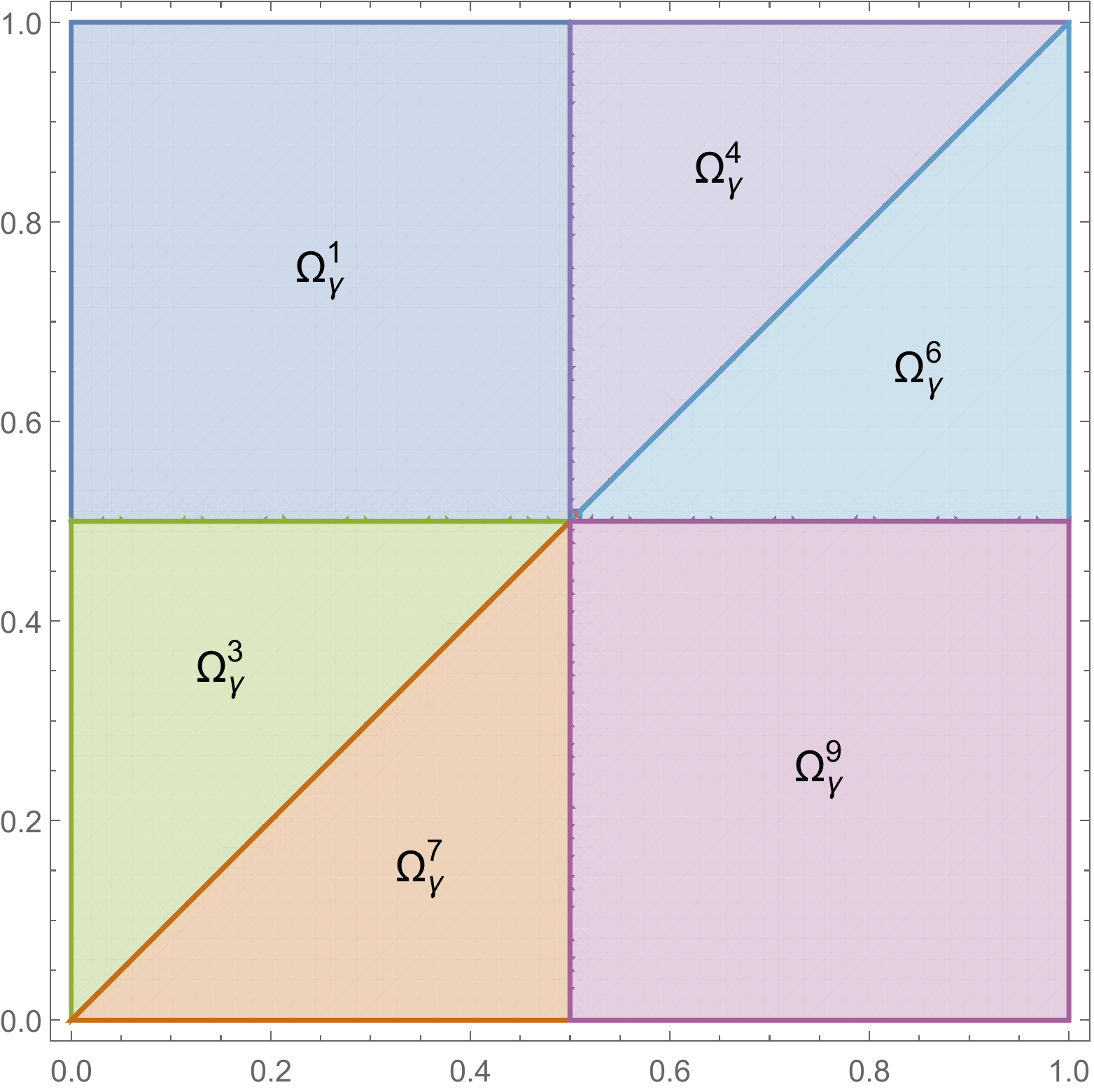} \hfil
            \includegraphics[width=5cm]{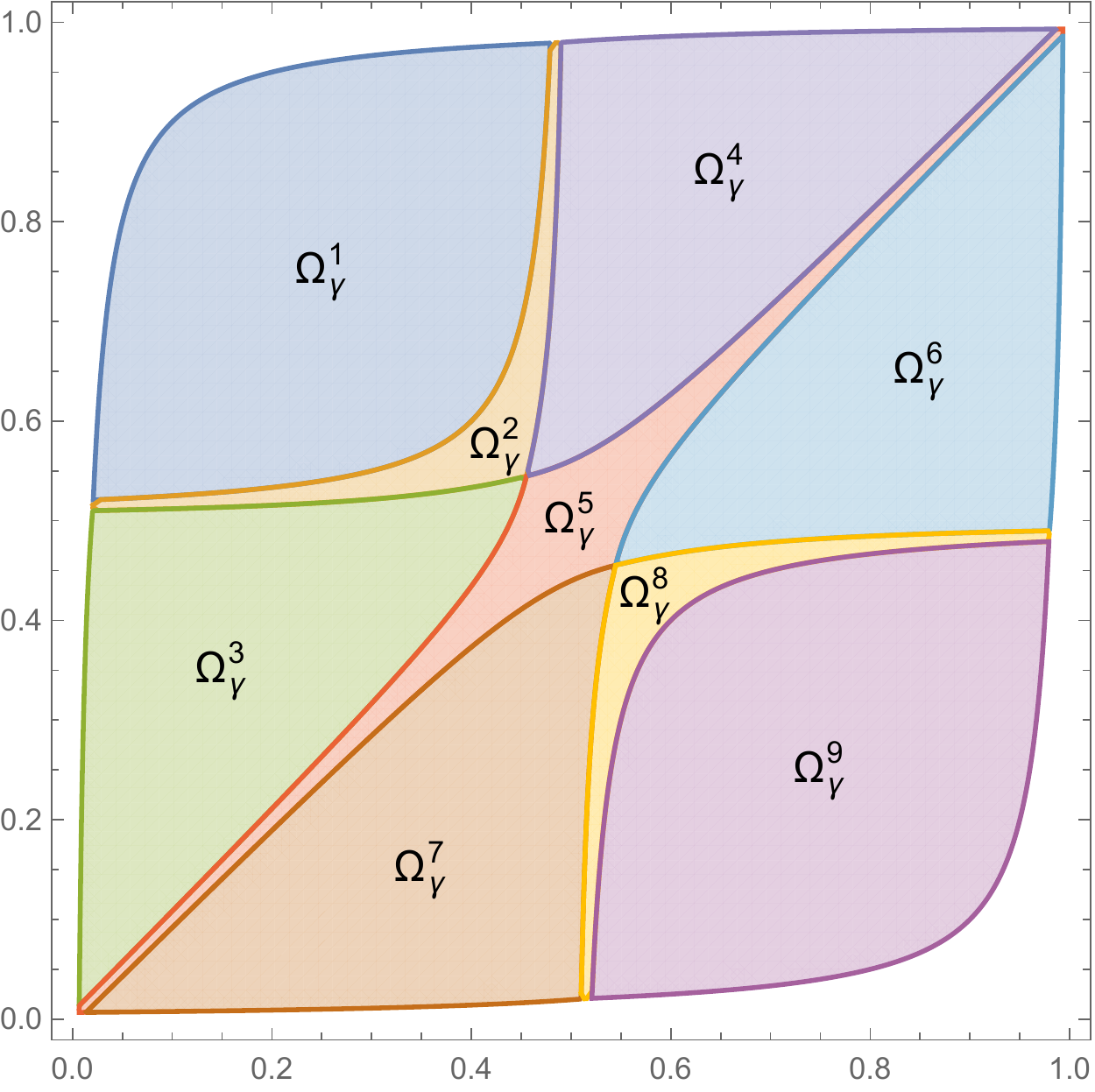} \hfil
            \includegraphics[width=5cm]{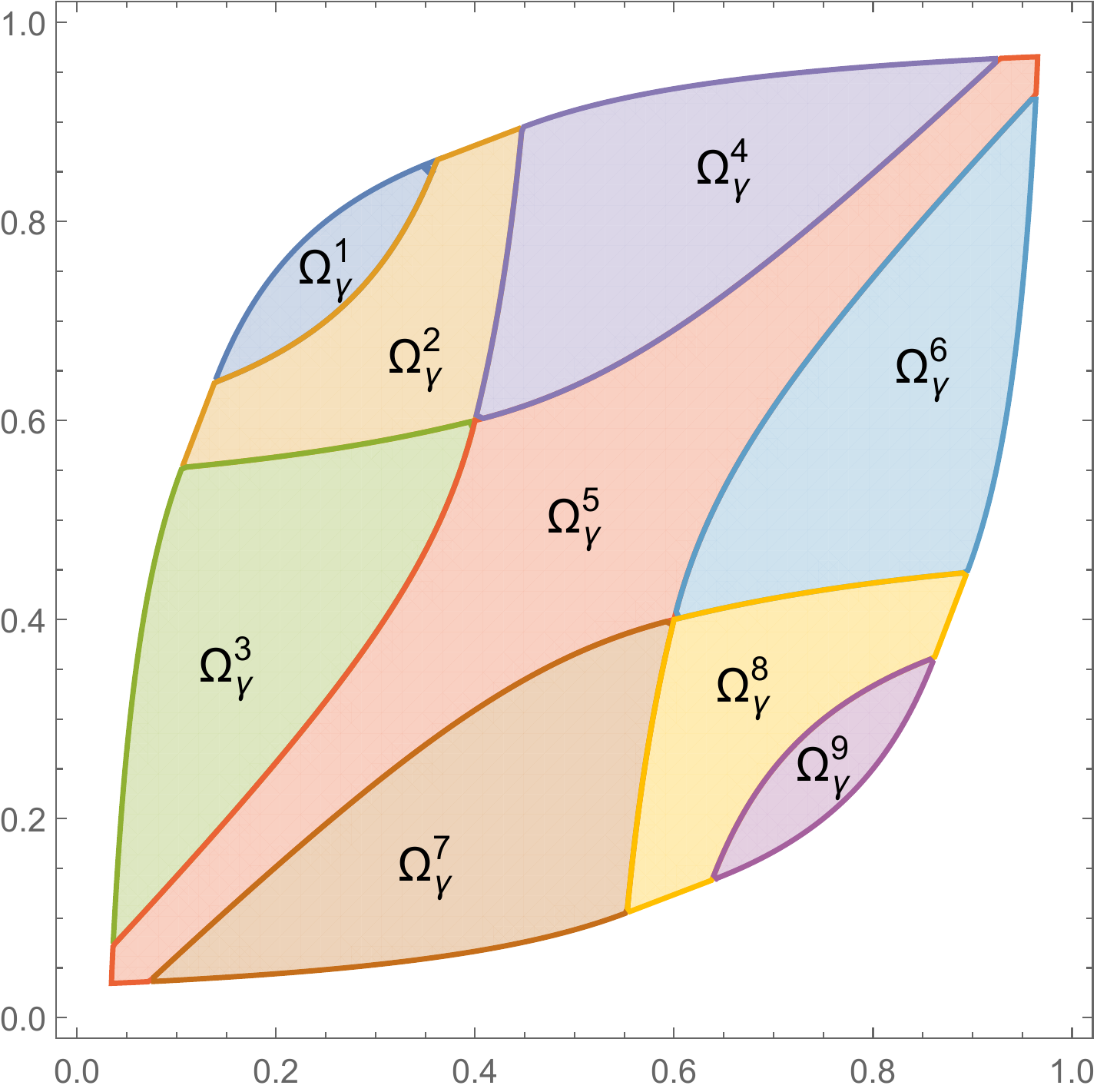} \\
			\includegraphics[width=5cm]{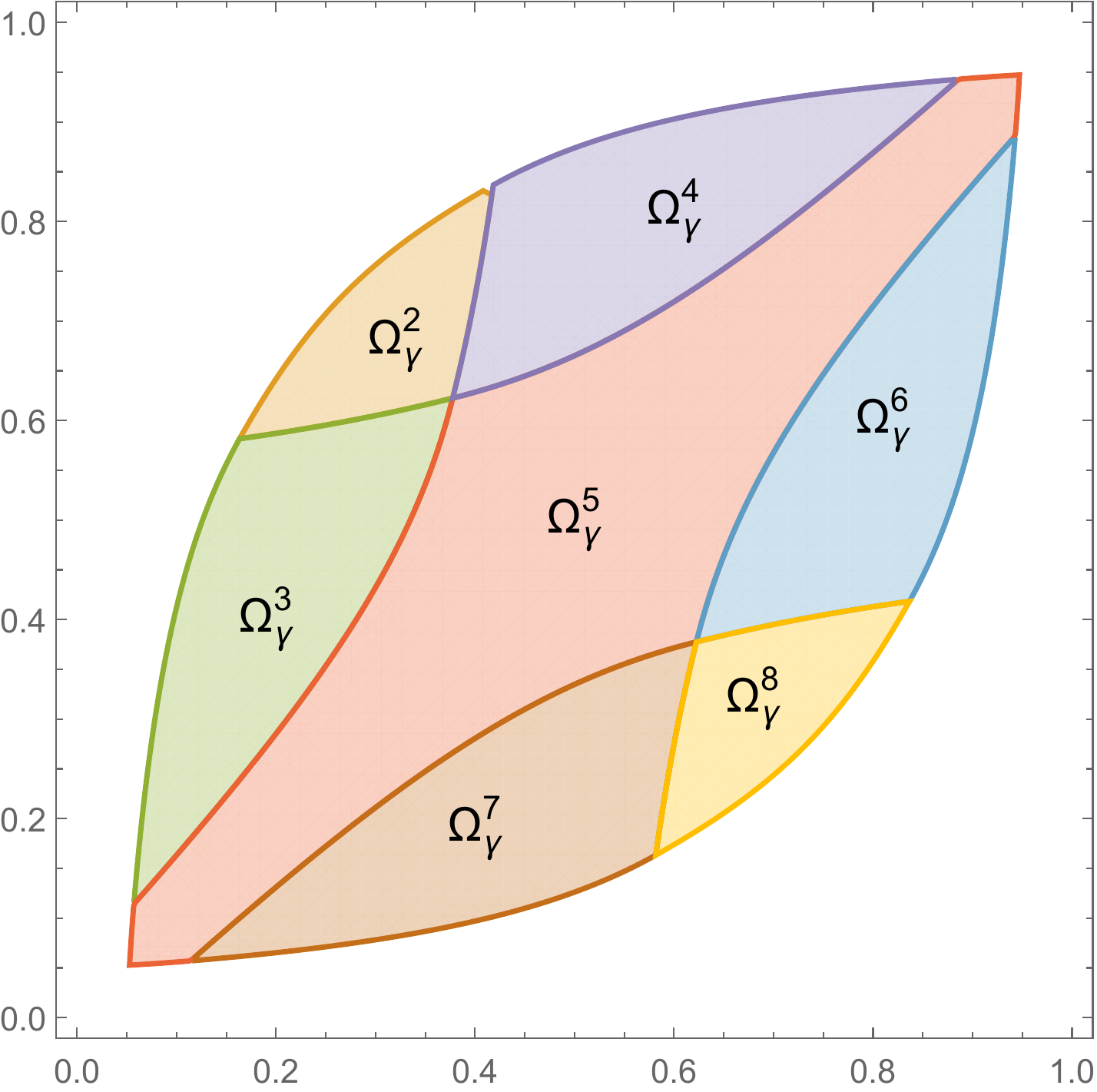} \hfil
            \includegraphics[width=5cm]{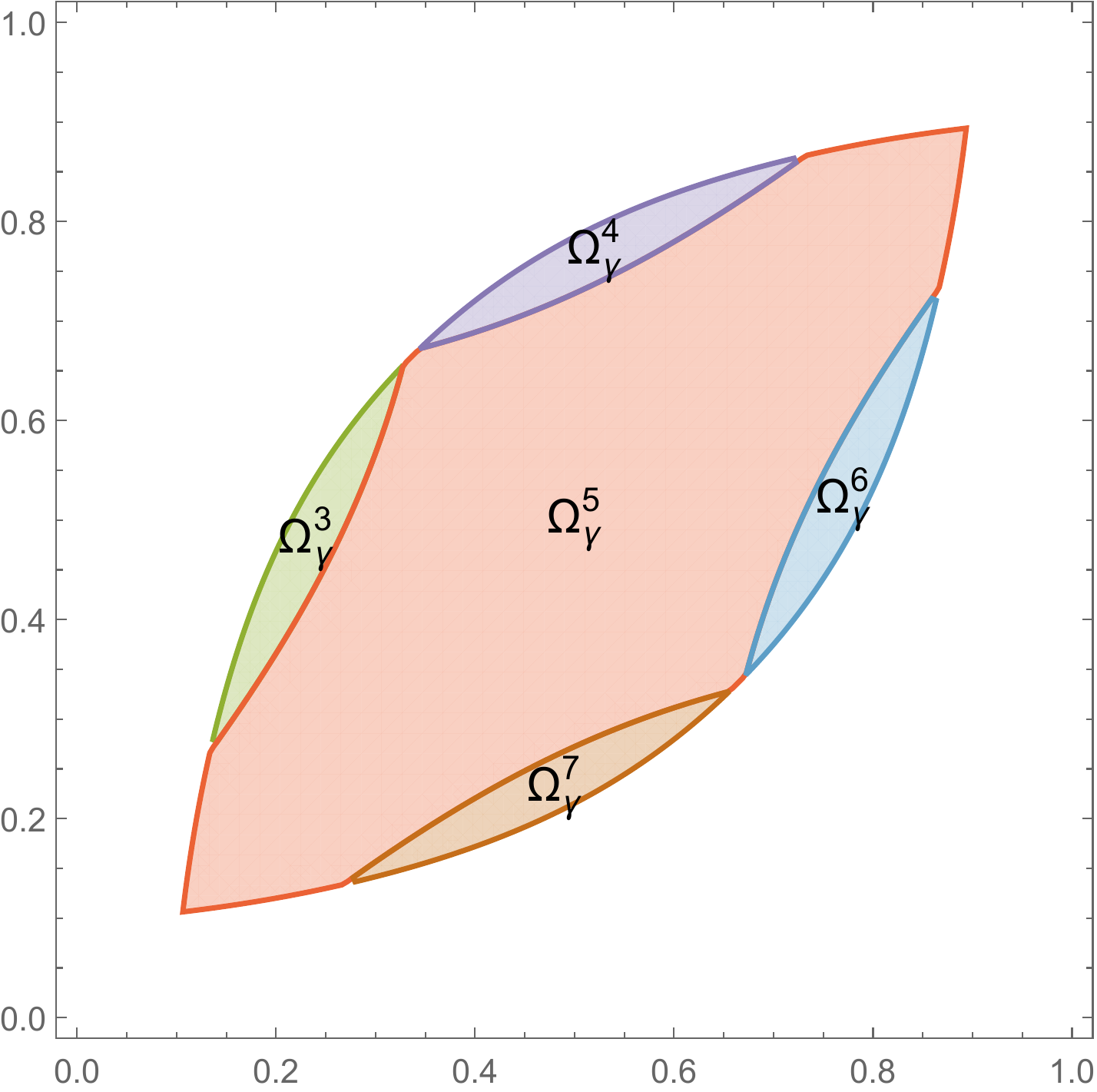} \hfil \includegraphics[width=5cm]{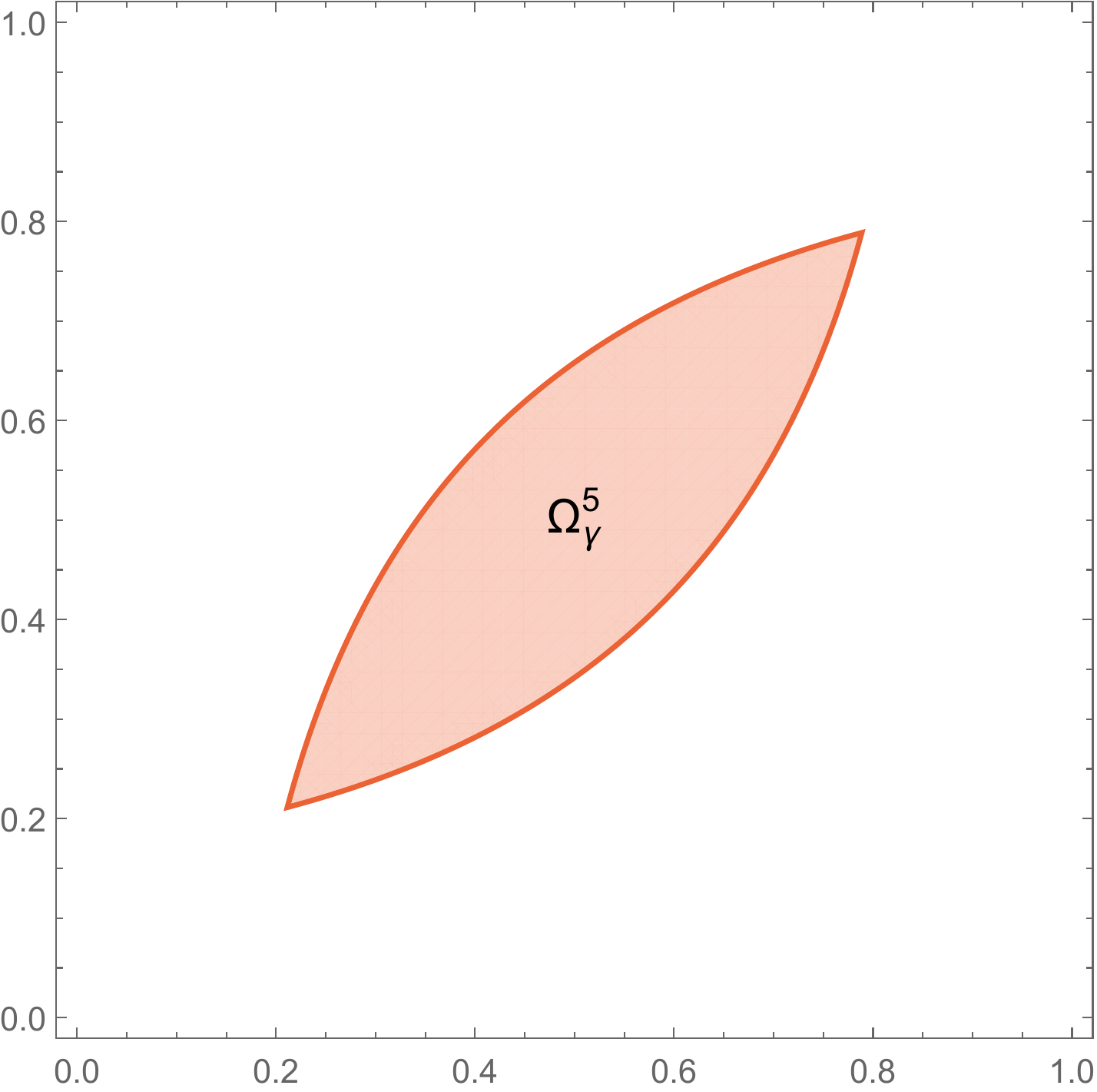}
            \caption{ Regionplots of the areas $\Omega_\gamma^1, \ldots, \Omega_\gamma^9$ for  $\gamma = -1, -\frac{24}{25},-\frac45, -\frac{7}{10}, -\frac{43}{100}, 0$.} \label{fig obmocja gamma}
\end{figure}

We also define functions of $\gamma$ depending on $(a,b)$ on these areas. Note that for a fixed value of $(a,b)$ they are inverses of $\underline{g}_{a,b}$. They are
\begin{align*}\label{omega_ab^i}
\omega_{a,b}^1(\gamma)&=\frac12 \left(a+b-1+\sqrt{(a+b-1)^2+1+\gamma}\right),\\
\omega_{a,b}^2(\gamma)&=\frac14 \left(a+3b-2+\sqrt{(a+b-1)^2+(1-2a)(1-2b)+2(1+\gamma)}\right),\\
\omega_{a,b}^3(\gamma)&=\frac17 \left(2a+4b-3+\sqrt{(2a+4b-3)^2+7(1+\gamma)}\right),\\
\omega_{a,b}^4(\gamma)&=\frac17 \left(3a+5b-4+\sqrt{(4a+2b-3)^2+7(1+\gamma)}\right),\\
\omega_{a,b}^5(\gamma)&=\frac12 \left(a+b-1+\frac{\sqrt{3}}{3}\sqrt{5(a+b-1)^2-2(1-2a)(1-2b)+2(1+\gamma)}\right),\\
\omega_{a,b}^6(\gamma)&=\frac17 \left(5a+3b-4+\sqrt{(2a+4b-3)^2+7(1+\gamma)}\right),\\
\omega_{a,b}^7(\gamma)&=\frac17 \left(4a+2b-3+\sqrt{(4a+2b-3)^2+7(1+\gamma)}\right),\\
\omega_{a,b}^8(\gamma)&=\frac14 \left(3a+b-2+\sqrt{(a+b-1)^2+(1-2a)(1-2b)+2(1+\gamma)}\right).
\end{align*}

We are now ready to state one of our main results.

\begin{theorem}\label{thm_gamma_upp}
The pointwise supremum $\overline{G}_{\gamma}$ 
of $\GG_{\gamma}$ for any $\gamma\in[-1, 1]$ and for any $(a,b)\in\II^2$ 
is given by
\begin{equation}\label{G_upper}
\overline{G}_{\gamma}(a,b)=\left\{ \begin{array}{ll}
        \omega_{a,b}^1(\gamma);      & \text{if } (a,b) \in \Omega_{\gamma}^1 \cup \Omega_{\gamma}^9 , \vspace{1mm}\\
        \omega_{a,b}^2(\gamma);      & \text{if } (a,b) \in \Omega_{\gamma}^2 , \vspace{1mm}\\
				\omega_{a,b}^3(\gamma);      & \text{if } (a,b) \in \Omega_{\gamma}^3 , \vspace{1mm}\\
				\omega_{a,b}^4(\gamma);      & \text{if } (a,b) \in \Omega_{\gamma}^4 , \vspace{1mm}\\
				\omega_{a,b}^5(\gamma);      & \text{if } (a,b) \in \Omega_{\gamma}^5 , \vspace{1mm}\\
				\omega_{a,b}^6(\gamma);      & \text{if } (a,b) \in \Omega_{\gamma}^6 , \vspace{1mm}\\
				\omega_{a,b}^7(\gamma);      & \text{if } (a,b) \in \Omega_{\gamma}^7 , \vspace{1mm}\\
				\omega_{a,b}^8(\gamma);      & \text{if } (a,b) \in \Omega_{\gamma}^8 , \vspace{1mm}\\
        M(a,b);       & \text{otherwise. }
       \end{array} \right.
\end{equation}
\end{theorem}

Before we give the proof we gather some observations:

\begin{corollary}\label{cor gamma1}
Suppose that $\overline{G}_{\gamma}$ is the supremum given in Theorem \ref{thm_gamma_upp}. Then:
\begin{enumerate}[(i)]
 \item We have $\overline{G}_{-1}=W$ and $\overline{G}_{\gamma}=M$ for $\gamma \in [\frac12, 1]$.
 \item For any $\gamma \in (-1, 0)$ the supremum $\overline{G}_{\gamma}$ is not a copula, but a proper quasi-copula.
 \item For any $\gamma\in [0,\frac12)$ the supremum $\overline{G}_{\gamma}$ is a copula  that is different from Fr\'echet-Hoeffding lower and upper bounds $W$ and $M$. It is singular. Its absolutely continuous part is distributed inside the bounded region enclosed by the graphs of hyperbolas $\omega^5_{a,b}=a$ and $\omega^5_{a,b}=b$ (as functions of $a$ and $b$). Its singular  component is distributed on the boundary of the region and on the two segments of the diagonal $a=b$ outside the region. (See Figure \ref{g zgoraj1}.)
 \item $\overline{G}_{\gamma}$ is increasing in $\gamma$ (in the concordance order on quasi-copulas).
 \item $\overline{G}_{\gamma}$ is symmetric and radially symmetric. 
 \item If we extend the measure of concordance $\gamma$ to any quasi-copula $Q$ by defining \begin{equation}\label{ext_gam}
     \gamma(Q)=4\int_0^1 \left(Q(t,t)+Q(t,1-t)\right) dt - 2
     \end{equation}
     then 
     $\gamma\left(\overline{G}_{\gamma}\right)>\gamma$ for all $\gamma\in(-1,1)$. (See Figure \ref{gamma(gamma)}.)
 \end{enumerate}
\end{corollary}

\begin{proof} First notice that $\overline{G}_{\gamma}$ equals the Fr\'echet-Hoeffding upper bound $M$ for any $\gamma \in [\frac12, 1]$, since then all the regions $\Omega_{\gamma}^i$ disappear. If $\gamma = -1$, then $\overline{G}_{-1}=W$. For any $\gamma \in (-1, 0)$ the point
$$\left(\frac12 \left(1- \frac{\sqrt{3}}{3} \sqrt{1-2 \gamma} \right), \frac12 \left(1- \frac{\sqrt{3}}{3} \sqrt{1-2 \gamma} \right)\right)$$
lies in the lower left corner of the area $\Omega_{\gamma}^5$ and the mixed second derivative $\frac{\partial^2 \omega_{a,b}^5(\gamma)}{\partial a \partial b}$ at this point has its value equal to $\frac{\gamma}{3}<0$, so the mixed second derivative is negative also inside the area $\Omega_{\gamma}^5$ due to continuity. Therefore, \cite[Theorem 2.1]{DuSe} implies that $\overline{G}_{\gamma}$ is {not} a copula for $\gamma \in (-1, 0)$. So, \textit{(i)} and \textit{(ii)} hold.

Next, assume that $\gamma\in[0,\frac12)$. Observe that only the area $\Omega^5_{\gamma}$ has nonempty interior for such $\gamma$ and so
$$ \overline{G}_{\gamma}(a,b)=\left\{ \begin{array}{ll}
        \omega_{a,b}^5(\gamma)& \text{if } (a,b) \in \Omega_{\gamma}^5 , \vspace{1mm}\\
        M(a,b);       & \text{otherwise. }
       \end{array} \right.$$
Also, note that $\Omega^5_{\gamma}$ is the area bounded by the graphs of hyperbolas $\omega^5_{a,b}=a$ and $\omega^5_{a,b}=b$. To prove that $\overline{G}_{\gamma}$  is a copula we use \cite[Theorem 2.1]{DuJa}. (See \cite{DuJa} for the definitions of Dini's derivatives as well.) For fixed $b$ the righthand side upper Dini derivative of $\overline{G}_{\gamma}(a,b)$ is
\begin{equation}\label{GG_lower-a}
D^+\overline{G}_{\gamma}(a,b)=\left\{ \begin{array}{ll}
        0; & \text{if } 0\leqslant a\leqslant \min\{b,\omega^5_{a,b}(\gamma)\}, \vspace{1mm}\\
       \frac12 \left(1+\frac{\sqrt{3}\left(5a+b-3\right)}{3\sqrt{5(a+b-1)^2-2(1-2a)(1-2b)+2(1+\gamma)}}\right);      & \text{if } a \text{ is such that } (a,b) \in \Omega_{\gamma}^5, \vspace{1mm}\\
        1;            & \text{otherwise. }
        \end{array} \right.
\end{equation}
For $a\in\II$ such that $(a,b)$ is in the interior of $\Omega_{\gamma}^5$, we have
\begin{equation}\label{GG_lower-ab}
\frac{\partial^2}{\partial a\partial b}\overline{G}_{\gamma}(a,b)=
        \frac{\sqrt{3}(6a+6b-12ab-2+\gamma)}{{3}\left(5(a+b-1)^2-2(1-2a)(1-2b)+2(1+\gamma)\right)^{\frac32}}.
\end{equation}
The diagonal points of $\Omega_{\gamma}^5$ have coordinates $\frac12\pm \frac{\sqrt{3}}{6}\sqrt{1-2\gamma}$. The second derivative in \eqref{GG_lower-ab} is positive on the square $[\frac12- \frac{\sqrt{3}}{6}\sqrt{1-2\gamma},\frac12+ \frac{\sqrt{3}}{6}\sqrt{1-2\gamma}]^2$ since the value of $6a+6b-12ab-2+\gamma$ is positive there and $\Omega_{\gamma}^5$ is contained in the square. The Dini derivative \eqref{GG_lower-a} has a positive jump at points on the boundary of the region that is enclosed by the graphs of hyperbolas $\omega^5_{a,b}=a$ and $\omega^5_{a,b}=b$ and on the two segments of the diagonal $a=b$ outside the region. So, it follows that statement \textit{(iii)} holds.

A careful analysis that we omit shows that $\overline{G}_{\gamma}$ is an increasing function of $\gamma$, as \textit{(iv)} asserts. Statement \textit{(v)} is a consequence of Lemma \ref{lem:symm}.

A technical calculation shows that \textit{(vi)} holds (see Figure \ref{gamma(gamma)} that was drawn using the Mathematica software \cite{Mathematica}).
\end{proof}

In  Figure \ref{g zgoraj} we give the 3D plot of the quasicopulas $\overline{G}_{\gamma}$ for $\gamma = -\frac78, -\frac35, -\frac{2}{13}$. In Figure \ref{g zgoraj1} we give the scatterplots of the copulas $\overline{G}_{\gamma}$ for $\gamma = 0, \frac14$.

\begin{figure}[h]
            \includegraphics[width=6cm]{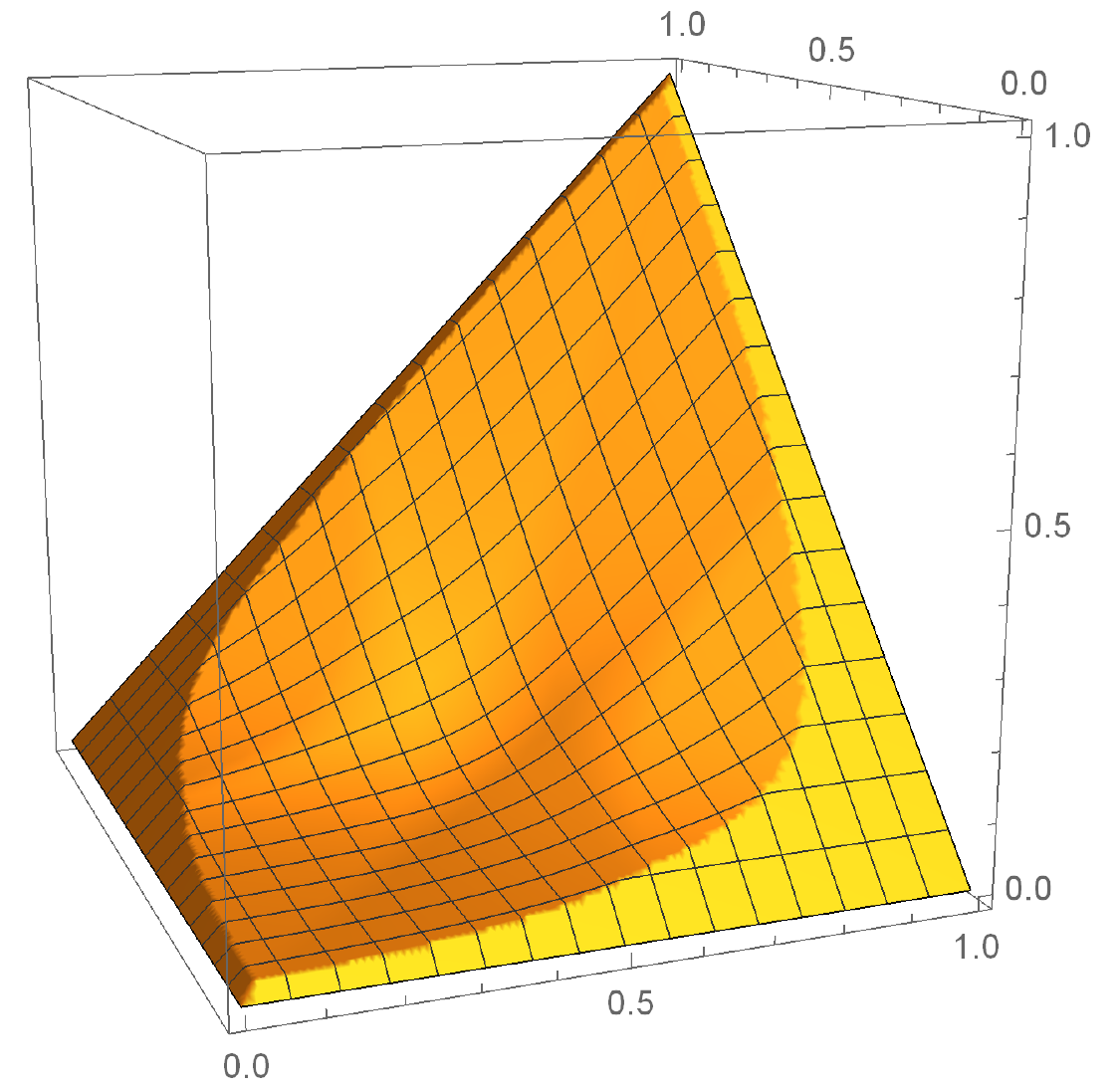}
            \includegraphics[width=5.3cm]{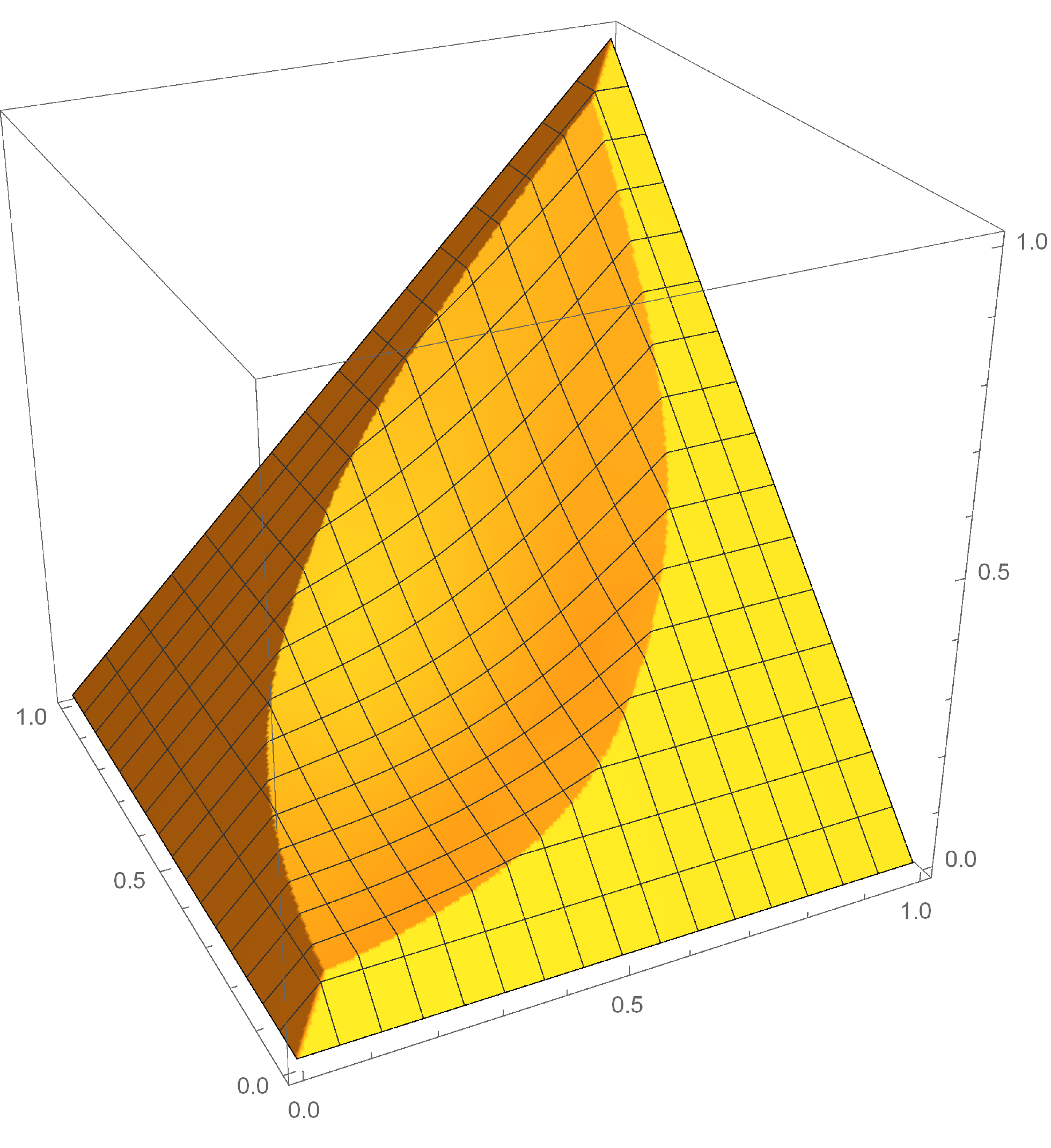}
            \includegraphics[width=6cm]{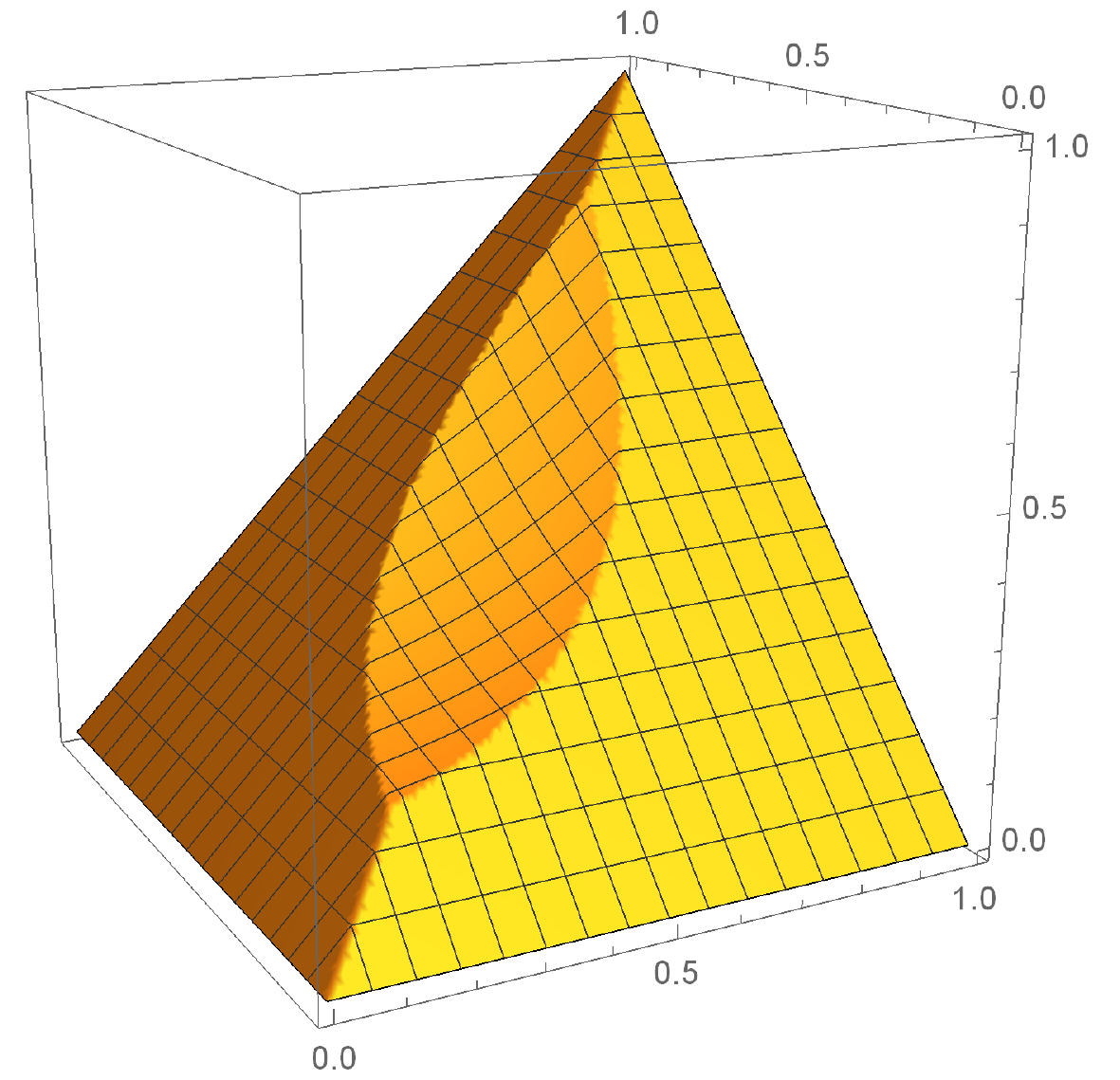}
            \caption{Graphs of quasicopulas $\overline{G}_{\gamma}$ for $\gamma = -\frac78, -\frac35, -\frac{2}{13}$} \label{g zgoraj}
\end{figure}

\begin{figure}[h]
            \includegraphics[width=6cm]{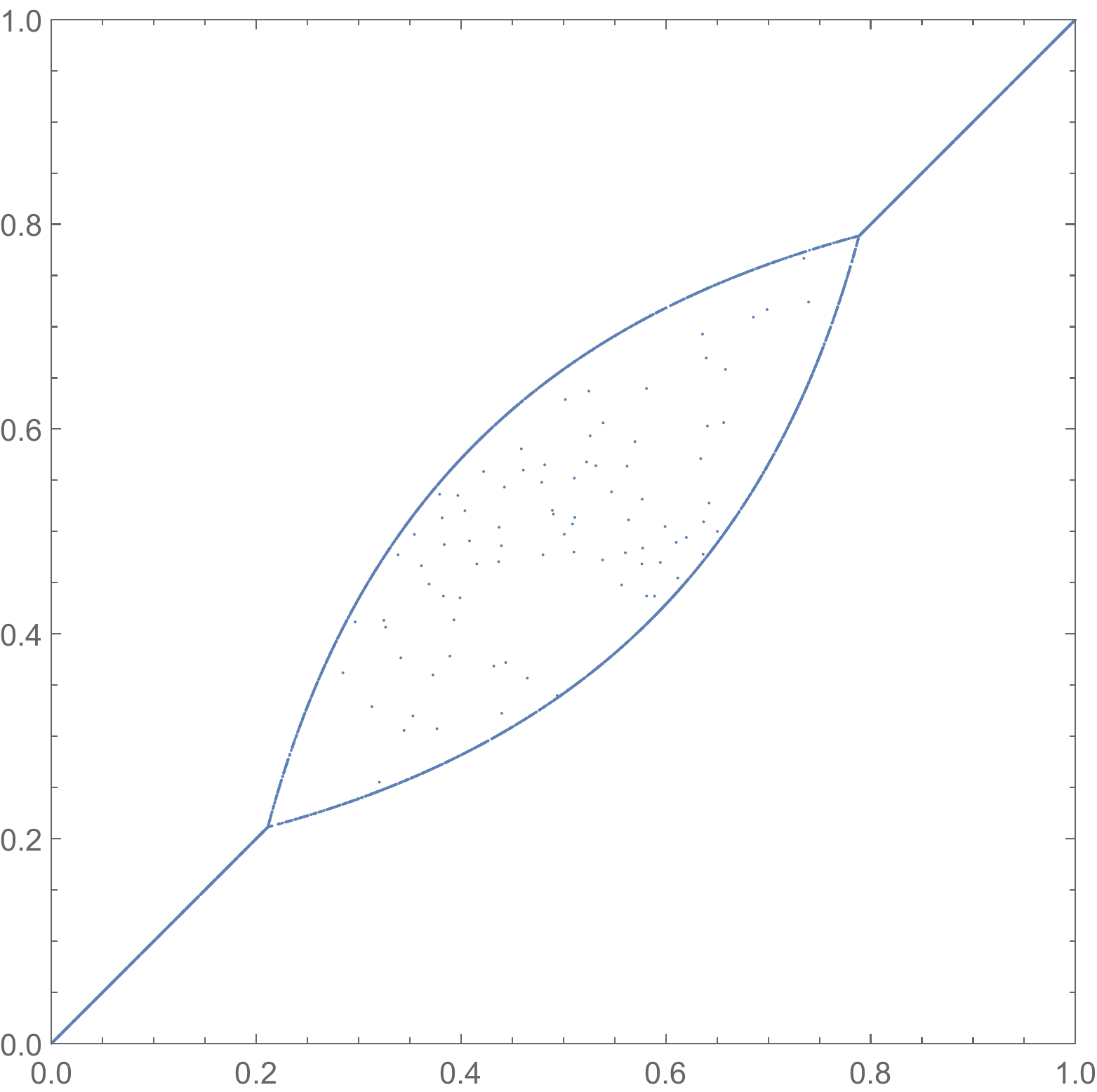} \hfil  \includegraphics[width=6cm]{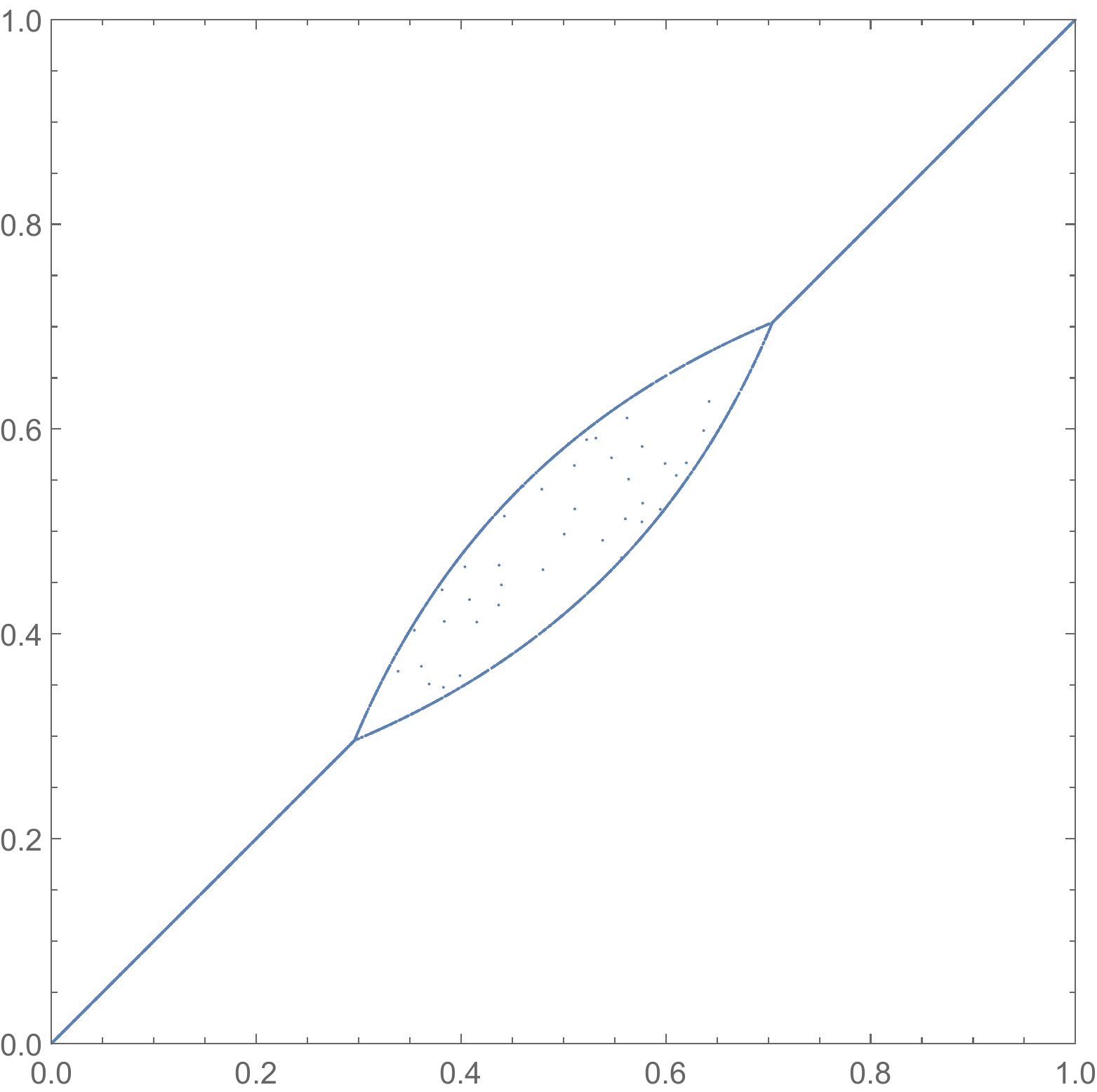}
            \caption{  Scatterplots of copulas $\overline{G}_{\gamma}$ for $\gamma = 0, \frac14$} \label{g zgoraj1}
\end{figure}

\begin{proof}[Proof of Theorem \ref{thm_gamma_upp}]
Due to Lemma \ref{lem:symm}~(a) we may assume that the point $(a,b)$ lies in the triangle $\Delta= \{ (a,b) \in \II^2; \, a \leqslant b, \, a+b \leqslant 1 \}$.
We use Proposition \ref{prop1} to show that for $(a,b)\in\Delta$ we have
\begin{equation}\label{g underline gamma}
  \begin{split}
       \underline{g}_{a,b}(d)&=\gamma\left(\underline{C}^{(a,b)}_{d-W(a,b)}\right)=\cQ\left(M,\underline{C}^{(a,b)}_{c_1}\right)+\cQ\left(W,\underline{C}^{(a,b)}_{c_1}\right)= \\
       &=\left\{ \begin{array}{ll}
        g_{a,b}^1(d);           & \text{if } b \geqslant d + \frac12, \vspace{1mm}\\
        g_{a,b}^2(d);           & \text{if } \frac12(1+d) \leqslant b \leqslant d + \frac12, \vspace{1mm}\\
        g_{a,b}^3(d);           & \text{if } a+d \leqslant b \leqslant \frac12(1+d), \vspace{1mm}\\
        g_{a,b}^5(d);           & \text{if } b \leqslant a+d.
                \end{array} \right. ,
	\end{split}
\end{equation}
where
\begin{align*}
g_{a,b}^1(d)&=4d \left(1-a-b+d\right)-1,\\
g_{a,b}^2(d)&=\left(1-2b+2d\right)^2+4d (1-a-b+d)-1,\\
g_{a,b}^3(d)&= d\left(6-4a-8b+7d\right)-1, \\
g_{a,b}^5(d)&= 6d\left(1-a-b+d\right)- \left(a-b\right)^2-1.
\end{align*}

\begin{figure}[h]
            \includegraphics[width=6cm]{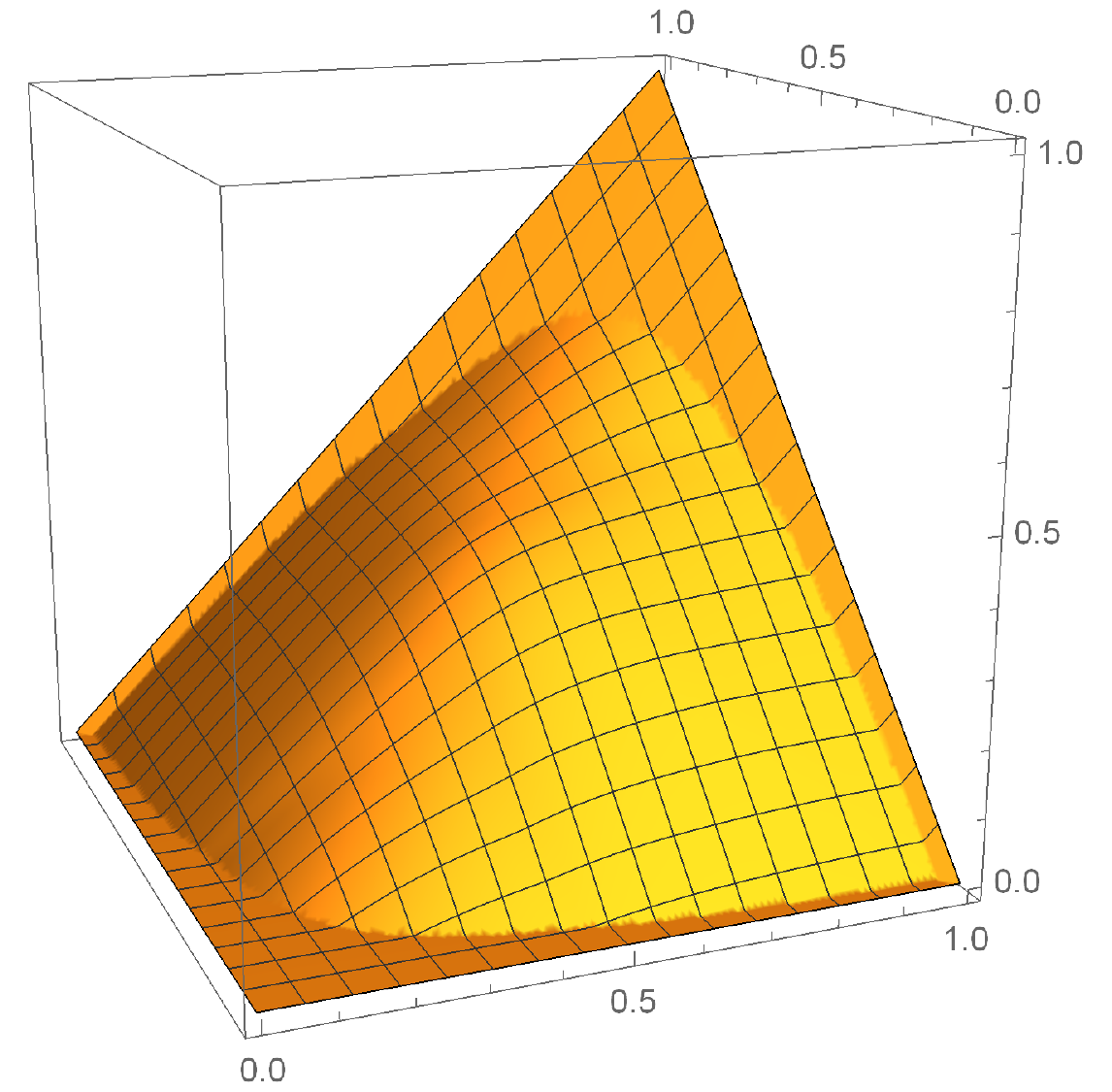}
            \includegraphics[width=5.3cm]{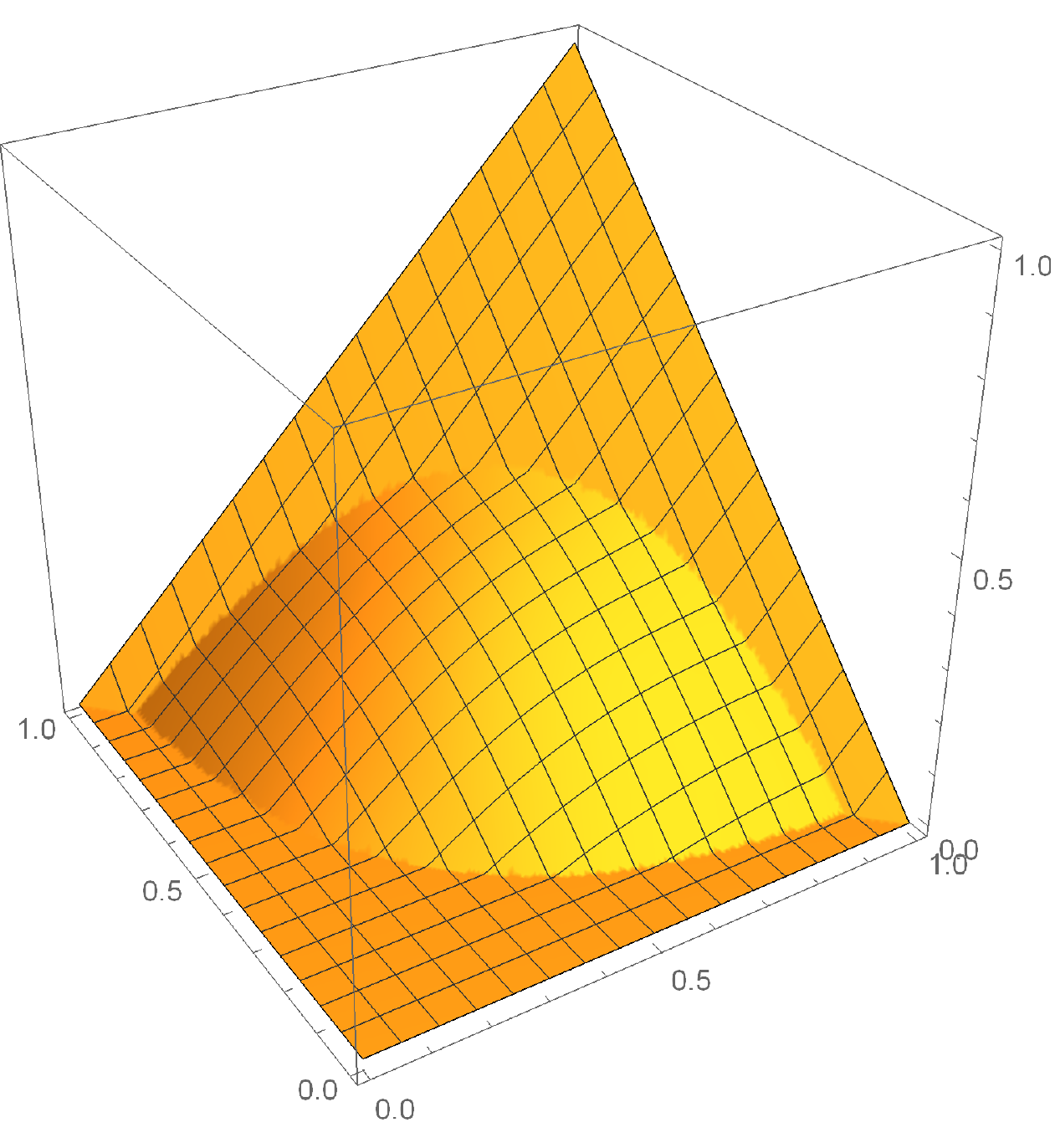}
            \includegraphics[width=6cm]{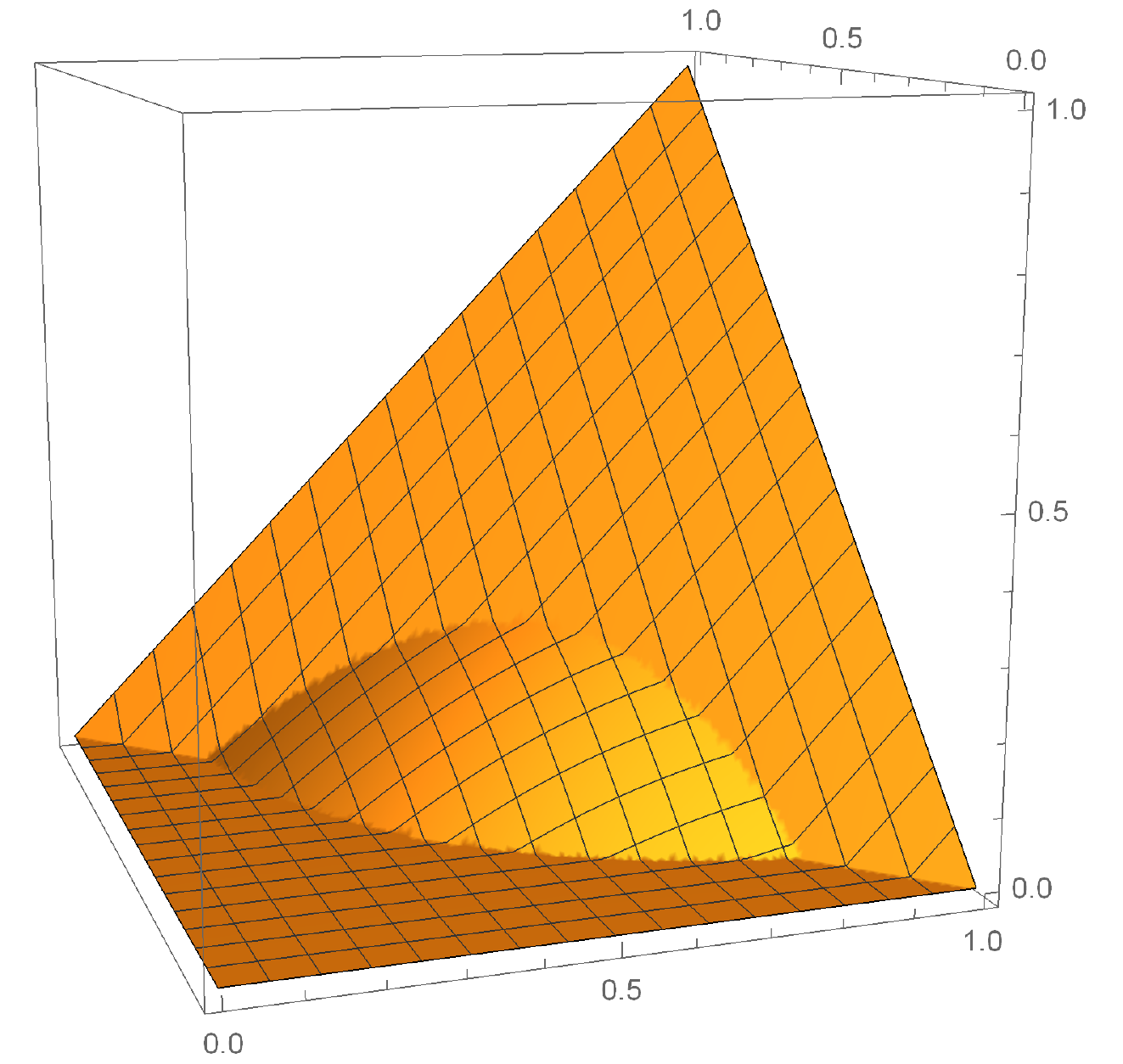}
            \caption{Graphs of quasicopulas $\underline{G}_{\gamma}$ for $\gamma = \frac{2}{13}, \frac35, \frac78$.} \label{g spodaj}
\end{figure}

Since $d=\underline{C}^{(a,b)}_{c_1}(a,b)$
it follows that $W(a,b)=0\leqslant d\leqslant M(a,b)=a$. For such values of $d$ the expression on the right-hand side of \eqref{g underline gamma} is increasing in $d$ and thus, the  maximal possible value of $\underline{g}_{a,b}(d)$ is achieved when $d=a$. Then, we have
$$\underline{g}_{a,b}(a)=\left\{ \begin{array}{ll}
        4a(1-b)-1;                             & \text{if } b \geqslant a + \frac12, \vspace{1mm}\\
        (1+2a-2b)^2+4a(1-b)-1;                     & \text{if } \frac12(1+a) \leqslant b \leqslant a + \frac12, \vspace{1mm}\\
        a(6+3a-8b)-1;		 & \text{if } 2a \leqslant b \leqslant \frac12(1+a), \vspace{1mm}\\
        6a(1-b)-(a-b)^2-1;     & \text{if } b \leqslant 2a.
                \end{array} \right. $$
For the function $\underline{g}_{a,b}:[0, a]\to [-1, \underline{g}_{a,b}(a)]$ we need to find its inverse.
If for a given value $\gamma \in \left[ -1, 1 \right]$ it holds that $\gamma \geqslant \underline{g}_{a,b}(a)$, we take $d=a=M(a,b)$.
Otherwise, we take the inverses of the expressions for $g_{a,b}^i$ which are $\omega_{a,b}^i(\phi)$ for $i=1, 2, 3, 5$.
The inequality $\gamma \geqslant \underline{g}_{a,b}(a)$ gives us the area
\begin{align*}
\bigg\{ (a,b)\in \Delta; \,
& \left( b \geqslant 1- \frac{1+\gamma}{4a}  \textrm{ and } b \geqslant a+\frac12 \right) \textrm { or } \\
& \left( (1+2a-2b)^2+4a(1-b) \leqslant 1+\gamma \textrm{ and } \frac12(1+a) \leqslant b \leqslant a + \frac12 \right) \textrm{ or } \\
& \left( b \geqslant \frac18 \left( 3a+6- \frac{1+\gamma}{a} \right)  \textrm{ and } 2a \leqslant b \leqslant \frac12(1+a) \right) \textrm{ or } \\
& \left( b \geqslant -2a+\sqrt{3a(a+2)-(1+\gamma)} \textrm{ and } b\leqslant 2a \right)
\bigg\}
\end{align*}
which is equal to the area $\Delta \setminus (\Omega^1_\gamma \cup \Omega^2_\gamma\cup \Omega^3_\gamma\cup \Omega^5_\gamma)$. We continue by considering the inequalities $\gamma \geqslant \underline{g}^i_{a,b}(a)$ for $i=1,2,3,5$. Their careful consideration yields areas where each of the expressions $\omega_{a,b}^i(\gamma)$ is valid. These are the areas $\Omega^i_\gamma \cap \Delta$ for $i=1,2,3,5$. Now, we reflect the expressions $\omega_{a,b}^1(\gamma),  \omega_{a,b}^2(\gamma),\omega_{a,b}^3(\gamma),\omega_{a,b}^5(\gamma) $ over the main diagonal and over the counter-diagonal to obtain the expressions $ \omega_{a,b}^1(\gamma), \ldots,  \omega_{a,b}^8(\gamma)$. The areas where they are valid are the reflections of the areas $\Omega^1_\gamma\cap \Delta, \Omega^2_\gamma\cap \Delta, \Omega^3_\gamma\cap \Delta, \Omega^5_\gamma\cap \Delta$ over the main diagonal and over the counter-diagonal, i.e., the areas $\Omega^1_\gamma, \ldots, \Omega^9_\gamma$.

\begin{figure}[h]
            \includegraphics[width=6cm]{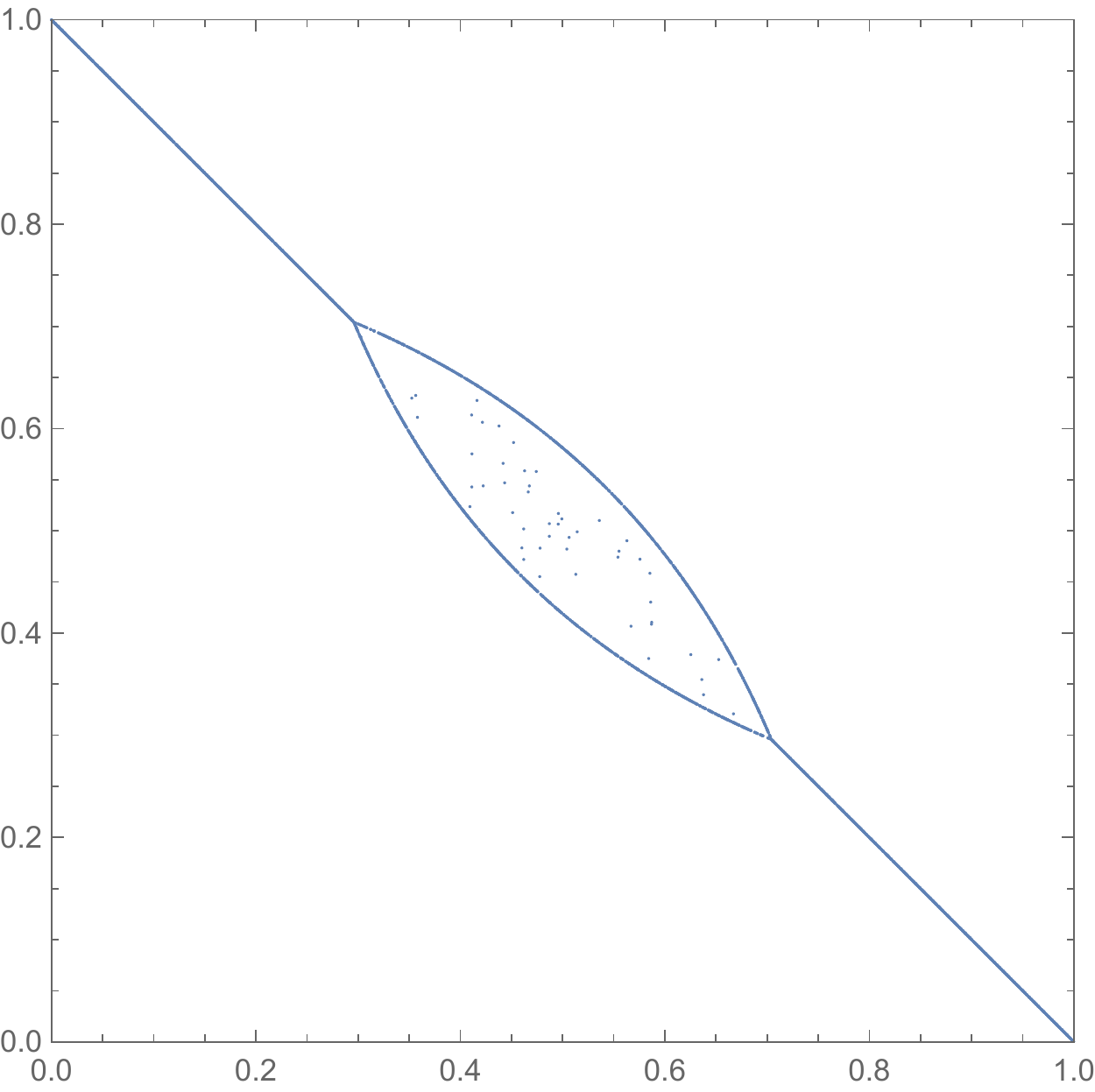} \hfil  \includegraphics[width=6cm]{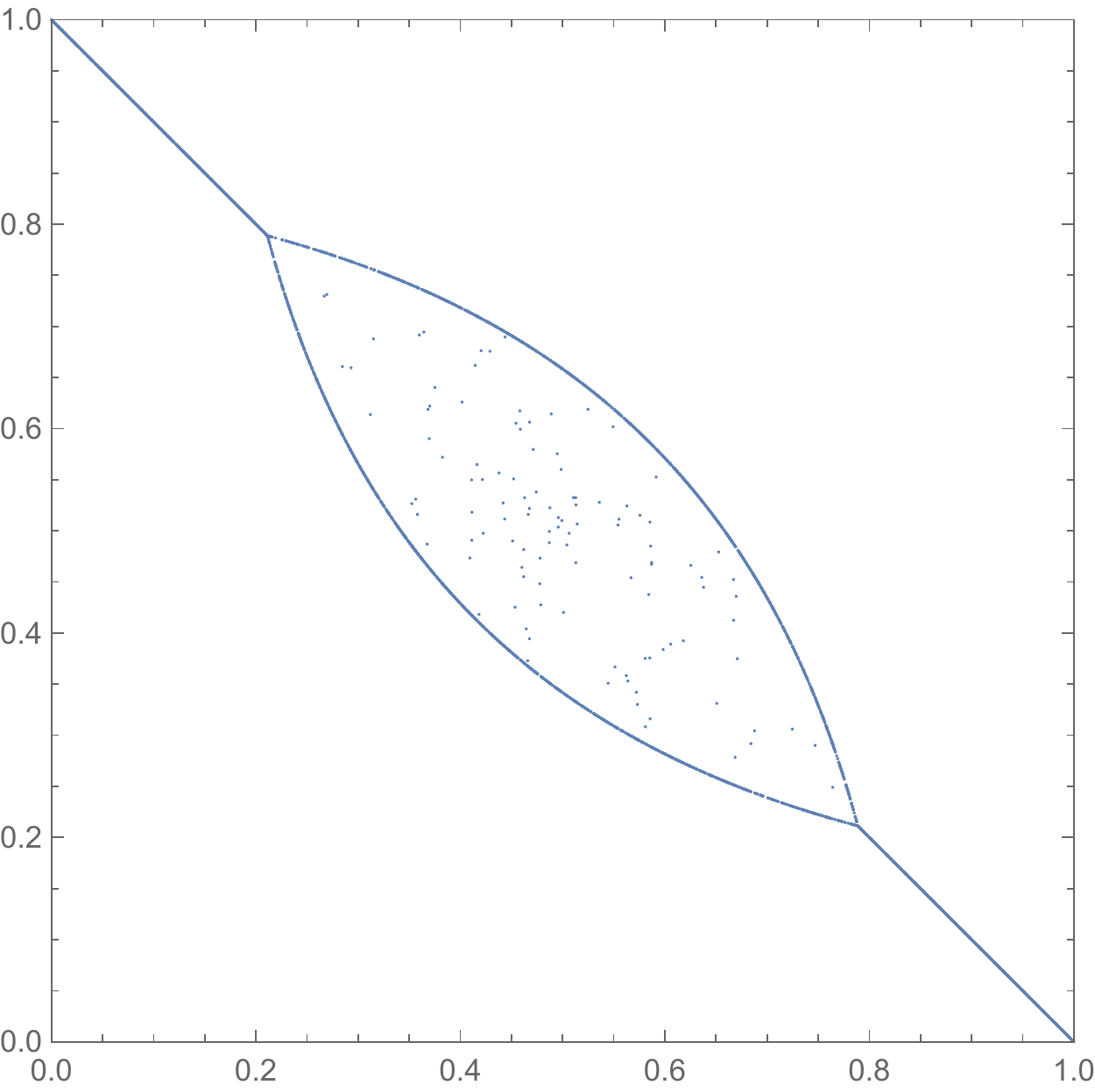}
            \caption{  Scatterplots of copulas $\underline{G}_{\gamma}$ for $\gamma = -\frac14$, $0$.} \label{g spodaj1}
\end{figure}

We conclude that the required upper bound is given by \eqref{G_upper}. The detailed calculations of the functions and their domains were done with a help of  using Wolfram Mathematica software \cite{Mathematica}.
\end{proof}

A direct consequence of Lemma \ref{lem:symm} (b) is the following corollary.

\begin{corollary}\label{thm_gamma_low}
The pointwise infimum $\underline{G}_{\gamma}$
of $\GG_{\gamma}$ for any $\gamma\in[-1, 1]$ is given by reflecting pointwise supremum $\overline{G}_{-\gamma}$ with respect to either variable, i.e.
$$\underline{G}_{\gamma}(a, b) = a- \overline{G}_{-\gamma}(a, 1-b) = b - \overline{G}_{-\gamma}(1-a, b).$$
\end{corollary}

\begin{figure}[h]
            \includegraphics[width=6cm]{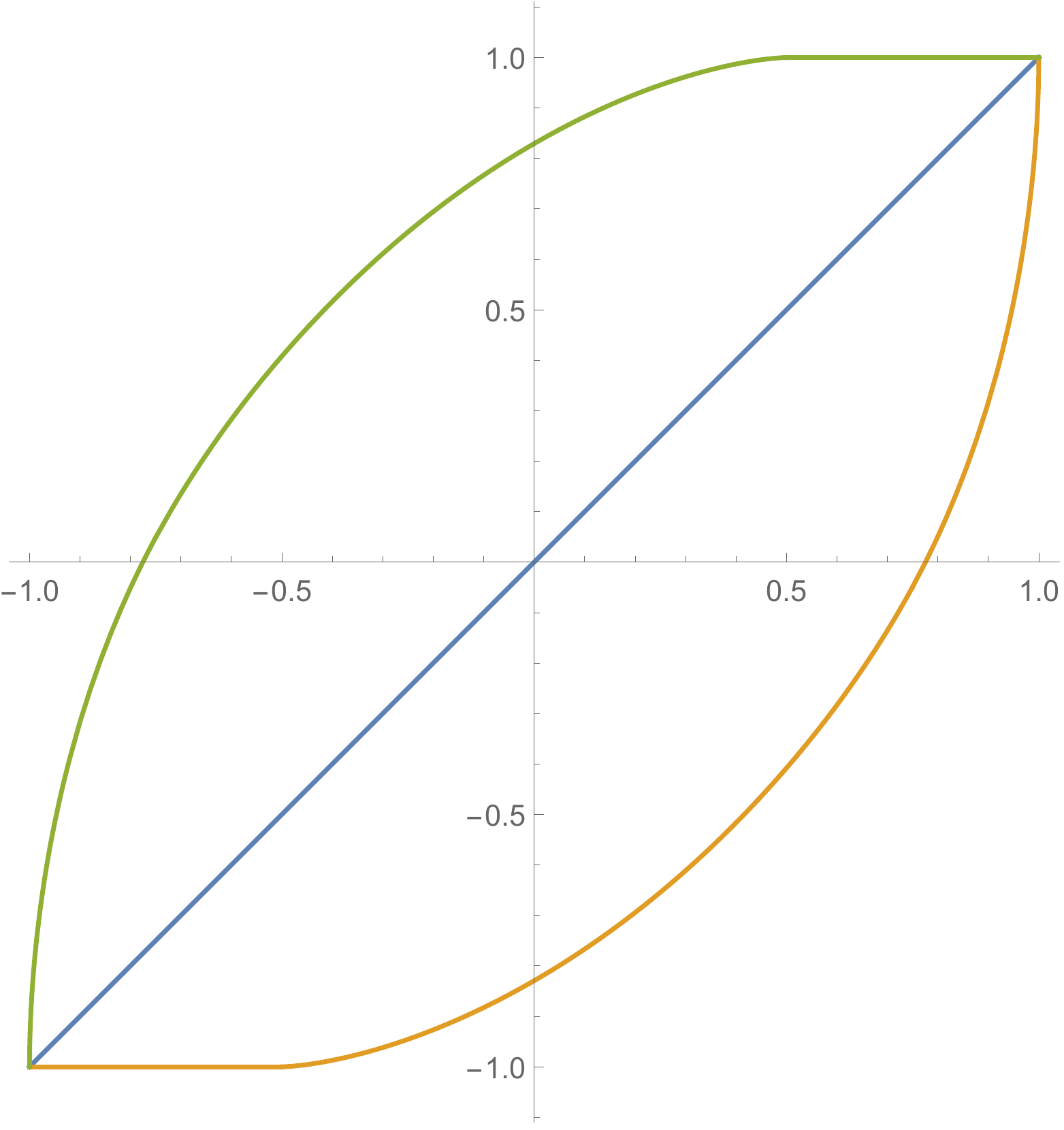}
            \caption{Graphs of values of $\gamma(\underline{G}_{\gamma})$ (orange) and $\gamma(\overline{G}_{\gamma})$ (green). } \label{gamma(gamma)}
\end{figure}

In  Figure \ref{g spodaj} we give the 3D plot of the quasicopulas $\overline{G}_{\gamma}$ for $\gamma = \frac{2}{13},  \frac35, \frac78,$ and in Figure \ref{g spodaj1} we give the scatterplots of the copulas $\overline{G}_{\gamma}$ for $\gamma = -\frac14, 0$.

Further observations then follow directly by Corollary \ref{cor gamma1}.

\begin{corollary}\label{cor gamma2}
Suppose that $\underline{G}_{\gamma}$ is the infimum given in Corollary \ref{thm_gamma_low}. Then:
\begin{enumerate}[(i)]
 \item We have $\underline{G}_{1}=M$ and $\underline{G}_{\gamma}=W$ for $\gamma \in [-1,-\frac12]$.
 \item For any $\gamma \in (0,1)$ the infimum $\underline{G}_{\gamma}$ is not a copula, but a proper quasi-copula.
 \item For any $\gamma\in (-\frac12,0]$ the infimum $\underline{G}_{\gamma}$ is a copula  that is different from Fr\'echet-Hoeffding lower and upper bounds $W$ and $M$. It is singular. Its absolutely continuous part is distributed inside the bounded region enclosed by the graphs of hyperbolas $\omega^5_{1-a,b}=1-a$ and $\omega^5_{a,1-b}=1-b$ (as functions of $a$ and $b$). Its singular  component is distributed on the boundary of the region and on the two segments of the anti-diagonal $a+b=1$ outside the region. (See Figure \ref{g spodaj1}.)
 \item $\underline{G}_{\gamma}$ is increasing in $\gamma$ (in the concordance order on quasi-copulas).
 \item $\underline{G}_{\gamma}$ is symmetric and radially symmetric. 
 \item If we extend the measure of concordance $\gamma$ to any quasi-copula $Q$ by \eqref{ext_gam}
     then 
     $\gamma\left(\underline{G}_{\gamma}\right)<\gamma$ for all $\gamma\in(-1,1)$. (See Figure \ref{gamma(gamma)}.)
 \end{enumerate}
\end{corollary}

\section{Comparison of local bounds}

In this section we give a comparison of effectiveness of bounds for Spearman's footrule and Gini's gamma and compare them with respect to the same bounds for Kendall's tau, Spearman's rho in Blomqvist's beta given by Nelsen and \'Ubeda-Flores in \cite{NeUbFl}.

Suppose that $\kappa:\mathcal{C}\to[-1,1]$ is a given measure of concordance, $k$ a fixed value in the range of $\kappa$ and that $\underline{K}_{k}$ {and} $\overline{K}_{k}$ are the lower and the upper bound of \eqref{eq:kappa}, respectively. Then the effectiveness of $\kappa$ is measured by the function \cite{NeUbFl}
\begin{equation}
m_{\kappa}(k)=1-6\int\int_{\II^2} \left|\overline{K}_{k}-\underline{K}_{k}\right|\,du\, dv.
\end{equation}
Here, the double integral represents the volume between the upper and lower local bound, and it is scaled so that $m_{\kappa}(k)=0$ means there is no improvement on the bound, i. e., $\overline{K}_{k}=M$ and $\underline{K}_{k}=W$, and $m_{\kappa}(k)=1$ means that the two bounds coincide. In the two tables below we give the values for $m_{\kappa}(k)$ for $\kappa=\phi$ and $\kappa=\gamma$. For Gini's gamma the effectiveness function $m_{\gamma}$ is an even function by Corollary \ref{thm_gamma_low}, so we give its values for $k\in[0,1]$ only. Their graphs are presented in Figure \ref{ff spodaj1}.
\vskip 15pt
\begin{center}
\begin{tabular}{|c|c|}
\hline
\parbox[c]{20mm}{\centering value of $k$} &
\parbox[c]{20mm}{\centering $m_{\phi}(k)$}  \\
\hline\hline
-0.5 & 0.7500\\
-0.4 &
  0.3718\\ -0.3&
  0.2244\\ -0.2&
  0.1352\\ -0.1&
  0.0820\\ 0.0&
  0.0574\\ 0.1&
  0.0569\\ 0.2&
  0.0763\\ 0.3&
  0.1108\\ 0.4&
  0.1562\\ 0.5&
  0.2146\\ 0.6&
  0.2895\\ 0.7&
  0.3860\\ 0.8&
  0.5130\\ 0.9&
  0.6889\\ 1.0&
  1.0000 \\
\hline
\end{tabular}
\quad\quad\quad
\begin{tabular}{|c|c|}
\hline
\parbox[c]{20mm}{\centering value of $k$} &
\parbox[c]{20mm}{\centering $m_{\gamma}(k)$} \\
\hline\hline
0.0& 0.0581\\
0.1& 0.0633\\
0.2& 0.0792\\
0.3& 0.1059 \\
0.4& 0.1438\\
0.5& 0.1942\\
0.6& 0.2587\\
0.7& 0.3422 \\
0.8& 0.4565\\
0.9& 0.6320\\
1.0& 1.0000\\
\hline
\end{tabular}
\vskip 10pt
{\sc Table 1.} Values of the effectiveness function for \\ Spearman's footrule (left) and Gini's gamma (right).
\end{center}

\vskip 15pt

A comparison of values in Table 1 with those given in \cite[Table 1]{NeUbFl} shows that Spearman's footrule and Gini's gamma have higher values of the effectiveness function, so their bounds are stricter, as compared to Spearman's rho and Kendall's tau. Blomqvist's beta, however, has highest values for all $k$ considered except for values of $k$ very close to $1$.

\vskip 15pt
\begin{figure}[h]
            \includegraphics[width=4.8cm]{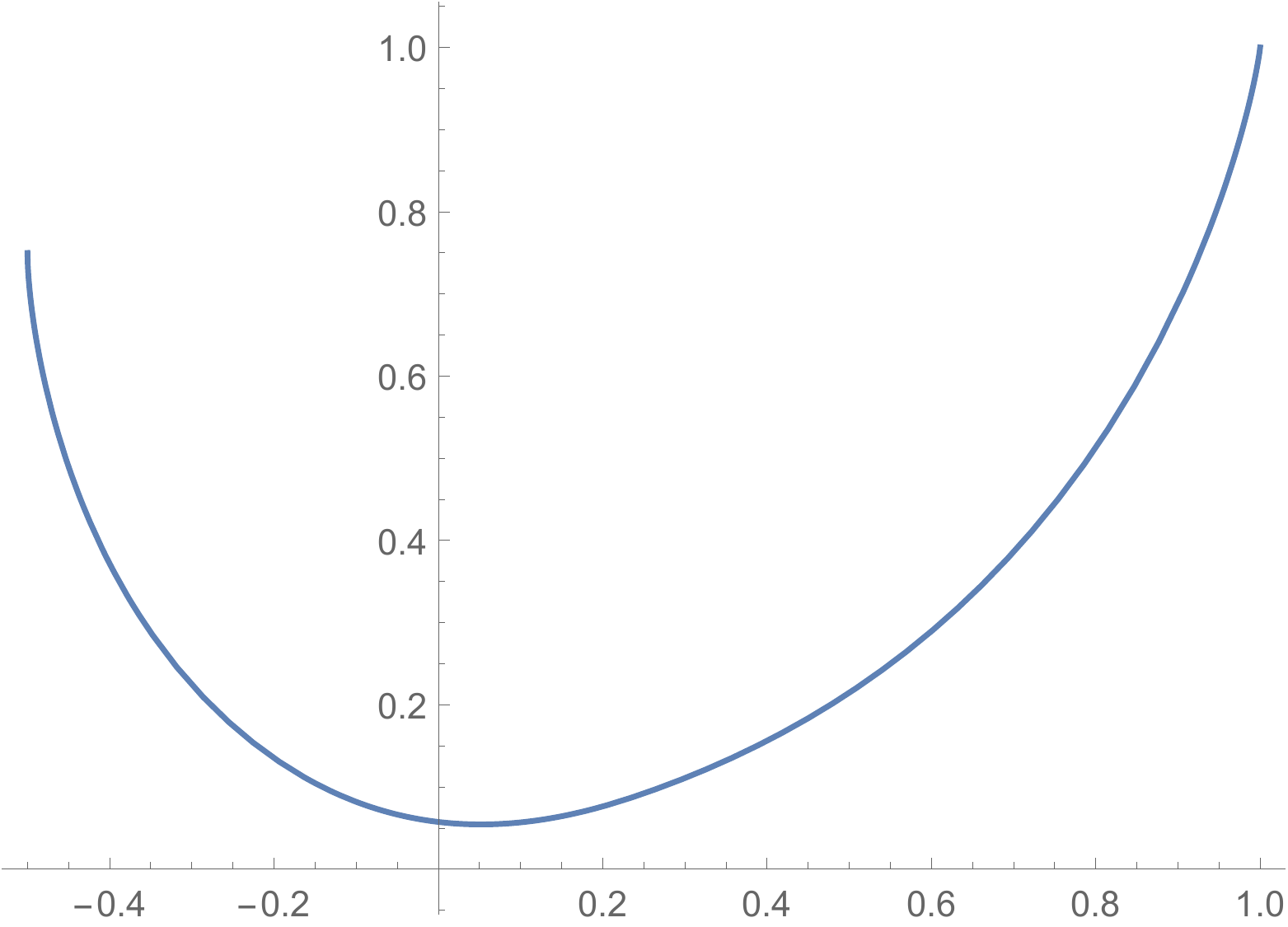} \hfil  \includegraphics[width=6cm]{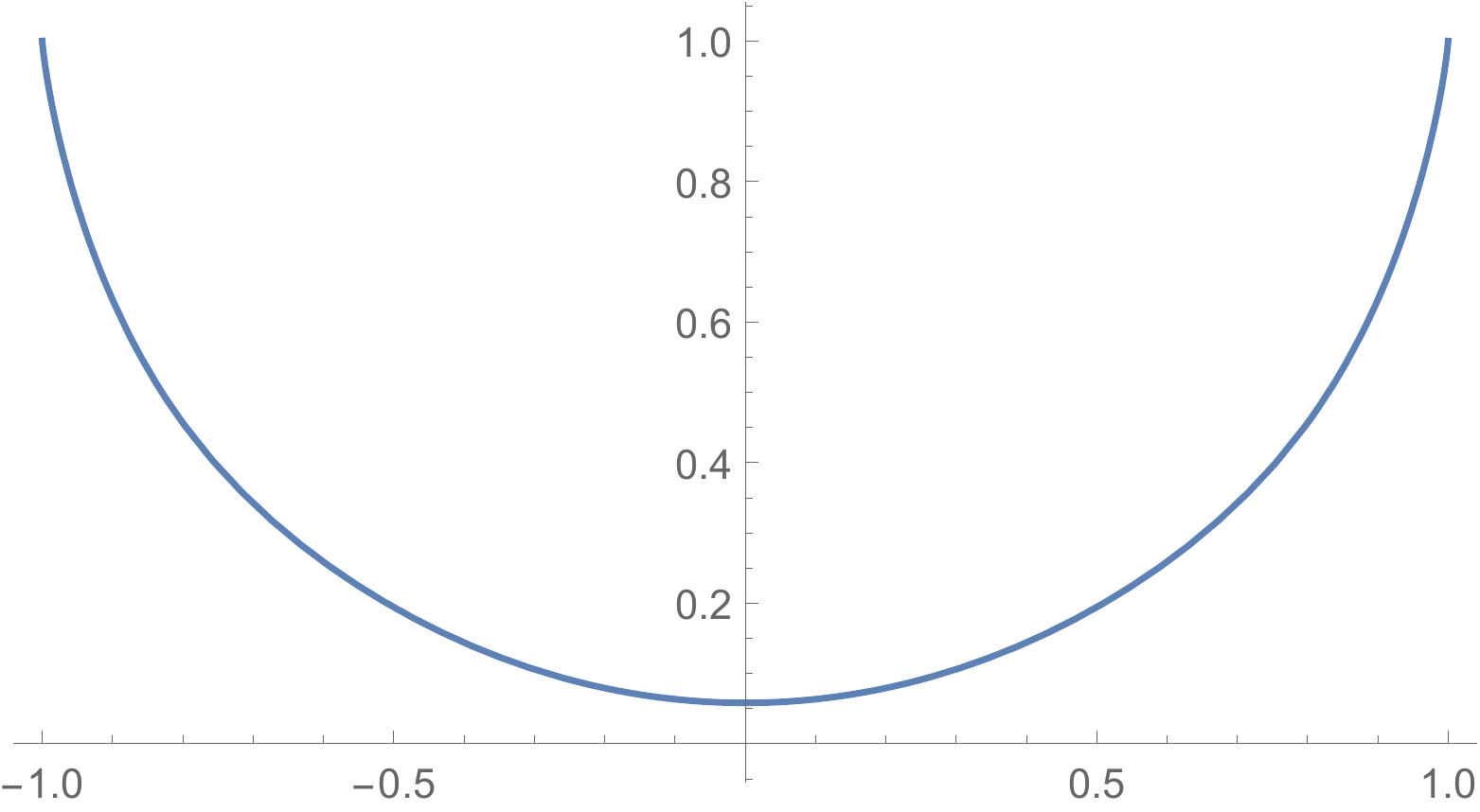}
            \caption{Graphs of effectiveness functions of Spearman's footrule (left) and Gini's gamma (right).} \label{ff spodaj1}
\end{figure}

\section{Spearman's footrule and Gini's gamma vs. Blomqvist's beta}\label{sec:beta}

In this section we give the exact regions of possible pairs of values $\left(\kappa(C),\beta(C)\right)$, $C\in \mC$, where $\kappa$ is either Spearman's footrule or Gini's gamma. Historically, a possible region of pairs of values of a pair of measure of concordance was studied first for Spearman's rho and Kendall's tau. In 1950s it was considered by Daniels, Durbin and Stuart, and Kruskal \cite{Dani,DuSt,Krus} (see also \cite[\S5.1.3]{Nels}), and later by other authors e. g. in \cite{Dani,DuSt,FrNe,GeNe2,Krus}. The exact region of possible pairs of values $(\rho(C),\tau(C))$ was given recently in \cite{ScPaTr}. For other pairs of measures of concordance we are aware only of the exact regions of possible pairs $\left(\kappa(C),\beta(C)\right)$ for $\kappa\in\{\rho,\tau,\gamma\}$ that are stated by Nelsen as Exercise 5.17 of \cite{Nels} with a hint of a proof. Our proof for the exact region between Gini's gamma and Blomqvist's beta uses results of our previous section and it is different from the one suggested in \cite{Nels}. The exact region for Spearman's footrule and Blomqvist's beta seems to be new.

\begin{theorem}\label{thm:phi} Let $C$ be any copula.
   \begin{enumerate}[(a)]
   \item If $\phi(C)=\phi$ for some $\phi\in\left[-\dfrac12,1\right]$, then
   \[
     1-2\sqrt{\frac{2}{3}(1-\phi)}\
     \leqslant \beta(C)\leqslant
     \begin{cases}       -1+2\sqrt{\frac23\left(1+2\phi\right)}, &
\mbox{if } -\dfrac12\leqslant\phi\leqslant \dfrac14\\       1, &
\mbox{if } \dfrac14\leqslant\phi\leqslant 1.     \end{cases}
   \]
     \item If $\beta(C)=\beta$ for some $\beta\in\left[-1, 1\right]$, then
   \[
     \dfrac{3(1+\beta)^2}{16}-\dfrac12\
     \leqslant \phi(C)\leqslant
     1-\dfrac{3(1-\beta)^2}{8} .
   \]
   \end{enumerate}
     The bounds are attained.
\end{theorem}

\begin{proof}
Suppose that $\phi(C)=\phi$. Then $\underline{F}_{\phi}(\frac12,
\frac12) \leqslant C(\frac12, \frac12)\leqslant
\overline{F}_{\phi}(\frac12, \frac12)$.
 From Theorem \ref{thm_gamma_low} we have $\underline{F}_{\phi}(\frac12,
\frac12) = \frac12\left(1-\sqrt{\frac{2}{3}(1-\phi)}\right)$. Since for
$\phi \leqslant \frac14$ the point $(\frac12, \frac12)$ lies in the area
$\Delta_{\phi}^4$ we have $\overline{F}_{\phi}(\frac12, \frac12) =
\delta_{\frac12, \frac12}^4(\phi) =
\frac12\sqrt{\frac23\left(1+2\phi\right)}$. For $\phi \geqslant \frac14$
we have $\overline{F}_{\phi}(\frac12, \frac12)= M(\frac12, \frac12) =
\frac12$. Now, $\beta(C) = 4C(\frac12, \frac12) - 1$ gives us point (a).
The bounds are attained by copulas $\overline{C}^{(\frac12,
\frac12)}_{c_2}$ and $\underline{C}^{(\frac12, \frac12)}_{c_1}$,
respectively.  Next, point (b) is obtained from (a) by inverting the
functions.
\end{proof}

In Figure \ref{fig3} we display the set of all possible  pairs
$(\phi(C), \beta(C))$ for a copula $C$. The expressions for the bounds
of the shaded regions are given in Theorem \ref{thm:phi}.

\begin{figure}[h]
             \includegraphics[width=5cm]{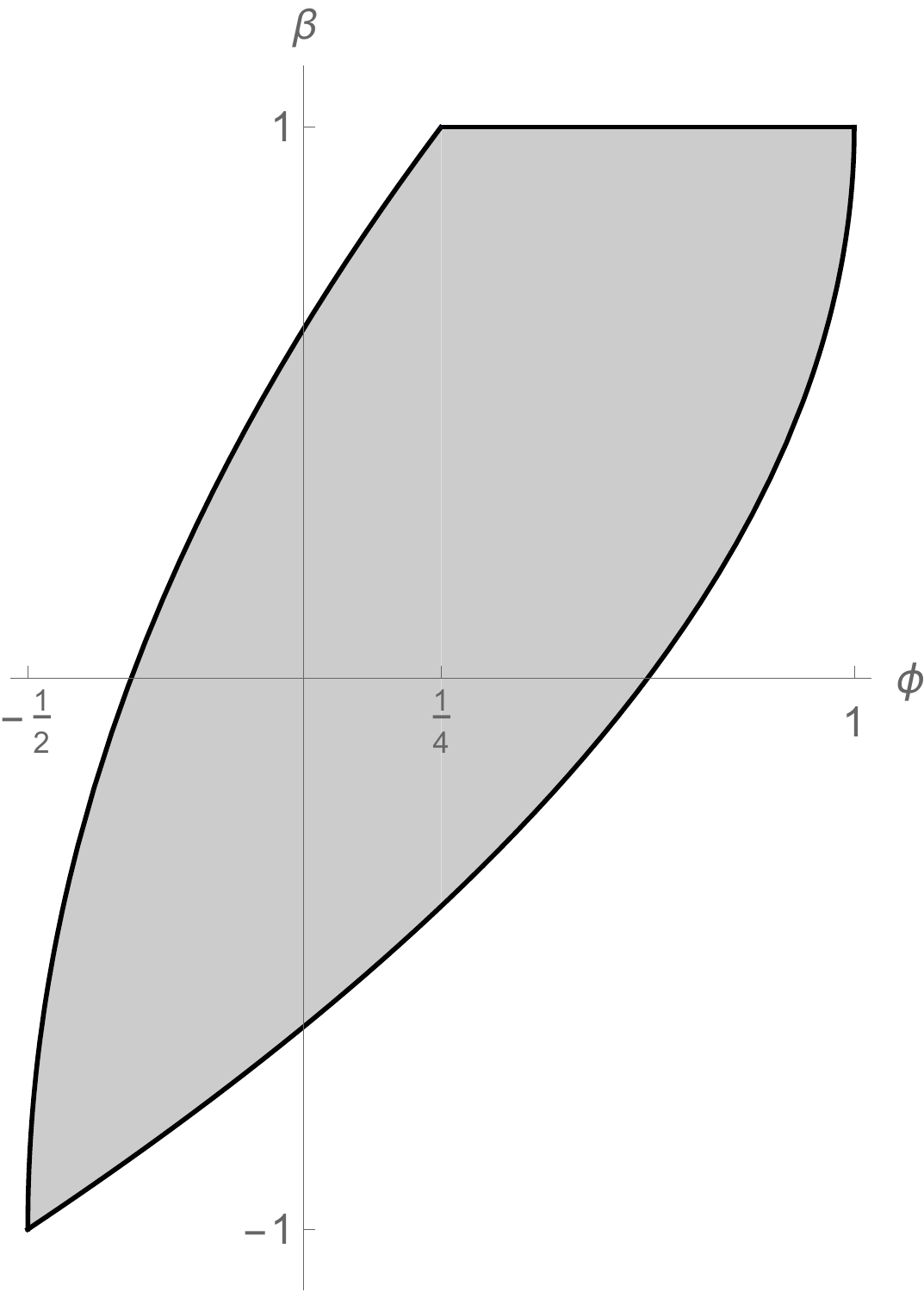} \hfil
\includegraphics[width=6.7cm]{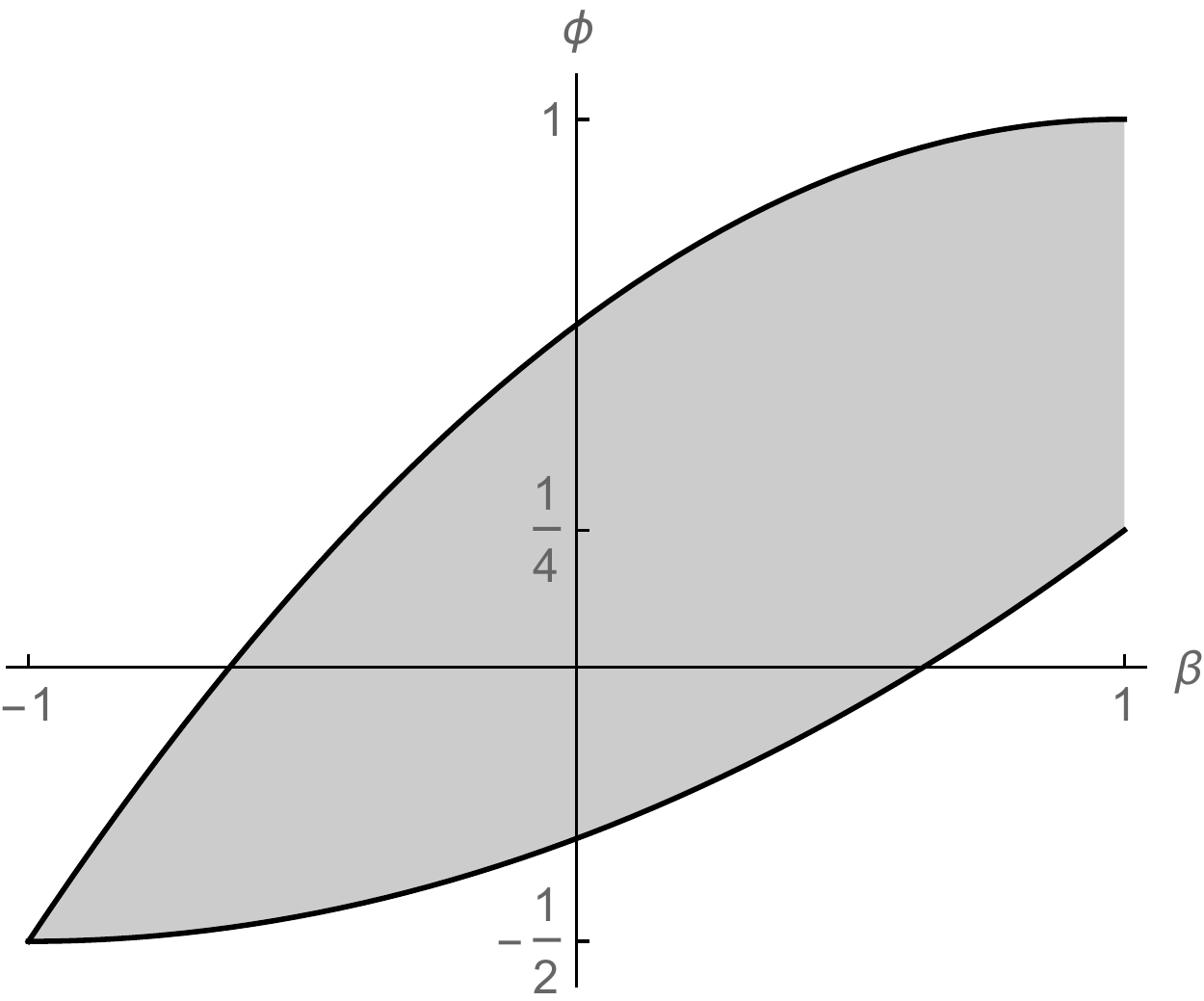}
             \caption{ Spearman's footrule vs.\ Blomqvist's beta
}\label{fig3}
\end{figure}

\begin{theorem}\label{thm:gamma} Let $C$ be any copula.
   \begin{enumerate}[(a)]
   \item If $\gamma(C)=\gamma$ for some $\gamma\in\left[-1, 1\right]$, then
   \[
     \left.\begin{array}{ll}
       -1, & \mbox{if } -1\leqslant\gamma\leqslant-\dfrac12 \\
       1-2\sqrt{\frac23(1-\gamma)}, & \mbox{if }
-\dfrac12\leqslant\gamma\leqslant1
     \end{array}\right\}
     \leqslant \beta(C)\leqslant
     \begin{cases}     -1+2\sqrt{\frac23(1+\gamma)}, & \mbox{if }
-1\leqslant\gamma\leqslant\dfrac12 \\  1, & \mbox{if }
\dfrac12\leqslant\gamma\leqslant1 .     \end{cases}
   \]
     \item If $\beta(C)=\beta$ for some $\beta\in\left[-1, 1\right]$, then
   \[
     \dfrac{3(1+\beta)^2}{8}-1\
     \leqslant \gamma(C)\leqslant
     1-\dfrac{3(1-\beta)^2}{8} .
   \]
   \end{enumerate}
     The bounds are attained.
\end{theorem}

\begin{proof}
Suppose that $\gamma(C)=\gamma$. Then $\underline{G}_{\gamma}(\frac12,
\frac12) \leqslant C(\frac12, \frac12)\leqslant
\overline{G}_{\gamma}(\frac12, \frac12)$.
For $\gamma \leqslant \frac12$ the point $(\frac12, \frac12)$ lies in
the area $\Omega_{\gamma}^5$ so we have $\overline{G}_{\gamma}(\frac12,
\frac12) = \omega_{\frac12, \frac12}^5(\gamma) =
\frac12\sqrt{\frac23(1+\gamma)}$. For $\gamma \geqslant \frac12$ we have
$\overline{G}_{\gamma}(\frac12, \frac12)= M(\frac12, \frac12) =
\frac12$. Next, $\underline{G}_{\gamma}(\frac12, \frac12) = \frac12 -
\overline{G}_{-\gamma}(\frac12, \frac12) =
\frac12\left(1-\sqrt{\frac23(1-\gamma)}\right)$ for $\gamma \geqslant
-\frac12$ and $\underline{G}_{\gamma}(\frac12, \frac12) = 0$ otherwise.
Now, $\beta(C) = 4C(\frac12, \frac12) - 1$ gives us point (a). The
bounds are attained by copulas $\overline{C}^{(\frac12, \frac12)}_{c_2}$
and $\underline{C}^{(\frac12, \frac12)}_{c_1}$, respectively.  Finally,
point (b) is obtained from (a) by inverting the functions.
\end{proof}

In Figure \ref{fig4} we display the set of all possible  pairs
$(\gamma(C), \beta(C))$ for a copula $C$. The expressions for the bounds
of the shaded regions are given in Theorem \ref{thm:gamma}.

\begin{figure}[h]
             \includegraphics[width=5cm]{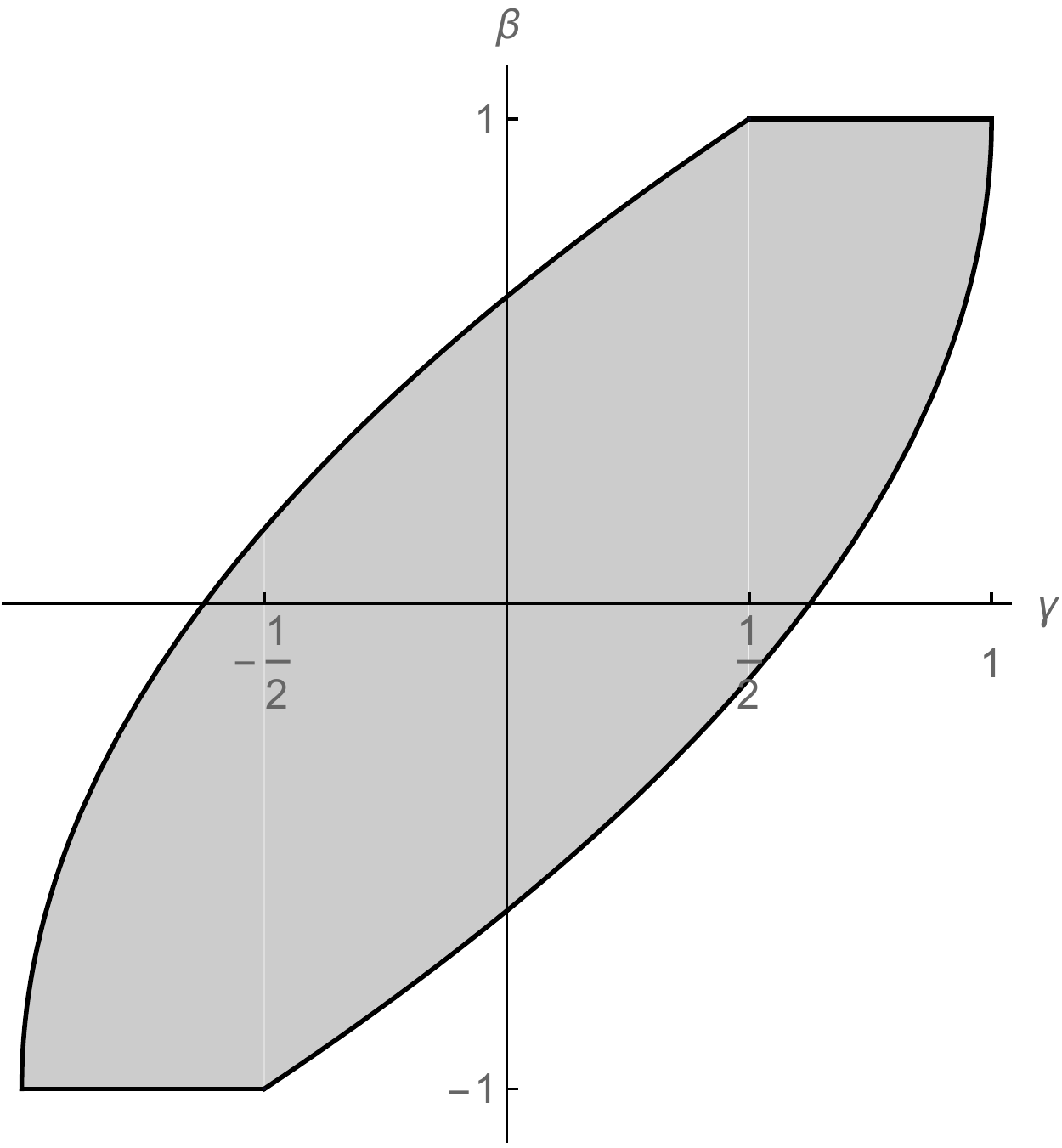} \hfil
\includegraphics[width=5cm]{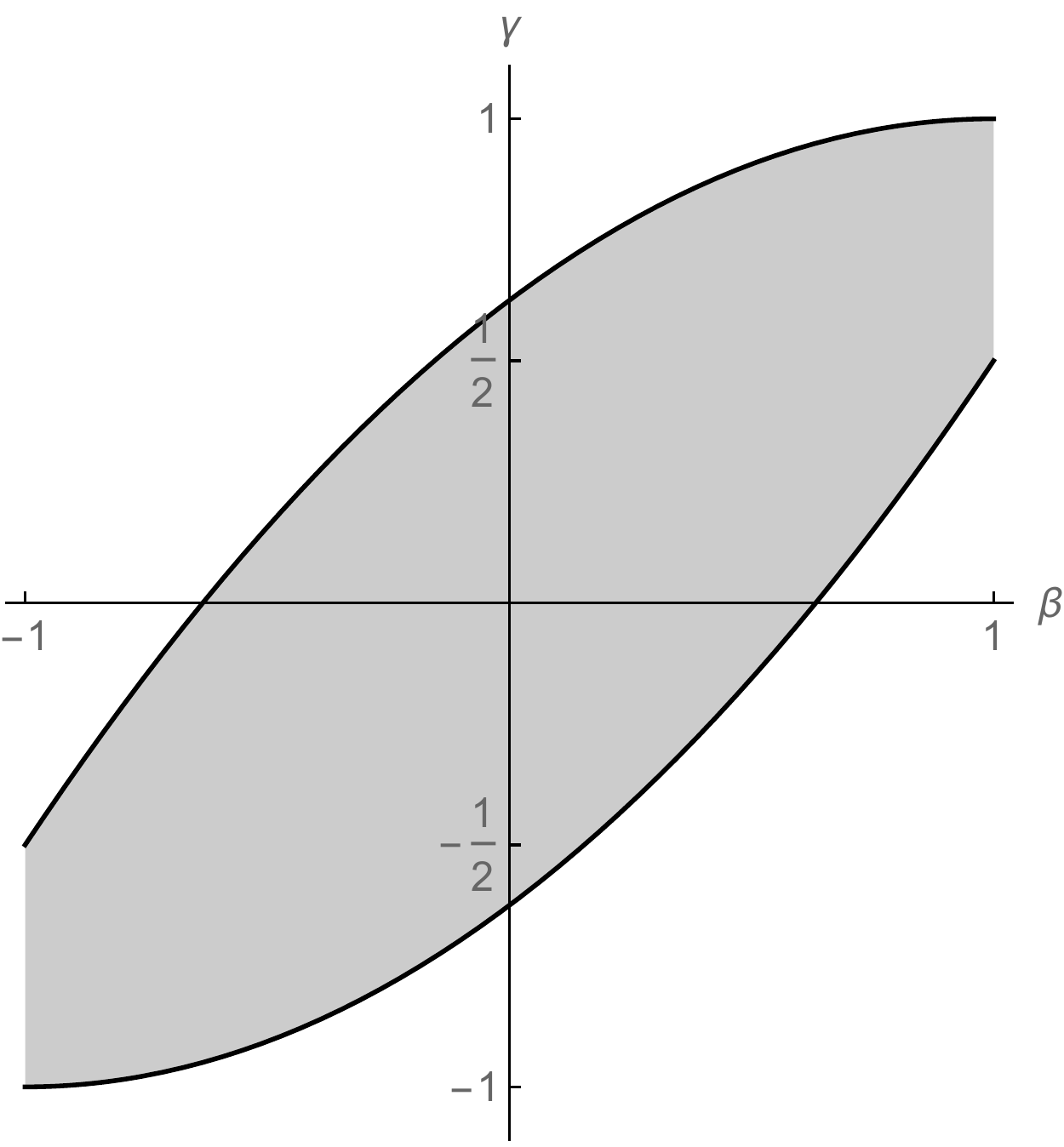}
             \caption{ Gini's gamma vs.\ Blomqvist's beta }\label{fig4}
\end{figure}

\noindent\textbf{Acknowledgement.} The authors are thankful to the referees. Their suggestions helped us to improve the paper. The figures in the paper were drawn using the Mathematica software \cite{Mathematica}.

\end{document}